\documentclass[12pt,twoside,english,reqno,a4paper]{amsart}

\usepackage[utf8]{inputenc}
\usepackage[english]{babel}
\usepackage[lined,boxed]{algorithm2e}
\usepackage{float} 
\usepackage{hyperref}
\usepackage{fullpage}
\usepackage{listings,graphicx,caption,mwe,amsmath,varioref,amscd,amssymb,color,xcolor,bm,amsthm,amsfonts,graphics,lineno,float,comment,soul}
\usepackage{mathrsfs} 
\usepackage[foot]{amsaddr}

\newcommand{\ignore}[1]{}

\theoremstyle{plain}
\newtheorem{definition}{Definition}[section]
\newtheorem{theorem}[definition]{Theorem}
\newtheorem{proposition}[definition]{Proposition}
\newtheorem{lemma}[definition]{Lemma}

\newtheorem{corollary}[definition]{Corollary}
\newtheorem{remark}[definition]{Remark}

\numberwithin{equation}{section}

\renewcommand{\theequation}{\thesection.\arabic{equation}}

\def\dis{\displaystyle}
\DeclareMathOperator*{\supp}{supp}
\def\R{\mathbb{R}}
\def\N{\mathbb{N}}
\def\RN{\mathbb{R}^{N}}
\def\norma#1#2{\|#1\|_{\lower 4pt \hbox{$ \scriptstyle #2$ }}}
\def\h{H^1(\mathbb{R}^N)}
\def\D{\mathcal{D}}
\def\A{{\mathcal A}}
\newcommand{\eps}{\varepsilon}
\newcommand{\e}{e}
\newcommand{\rr}{\bar r}
\newcommand{\emax}{e_{max}}
\newcommand{\emin}{e_{min}}

\author[R. Durastanti]{Riccardo Durastanti$^1$}

\author[L. Giacomelli]{Lorenzo Giacomelli$^{2,*}$}

\address{$^1$
Department of Mathematics and Applications ``Renato Caccioppoli'', University of Naples ``Federico II'', Via Cintia, Monte S. Angelo, 80126 Napoli, Italy
\\ riccardo.durastanti@unina.it}

\address{$^2$ SBAI Department, Sapienza University of Rome, Via Antonio Scarpa 16, 00161 Roma, Italy
\\ lorenzo.giacomelli@uniroma1.it}

\address{$^*$ Corresponding author}


\begin{document}

\title[Spreading equilibria: pancakes vs droplets]{
Spreading equilibria under mildly singular potentials: pancakes versus droplets
}

\begin{abstract}
We study global minimizers of a functional modeling the free energy of thin liquid layers over a solid substrate under the combined effect of surface, gravitational, and intermolecular potentials. When the latter ones have a mild repulsive singularity at short ranges, global minimizers are compactly supported and display a microscopic contact angle of $\pi/2$. Depending on the form of the potential, the macroscopic shape can either be droplet-like or pancake-like, with a transition profile between the two at zero spreading coefficient. These results generalize, complete, and give mathematical rigor to de Gennes' formal discussion of spreading equilibria. Uniqueness and non-uniqueness phenomena are also discussed.
\end{abstract}

\maketitle

\section{Introduction and results}

\subsection{The problem} We consider a class of singular energy functionals of the form
\begin{equation}
\label{def-E}
E[u]=\int_{\R^N} \left(\tfrac{1}{2}|\nabla u|^2 + Q(u)\right)dx,
\end{equation}
where the potential $Q(u)$ satisfies the following structural assumptions:
\begin{equation}
\label{H}
\left\{\begin{array}{l}
Q\in C((0,+\infty)), \ Q\equiv 0 \ \mbox{in}\ (-\infty,0], \ \inf Q>-\infty, \\
[1ex]
\dis Q(u)\sim A u^{1-m} \quad \text{as } u\to 0^+, \ m>1, \ A>0.
\end{array}\right.
\end{equation}
In view of \eqref{H}, \eqref{def-E} may be seen as a generalized and singular version of the Alt-Cafferelli or Alt-Phillips functionals \cite{AC,AP}. When modeling the height $u(x)$ of a liquid film over a solid substrate in lubrication approximation, $\gamma E$ represents the free energy of the liquid, with $\gamma$ the liquid-vapor surface tension. In this case, the potential $Q$ usually combines the effects of intermolecular, gravitational, and surface forces:
$$
Q(u)=(P(u)+G(u)-S)\chi_{\{u>0\}}
$$
\cite{DG,ODB}. The function $P$ is a inter-molecular potential, singular at $u=0$ and decaying at $u=+\infty$; the function $G$ is a gravitational potential; $S$ is the non-dimensionalized spreading coefficient:
\begin{equation*}
S=\frac{\mbox{spreading coefficient}}{\gamma}= \frac{\gamma_{SG}-\gamma_{SL}-\gamma}{\gamma},
\end{equation*}
where $\gamma_{SG}$ and $\gamma_{SL}$ are the solid-gas and solid-liquid tensions, respectively. There is, however, a caveat to be made at this point.

In thermodynamic equilibrium of the solid with the vapor phase (the so-called ``moist'' case, which concerns for instance a surface which has been pre-exposed to vapor), $\gamma_{SG}$ is usually denoted by $\gamma_{SV}$, and its value can never exceed $\gamma_{SL}+\gamma$. Indeed, otherwise the free energy of a solid/vapor interface could be lowered by inserting a liquid film in between: the equilibrium solid/vapor interface would then comprise such film, leading to $\gamma_{SV}=\gamma_{SL}+\gamma$. Therefore, $S\le 0$ in the ``moist'' case. On the other hand, when the solid and the gaseous phase are not in equilibrium (the so-called ``dry'' case),  there is no constraint on the sign of $S$.

The cases $S<0$, resp. $S\ge 0$, are commonly referred to as {\it partial wetting}, resp. {\it complete wetting}: indeed, when $Q\equiv -S$, the global minimizer's support is compact if $-S>0$, whereas if $-S\le 0$ the global minimizer does not exist and the final spreading equilibrium is a zero-thickness unbounded film (see e.g. \cite[\S 19.4]{Maggi}, where the complete form of the surface energy is considered instead of its lubrication approximation).

\medskip

We are interested in non-negative {\em global} minimizers (hereafter simply called {\it minimizers}) of $E$ under the constraint of fixed mass; that is, in the set
\begin{equation}
\label{def-D}
\D= \D_M= \left\{u\in H^1(\RN) : u\geq 0, \int u=M \right\}
\end{equation}
(we shall omit the subscript $M$ when unnecessary).

\subsection{The potential $Q$}
When gravity is not taken into account, $Q$ is characterized by
\begin{equation}\label{modelQ0}
Q(u)= \left\{\begin{array}{ll}
Au^{1-m}(1+o(1)) & \mbox{as $u\to 0^+$}
\\[1ex]
-S- Bu^{1-n}(1+o(1)) & \mbox{as $u\to +\infty$}
\end{array}\right.
\end{equation}
with $A>0$, $B\in \R$, and $m,n>1$. Since $A>0$, the singularity of $Q$ at $u=0$ disfavors small heights of the droplet and corresponds to short-range repulsive forces. When the strength of the singularity is sufficiently high, namely when $m\ge 3$, the very existence of a minimizer is precluded, since $E[u]\equiv +\infty$ for any $u\in \mathcal D$ (see Lemma \ref{m>3} below). However, this is not the case when the singularity is milder ($m<3$), which is the focus of this manuscript.

\smallskip

At long ranges, $B<0$ corresponds to considering the effect of repulsive forces only (cf. the discussion in \cite[II.D.1]{DG} and references therein), whereas $B>0$ corresponds to considering short-range repulsive, long-range attractive, forces (cf. \cite[II.E]{ODB} and references therein). We anticipate that the long-range decay exponent $n$ is not essential: it enters the analysis only in critical cases.

\smallskip

Though our results cover a wide range of potentials, it will be convenient to introduce a few prototypical cases (Fig. \ref{fig-potentials}.A). The first one, $Q_a$, is repulsive-attractive for $B>0$ and purely repulsive for $B\le 0$:
\begin{equation}
\label{modelQ1a}
\begin{array}{c}
Q_{a}(u) = Au^{1-m}- Bu^{1-n}-S \quad \mbox{for $u>0$},
\\[1ex]
\quad B\in \R, \quad  1<n<m.
\end{array}
\end{equation}
For purely repulsive potentials ($B<0$), a long-range decay exponent $n$ larger than the short-range growth exponent $m$ is often considered. A prototype which is suited to this situation is
\begin{subequations}\label{modelQ1b}
\begin{equation}
Q_{b}(u) = \frac{A|B| u}{|B|u^m+Au^{n}}-S \quad \mbox{for $u>0$},
\quad B< 0, \quad  m<n,
\end{equation}
for which we only consider the convex case, corresponding to the constraint
\begin{equation}\label{strange}
1 + 2 m + m^2 + 2 n - 6 m n + n^2 \le 0.
\end{equation}
\end{subequations}
Finally, when gravity is taken into account, the potential $G$ has to be added:
\begin{equation}
\label{modelQg}
Q_{a,g}(u)= Q_a(u)+\tfrac12 Du^2, \quad Q_{b,g}(u)= Q_b(u)+\tfrac12 Du^2.
\end{equation}

\captionsetup{width=0.9\linewidth}
{\centering\raisebox{\dimexpr \topskip-\height}{
  \includegraphics[width=0.45\textwidth]{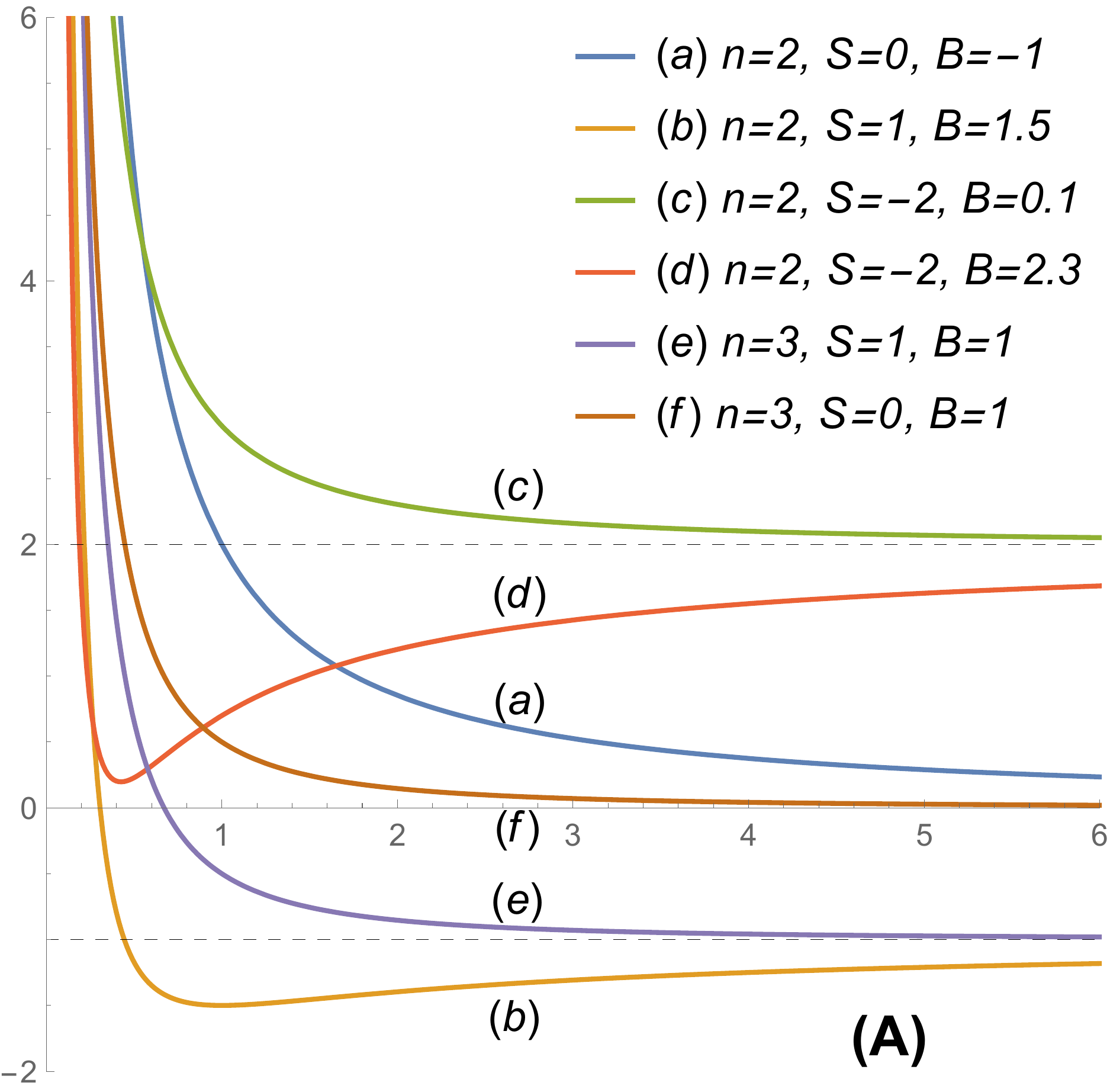} \qquad \includegraphics[width=0.45\textwidth]{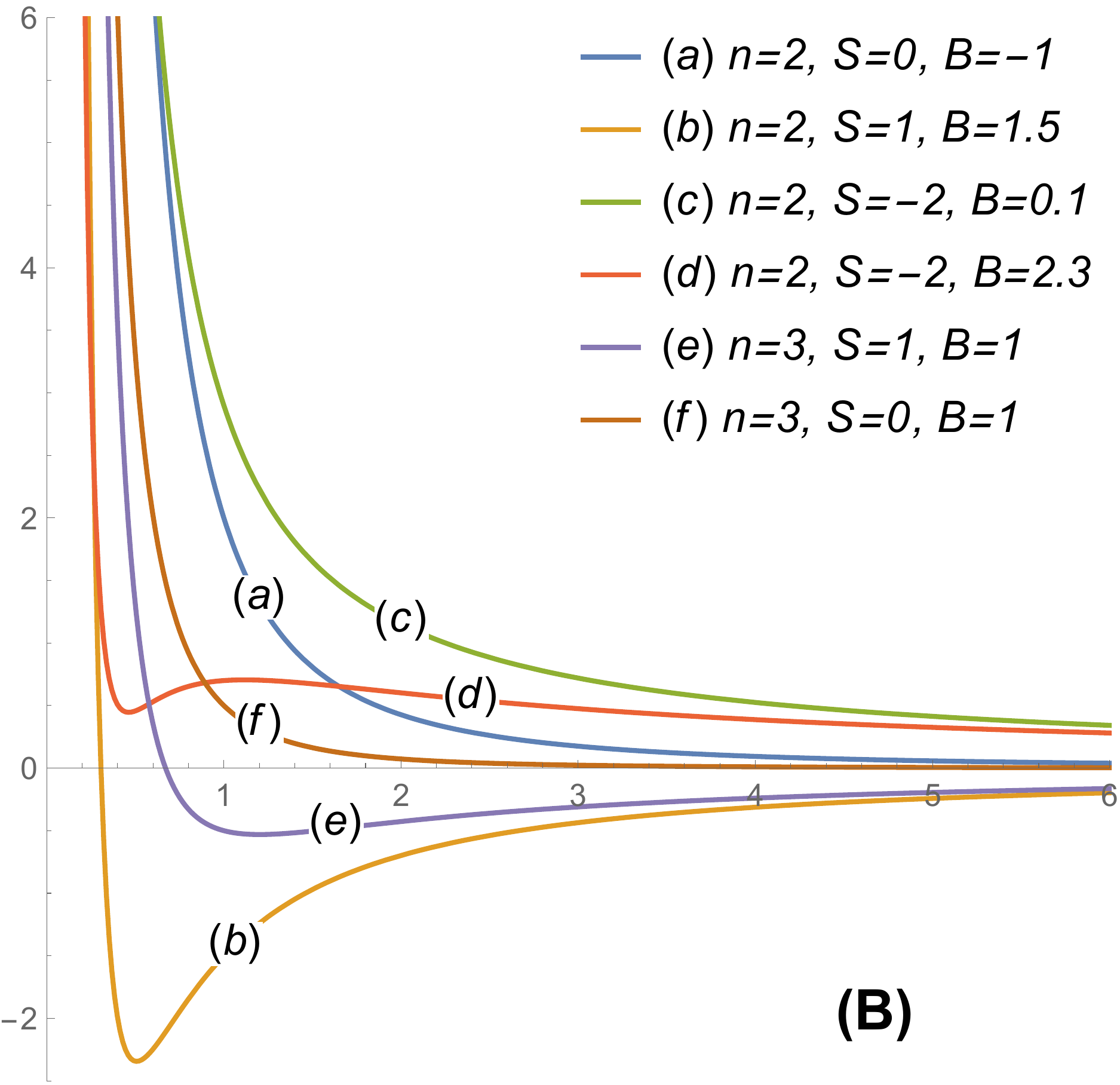} }
  \captionof{figure}{{\footnotesize {\bf (A)} On the left, prototypes of $Q$ with $A=1$ and $m=5/2$: from (a) to (d), $Q_a$ with $n=2$; in (e) and (f), $Q_b$ with $n=3$. {\bf (B)} On the right, the corresponding functions $R(u)=Q(u)/u$: $R'$ has no zeros in (a), (c), and (f), one zero in (b) and (e), two zeroes in (d); $\e_*<+\infty$ in (b) and (e).
  }}
  \label{fig-potentials}
}
\subsection{The framework}\label{ss-frame}

An enormous amount of work has been done on the fundamentals of wetting phenomena, from different perspectives: referenced discussions may be found in \cite{DG,Finn,ODB,Bonn,Maggi,A}. Concerning the analysis of energy functionals $E$ of the form \eqref{def-E} {\em with a singular potential $Q$}, the focus has mainly been on two aspects.

\begin{itemize}
\item {\em Positive minimizers with $Q\equiv +\infty$ for $u\le 0$ and/or $m\ge 3$}. In this case, short range repulsion is so strong that compactly supported minimizers do not exist, and energy minimization forces the creation of a tiny liquid layer fully separating gas and solid.
In this framework, interesting qualitative properties of minimizers, such as (in)stability of the flat film, bifurcation, concentration, and asymptotic scaling laws with respect to the potential's parameters, have been successfully investigated, also in relation to dynamic phenomena such as coarsening and dewetting; see \cite{Getal,BGW,GW,CJ,CJL,Ji,LP1,LP2,LP3,LP4,LW,ORS,GORS}, the references therein, and \cite{W} for a recent overview.

\item {\em Potentials with $A<0$.} In this case, minimizers and critical points also have a rich structure: we refer to \cite{JN} and again to \cite{LP1,LP2,LP3,LP4} for a thorough study, including classification, stability, and other qualitative properties.
\end{itemize}

On the other hand, in the case of mildly singular potentials,
\begin{equation}\label{main-m}
Q(u)\equiv 0 \ \mbox{ for } \ u\le 0 \quad\mbox{and} \quad 1<m<3,
\end{equation}
the minimization problem \eqref{def-E}-\eqref{def-D} does not seem to have been explored so far. We are only aware of two very recent and interesting works \cite{DeSa1,DeSa2}, where the model case $Q(u)=u^{1-m}\chi_{\{u>0\}}$ is considered on a bounded domain $\Omega$ with Dirichlet boundary condition (and no mass constraint). There, existence and regularity of minimizers is discussed, together with the regularity of the free boundary $\partial\{u>0\}\cap \Omega$ and the $\Gamma$-limit as $m\to 3^-$.

\smallskip

The case \eqref{main-m} is the focus of the present manuscript. Given the vastity of the potentials which have been introduced and considered through the years, we prefer to study generic potentials rather than concentrating on model cases.

\subsection{Existence and basic properties of minimizers}

Solely under \eqref{H} and \eqref{main-m}, the existence of a minimizer of $E$ in $\D$ is guaranteed by standard direct methods and symmetry arguments (see Theorem \ref{exist}). The assumption $m<3$ is crucial, since $E[u]\equiv +\infty$ on $\D$ if $m\ge 3$ (see Lemma \ref{m>3}).  It turns out that the minimizer we obtain is:

\begin{itemize}
\item[(a)] compactly supported;
\item[(b)] radially symmetric (up to a translation of $x$);
\item[(c-)] non-increasing along radii.
\end{itemize}

In the rest of this introduction we assume in addition that $Q\in C^1((0,+\infty))$. If either $N=1$, or if $Q'(u)$ satisfies a very mild additional condition for $u\gg 1$ (see \eqref{H2} below), then (a), (b), and (c-) in fact hold for {\em any} minimizer; in addition, {\em any} minimizer is (see Theorem \ref{alldec}):
\begin{itemize}
\item[(c)] strictly decreasing along radii;
\item[(d)] a smooth solution to the Euler-Lagrange equation for some $\lambda\in \R$:
\begin{equation}
\label{EL-intro}
-\Delta u+Q'(u)=\lambda \quad\mbox{in $\{u>0\}$}.
\end{equation}
\end{itemize}

\subsection{The one-dimensional case}

For $N=1$ we are able to obtain much more detailed information, such as uniqueness and asymptotic results, which are discussed in the next paragraphs. The key to both of them is the identification of $\lambda$, which we prove via a combination of ODE and variational arguments (Theorem \ref{th-lambda}):
\begin{equation}\label{def-R-intro}
\lambda = R(u(0)), \quad\mbox{where}\quad  R(u):=\frac{Q(u)}{u}.
\end{equation}
Not surprisingly, the function $R$ plays a crucial role in the analysis. First of all,
it follows from \eqref{def-R-intro} that a constant function $u_s\in (0,+\infty)$ is a stationary solution to \eqref{EL-intro} if and only if $Q'(u_s)=\lambda= R(u_s)=Q(u_s)/u_s$; since $u^2 R'(u)= uQ'(u)-Q(u)$, in fact
$$
\mbox{$u_s$ is a stationary solution to \eqref{EL-intro} if and only if $R'(u_s)=0$.}
$$
We assume, as in the model cases \eqref{modelQ1a}, \eqref{modelQ1b}, and \eqref{modelQg}, that these stationary solutions do not accumulate at $0$ or $+\infty$:
\begin{equation*}
\mbox{$\delta\in (0,1)$ exists such that \  $R'\ne 0$ \ in \, $(0,\delta)\cup (\delta^{-1},+\infty)$.}
\end{equation*}
Of crucial importance is the smallest among the absolute minimum points of $R$, provided they exist:
\begin{equation}
\label{def-s*-intro}
\e_*:=\left\{\begin{array}{ll}
+\infty & \mbox{if $\not\exists \, \min R$}
\\[1ex]
\min R^{-1}(\min R) & \mbox{otherwise}.
\end{array}\right.
\end{equation}
In the model cases, $\e_*$ coincides with the unique global minimum point of $R$, whenever such point exists (Fig. \ref{fig-potentials}.B).

\subsection{Uniqueness}\label{ss-intro-uniq}

As is often the case, uniqueness is related to convexity. If $Q$ is convex in $(0,\e_*)$, by comparison arguments we show that the minimizer is unique (see Theorem \ref{uniq1d}). In terms of the model cases, this translates into (see Section \ref{s:mod}):
\begin{itemize}
\item uniqueness for $Q_{a}$ if $B\le 0$, or if $B>0$ and $-S\le 0$, or if $-S\ge 0$ and $B\ge c_1(A,S)$, where
$$
c_1(A,S):= \textstyle (m-1)\left(\frac{A}{n-1}\right)^{\frac{n-1}{m-1}} \left(\frac{-S}{m-n}\right)^{\frac{m-n}{m-1}};
$$
\item uniqueness for $Q_{a,g}$ if $-S\le 0$ or if $B\le c_3(A,D)$, where
$$\textstyle
c_3(A,D):= \frac{m+1}{n(n-1)}\left(\frac{Am(m-1)}{n+1}\right)^{\frac{n+1}{m+1}}\left(\frac{D}{m-n}\right)^{\frac{m-n}{m+1}};
$$
\item uniqueness for $Q_{b}$ and $Q_{b,g}$.
\end{itemize}
Interestingly, however, potentials $Q$ exist such that the minimizer is {\em not} unique for at least one value of the mass $M$. Generally speaking, this occurs when $R$ is not injective in $(0,\e_*)$ (see Theorem \ref{no-un}): this is the case, for instance, in model $Q_a$ with $-S>0$ and $c_2(A,S) \leq B < c_1(A,S)$, where (see Section \ref{s:mod})
$$
c_2(A,S):= \tfrac{m-1}{n}\left(\tfrac{Am}{n-1}\right)^{\frac{n-1}{m-1}} \left(\tfrac{-S}{m-n}\right)^{\frac{m-n}{m-1}}.
$$
\captionsetup{width=0.9\linewidth}
{\centering\raisebox{\dimexpr \topskip-\height}{\includegraphics[width=0.9\textwidth]{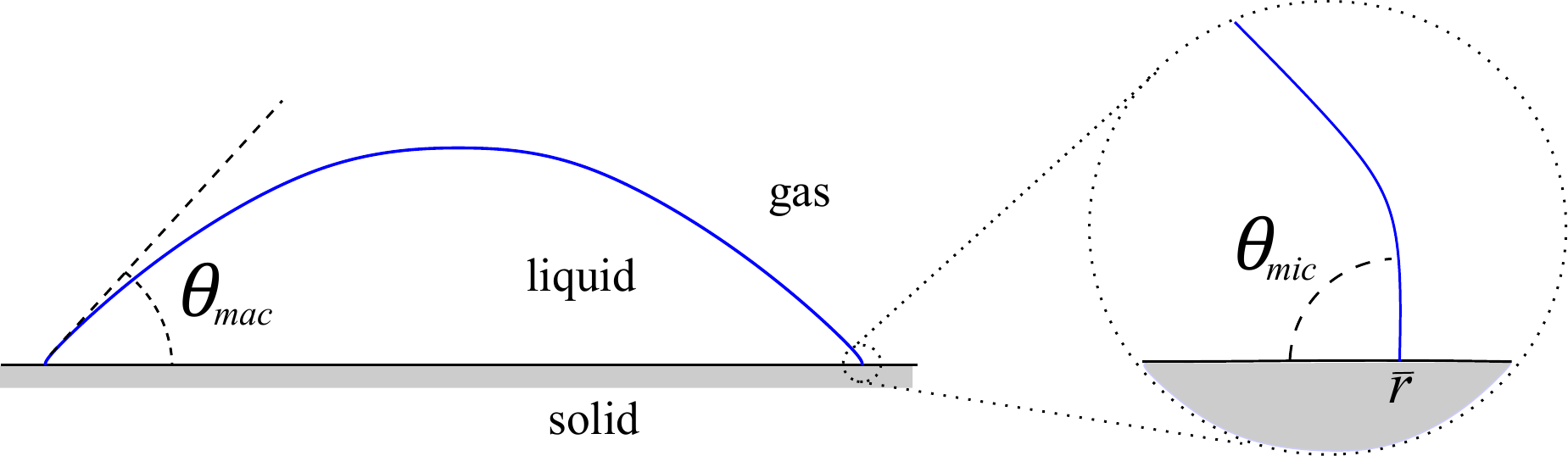}}
  \captionof{figure}{{\footnotesize A prototypical droplet ($N=1$). On the left, the macroscopic profile and the macroscopic contact angle $\theta_{mac}$; on the right, a zoom into the contact-line $\rr$ and the microscopic contact angle $\theta_{mic}$. All minimizers of $E$ in $\mathcal D$ have $\theta_{mic}=\pi/2$.
  }}
  \label{fig-droplet}
}

\subsection{Micro-macro relations and the regime $M\gg 1$}

When continuum models are considered, wetting phenomena are characterized by the presence of two interfaces of codimension-one (liquid-solid and liquid-gas) and an (unknown) {\em contact line}, i.e. a codimension-two interface where the solid, the liquid, and the surrounding gas or vapour meet (Fig. \ref{fig-droplet}). Among the main topics of interest to the physics and applied math communities, are the modelling of the ``microscopic'' physics near these interfaces --e.g. in terms of intermolecular potentials, substrate's corrugation and, in a dynamic context, slippage, contact-line free boundary conditions, and rheological properties-- and the analysis of how such microscopic laws affect the ``macroscopic'' behaviour of droplets. See the reviews \cite{DG,ODB,Bonn,A} and also \cite{FK,FliK,RHE} for referenced discussions.

\smallskip

In this context, of particular importance are the {\em microscopic contact angle} $\theta_{mic}$, identified with the arctangent of the droplet's slope {\em at} the contact line, and (various notions of) {\em macroscopic, or effective, or apparent contact angle} $\theta_{mac}$: generally speaking, this is the arctangent of the slope, near the contact line, of the profile that the droplet assumes in the bulk of the wetted region, see Fig. \ref{fig-droplet}.

\smallskip

For droplet's dynamics, after the pioneering works \cite{HS,DD,Voinov1976,Tanner}, the relation of $\theta_{mac}$ and macroscopic profile with $\theta_{mic}$, microscopic modelling, and speed of the contact line has been extensively studied via both formal asymptotic methods (see e.g. \cite{Cox,Hoc1,Hoc2,HM,BDDG,ES,AG,CG} and the references therein) and rigorous arguments \cite{GO3,GGO,DM}, especially in the case $\theta_{mic}=0$. More details may be found e.g. in \cite{ES} and \cite[\S C]{Bonn}.

\smallskip

In the framework of this paper, which is concerned with statics, the ``microscopic'' physics are encoded in the intermolecular part $P$ of the potential $Q$. In order to associate to $P$ a microscopic length-scale $\eps$, for a given reference potential $P_0$ one could set
\begin{equation}\label{case-eps}
P(u) = P_0\left(\tfrac{u}{\eps}\right)\quad \mbox{with } \ \eps\ll 1.
\end{equation}
Then the macroscopic profile of minimizers could be identified by taking the limit as $\eps\to 0$. However, due to the lack of scaling invariance of $E$ for general $Q$, it is more convenient to look at the limit as $M\to +\infty$. The two regimes are equivalent when $E$ has a scaling invariance. This is the case, for instance, when $P$ has the form \eqref{case-eps} and $G=0$ (no gravity): indeed, with the scaling $x=\eps\hat x$, $u=\eps \hat u$, one obtains
$$
E[u]= \int_{\R} \left(\tfrac12 u_x^2+P_0\left(\tfrac{u}{\eps}\right) -S \right)d x = \eps \int_{\R} \left( \tfrac12\hat u_{\hat x}^2+P_0(\hat u) -S\right)d \hat x
$$
with
$$
\int_{\R} \hat u d\hat x = \eps^{-2} \int_{\R} u d x \gg 1.
$$
Hence, we will analyze the limit $M\to +\infty$: the goal is to identify a macroscopic profile, whence a macroscopic contact angle (if it exists), as the limit of (suitably rescaled) minimizers $u_M$ of $E$ in $\mathcal D_M$.

\subsection{Microscopic behavior}\label{ss:micro}
The microscopic behavior of minimizers of $E$ in $\mathcal D$ is universally determined by the short-range form of the potential. Indeed, we show in Theorem \ref{asbe1} that
$$
u(x) \sim \left(\tfrac{A(m+1)^2}{2}\right)^{\frac{1}{m+1}}(\rr- x)^{\frac{2}{m+1}} \quad \text{as } x \to \rr^-\,,
$$
where $\rr$ denotes the right boundary of the minimizer's support. This shows that mildly singular potentials produce steady states with $\theta_{mic}=\pi/2$ (Fig. \ref{fig-droplet}).

\subsection{Macroscopic behavior: Pancakes versus droplets}\label{ss:macro}

Let $u_M$ ba a minimizer of $E$ in $\mathcal D_M$. By translation invariance, we may assume that $\supp u_M=[-\rr_M,\rr_M]$ and that the maximal height is $u_{0M}=u_M(0)$. The behavior of $u_M$ for $M\gg 1$ is essentially influenced by two quantities: the constant $\e_*$ defined in \eqref{def-s*-intro}, which is always finite in presence of gravity (i.e. $D>0$), and the non-dimensionalized spreading coefficient $S$, which for a generic potential $Q$ is defined by
$$
(-\infty,+\infty)\ni -S := \lim_{u\to +\infty} Q(u) \quad\mbox{when the limit exists and is finite.}
$$

We will prove in Section \ref{s4} that there are two generic behavior of $u_M$ as $M\to +\infty$.

\medskip
\captionsetup{width=\linewidth}
\centerline{\includegraphics[width=0.8\textwidth]{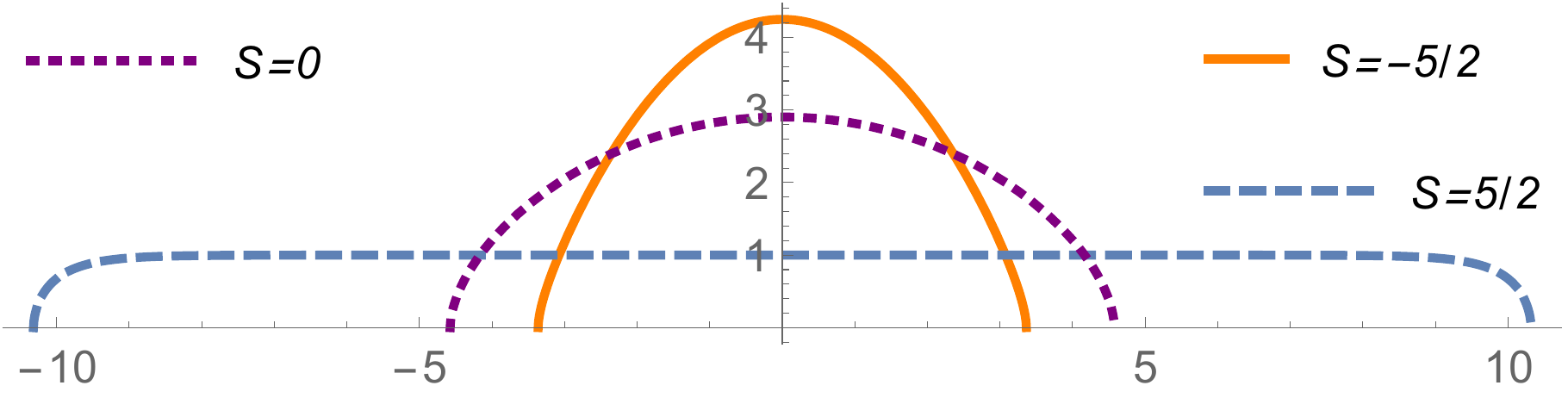}}
\captionof{figure}{{\footnotesize In the case $Q_a$ with $A=1$, $B=0$, $m=5/2$, and $M=20$, the unique minimizer of $E$ for different values of $S$. Minimizers are obtained as numerical solutions of \eqref{EL-intro}-\eqref{def-R-intro} with $u'(0)=0$, tuning $u(0)$ so that $M=20$ (with a tolerance of $10^{-2}$).}}
\label{fig-ode}
\medskip

\noindent $\bullet$  {\bf Droplet:} as $M\to +\infty$,
\begin{equation}\label{par-intro}
\begin{array}{c}
\displaystyle u_{0M}^4 \sim \frac{9|S|}{32}M^2,\quad \rr_M^4\sim \frac{9}{8|S|} M^2, \quad
u_{0M}^{-1} u_M(\rr_My) \sim 1-y^2.
\end{array}
\end{equation}

\noindent $\bullet$  {\bf Pancake:} as $M\to +\infty$,
\begin{equation}
\label{pan-intro}
\begin{array}{c}
\displaystyle u_{0M}\sim \e_*, \quad \rr_M \sim \frac{1}{2\e_*}M, \quad
u_M(\rr_M y) \sim \e_*.
\end{array}
\end{equation}

\noindent If $S=0$ and $Q$ is strictly decreasing, we also find a third, intermediate behavior, which connects droplets to pancakes (cf. Remark \ref{rem:f1}) through the decay exponent of $Q(u)$ as $u\to +\infty$.

\smallskip

\noindent $\bullet$  {\bf Transition profile}: with $S=0$, $Q'<0$, and $Q(u)\sim Ku^{1-p}$ as $u\to +\infty$,
\begin{equation}\label{crit-intro}
u_{0M}^{p+3} \sim \frac{pK}{2 c_p^2f_p(0)^2} M^2, \quad
\rr_M^{p+3} \sim \frac{f_p(0)^2}{2^{p+2} p K c_p^{p+1}} M^{p+1},\quad u_{0M}^{-1} u_M(\rr_My) \sim f_p^{-1}(f_p(0)|y|),
\end{equation}
where
\begin{equation*}
f_p(w):=\int_{w}^1 \frac{\sqrt{p}\tilde w^{\frac{p-1}{2}}d\tilde w}{\sqrt{1-\tilde w^p}}\quad\mbox{and}\quad c_p:=\int_0^1 f_p^{-1}(f_p(0)y)dy.
\end{equation*}

In Fig. \ref{fig-ode} we report numerical solutions to the minimization problem in a prototypical case in which uniqueness holds. Table 1 summarizes the main assumptions which lead to each of these behaviors, together with the corresponding model cases and with references to the corresponding results.

\begin{table}[H]
\centering
\begin{tiny}
\begin{tabular}{|p{0.18\textwidth}|p{0.12\textwidth}|p{0.15\textwidth}|p{0.26\textwidth}|p{0.18\textwidth}|}
\hline
\begin{center}{\bf Main assumptions}\end{center}  & \begin{center}{\bf Asymptotics as $M\to+\infty$} \\[2ex] (Theorem \ref{thm:finale}, \\ Theorem \ref{thm:crit}) \end{center} & \begin{center}{\bf Uniqueness} \\ $\ $ \\[2ex] (Theorem \ref{uniq1d},\\ Theorem \ref{no-un}) \end{center} & \begin{center}{\bf Model cases \\ (cf. \eqref{modelQ1a}, \eqref{modelQ1b}, \eqref{modelQg})} \\[2ex] (Section \ref{s:mod}) \end{center} &  \begin{center}{\bf Uniqueness in \\ model cases}\\[2ex] (Section \ref{s:mod}) \end{center} \\
\hline
\begin{tabular}[t]{l}
$\ $
\\
$\bullet$ $\e_*<+\infty$
\end{tabular}
& \begin{center} Pancake \end{center} &  \begin{center} Yes if $Q''\geq 0$ in $(0,\e_*)$ \end{center} &
\begin{tabular}[t]{l}
$\ $
\\
$\bullet$ $Q_{a}$, $-S<0$
\\[1ex]
$\bullet$ $Q_{a}$, $-S=0$, $B>0$
\\[1ex]
$\bullet$ $Q_{a}$, $-S>0$, $B\ge c_1(A,S)$
\\[1ex]
$\bullet$ $Q_{b}$, $-S<0$
\\[1ex]
$\bullet$ $Q_{b,g}$
\\[1ex]
$\bullet$ $Q_{a,g}$
\\[1ex]
\end{tabular}
&
\begin{tabular}[t]{l}
$\ $
\\
$\bullet$ Yes
\\[1ex]
$\bullet$ Yes
\\[1ex]
$\bullet$ Yes
\\[1ex]
$\bullet$ Yes
\\[1ex]
$\bullet$ Yes
\\[1ex]
$\bullet$ Yes if $-S\le 0$
\\ \phantom{$\bullet$} or $B\le c_3(A,D)$
\\[1ex] $\ $
\end{tabular}
\\

\hline

\begin{tabular}[t]{l}
$\ $
\\
$\bullet$ $\e_*=+\infty$
\\[1ex]
$\bullet$ $-S > 0$
\end{tabular}
&
\begin{center} Droplet \end{center} & \begin{center} Yes if $Q''\geq 0$ \end{center}
&
\begin{tabular}[t]{l}
$\ $
\\
$\bullet$ $Q_{a}$, $-S>0$, \\
\quad  $c_2(A,S)\leq B <c_1(A,S)$
\\[1ex]
$\bullet$ $Q_{a}$, $-S> 0$,
\\
\quad $0< B< c_2(A,S)$
\\[1ex]
$\bullet$ $Q_{a}$, $-S> 0$, $B\le 0$
\\[1ex]
$\bullet$ $Q_{b}$, $-S> 0$
\\[1ex] $\ $
\end{tabular}
&
\begin{tabular}[t]{l}
$\ $
\\
$\bullet$ No \\
$\ $
\\[1ex]
$\bullet$ Not known \\
$\ $
\\[1ex]
$\bullet$ Yes
\\[1ex]
$\bullet$ Yes
\\[1ex] $\ $
\end{tabular}
\\
\hline\hline

\begin{tabular}[t]{l}
$\ $
\\
$\bullet$ $\e_*=+\infty$
\\[1ex]
$\bullet$ $-S=0$
\\[1ex]
$\bullet$ $Q'<0$ in $(0,+\infty)$
\\[1ex] $\ $
\end{tabular}
& \begin{center} Transition profile \end{center} & \begin{center} Yes if $Q''\geq 0$\end{center}
&
\begin{tabular}[t]{l}
$\ $
\\
$\bullet$ $Q_{a}$, $-S=0$, $B\le 0$
\\[1ex]
$\bullet$ $Q_{b}$, $-S=0$
\end{tabular}
&
\begin{tabular}[t]{l}
$\ $
\\
$\bullet$ Yes
\\[1ex]
$\bullet$ Yes
\end{tabular}
\\
\hline

\end{tabular}
\end{tiny}
\caption{Synopsis of the main results.}\label{tabella}
\end{table}

\subsection{Profiles of minimizers: macroscopic contact angles and thickness}\label{ss:combi}

Combining the information in Paragraphs \ref{ss:micro} and \ref{ss:macro}, we can characterize minimizers as follows.

\smallskip

\noindent $\bullet$  {\bf Droplet:} we have
\begin{equation}\label{dfg}
u(x)\sim \left\{
\begin{array}{lll}
\frac{\sqrt{|S|}}{\rr\sqrt 2}(\rr^2-x^2) &  \delta\lesssim \rr- x\le \rr & \mbox{(macroscopic profile),}
\\
\left(\tfrac{A(m+1)^2}{2}\right)^{\frac{1}{m+1}}(\rr- x)^{\frac{2}{m+1}} & 0\le \rr- x\lesssim \delta & \mbox{(microscopic profile),}
\end{array}\right. \quad
\end{equation}
(see Fig. \ref{fig-4}), where
$$
\rr^4 \sim \tfrac{9 M^2}{8|S|}, \quad \delta^{m-1} \sim (2|S|)^{-\frac{m+1}{2}} \tfrac{A(m+1)^2}{2}.
$$
In this case, it is natural to define the {\em macroscopic contact angle} $\theta_{mac}$ as the arctangent of the slope of the macroscopic profile at the boundary of its support:
\begin{equation}\label{our-t}
\tan\theta_{mac}= \frac{\sqrt{|S|}}{\rr\sqrt 2}\frac{d}{d x}(\rr^2-x^2)\big|_{x=-\rr}=\sqrt{2|S|}.
\end{equation}
This analysis also identifies the {\em transitional thickness} as the height $\sqrt{2|S|}\delta$ at which the cross-over takes place:
\begin{equation}\label{our-thick}
\sqrt{2|S|}\delta \sim \left(\tfrac{A(m+1)^2}{4|S|}\right)^\frac{1}{m-1}.
\end{equation}

\captionsetup{width=0.9\linewidth}
{\centering{
\raisebox{\dimexpr \topskip-\height}{
  \includegraphics[width=0.45\textwidth]{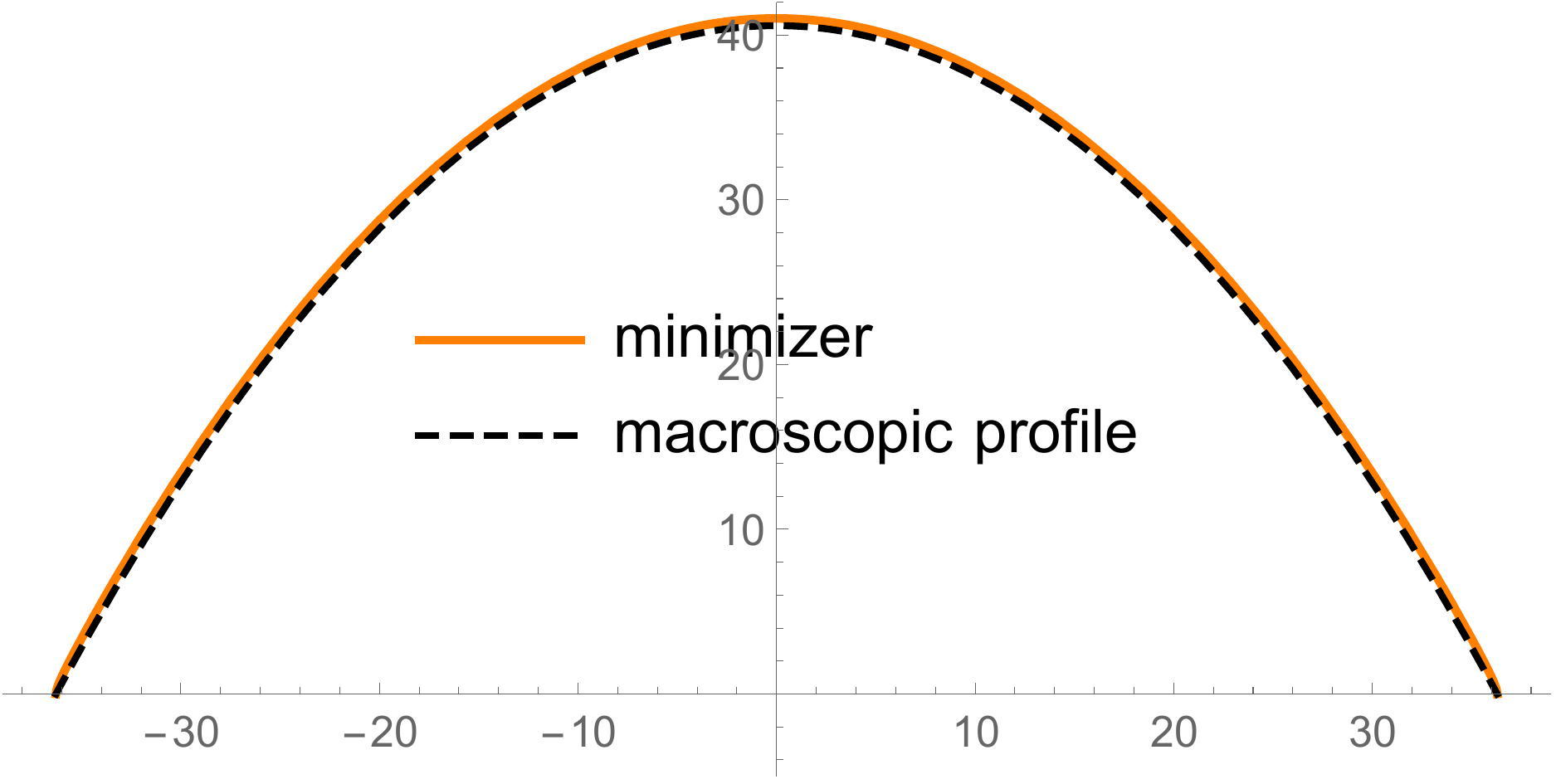} \quad \hfill\quad
  \includegraphics[width=0.45\textwidth]{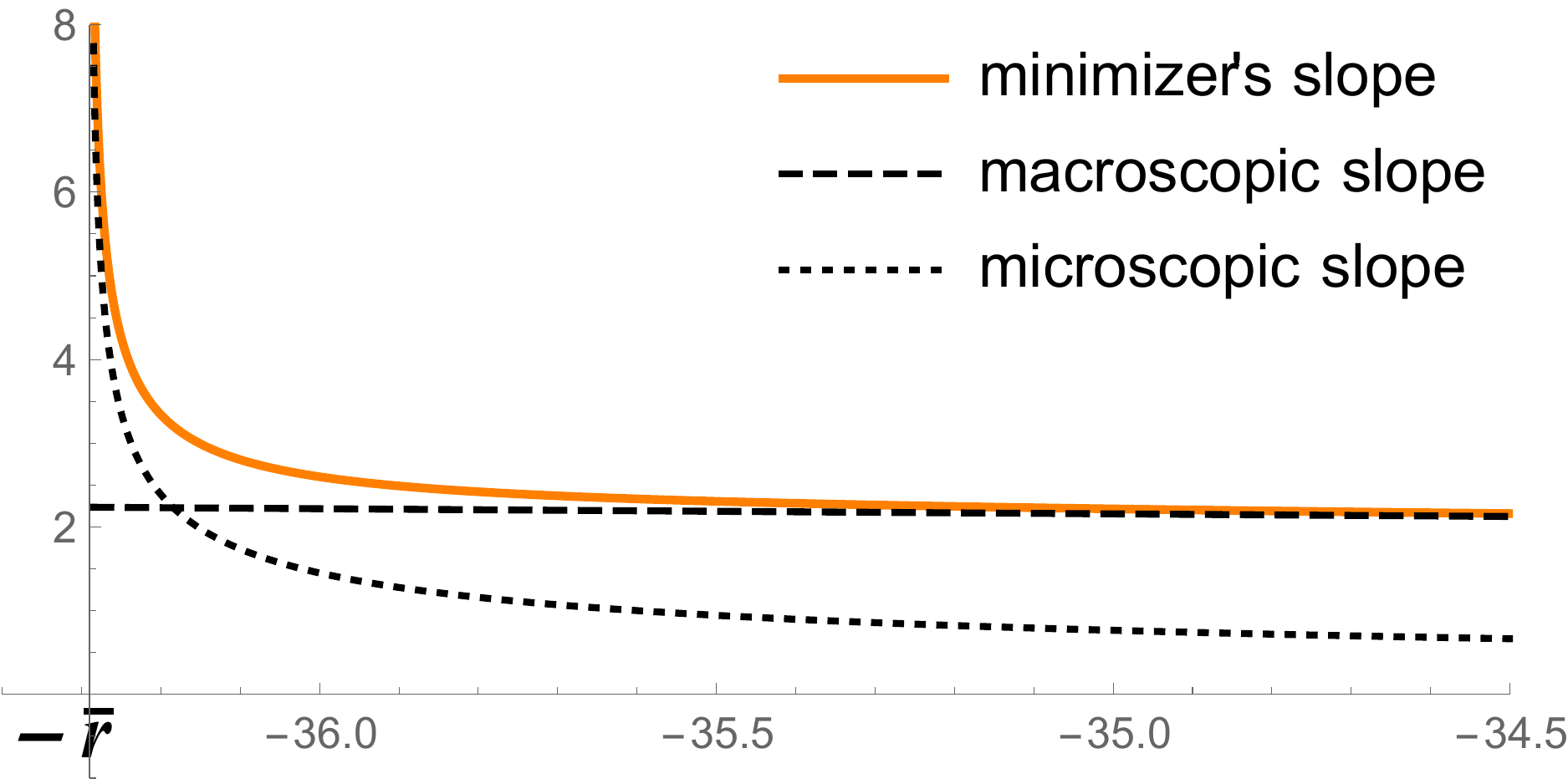}
  }
  \captionof{figure}{{\footnotesize
  As in Fig. \ref{fig-ode}, $Q=Q_a$ with $A=1$, $B=0$, $m=5/2$; here $S=-5/2$. {\bf Left:} the minimizer of $E$ (solid) with $M=2000$ --obtained as in Fig. \ref{fig-ode} with a tolerance of $10^{-1}$-- and its macroscopic profile (dashed), as defined in \eqref{dfg}. {\bf Right:} the slope of the same functions near the contact line $-\rr$, together with the slope of the microscopic profile (dotted), as defined in \eqref{dfg}.
  }}\label{fig-4}
  }}

\noindent $\bullet$  {\bf Pancake:} when $\e_*<+\infty$, we have
\begin{equation*}
u(x)\sim \left\{
\begin{array}{lll}
\e_* &  \delta\lesssim \rr- x\le \rr & \mbox{(macroscopic profile),}
\\
\left(\tfrac{A(m+1)^2}{2}\right)^{\tfrac{1}{m+1}}(\rr- x)^{\frac{2}{m+1}} & 0\le \rr- x\lesssim \delta & \mbox{(microscopic profile),}
\end{array}\right. \quad
\end{equation*}
where
$$
\rr \sim \frac{M}{2\e_*}, \quad \delta^{2} \sim  e_*^{m+1} \left(\tfrac{A(m+1)^2}{2}\right)^{-1}.
$$
The pancake's thickness $e_*<+\infty$ is defined in \eqref{def-s*-intro}: it satisfies $R'(\e_*)=0$, that is,
\begin{equation}\label{char-e}
Q(\e_*)=\e_* Q'(\e_*).
\end{equation}

$\bullet$ {\bf Transition profile:} the behavior of $f_p$ and $f_p^{-1}$ is detailed in Remark \ref{rem:f}. From there, we see that the droplet has three regimes:
\begin{equation*}
u(x)\sim \left\{
\begin{array}{lll}
\frac{u_{0}}{\rr^2} \left(\rr^2 - \tfrac{f_p^2(0)}{4}x^2\right) &   \rr- x\approx \rr & \mbox{(macroscopic profile),}
\\
\frac{u_{0}}{\rr^\frac{2}{p+1}} \left(\tfrac{(p+1)f_p(0)}{2\sqrt{p}} (\rr-x)\right)^\frac{2}{p+1} &   1\ll \rr- x\ll \rr & \mbox{(intermediate profile),}
\\
\left(\tfrac{A(m+1)^2}{2}\right)^{\frac{1}{m+1}}(\rr- x)^{\frac{2}{m+1}} & 0\le \rr- x\ll 1 & \mbox{(microscopic profile),}
\end{array}\right. \quad
\end{equation*}
where $\rr=\rr_M$ and $u_0=u_{0M}$ are as in \eqref{crit-intro}.

\subsection{Repulsive potentials: comparison with de Gennes' results}

In part II.D of his milestone review \cite{DG}, where final spreading equilibria are discussed, de Gennes considers two model cases. The first one (``van der Waals forces''), on which we focus, is of the generic form \eqref{modelQ0} with $m=3$ and $n=4$, which corresponds to $Q_{b}$ with $m=3$ and $n=4$. Now, we know from Lemma \ref{m>3} that minimizers of $E$ in $\mathcal D$ do not exist if $m\ge 3$. However, de Gennes confines his analysis to scales not below 30\AA, where ``a continuum picture is still applicable''. In any event, our results show that, replacing $m=3$ by a generic exponent $m\in(1,3)$, most of his formal predictions can be rigorously justified down to $u=0$. To proceed further, we distinguish three cases.

\smallskip

{\bf Partial wetting} ($-S> 0$ in $Q_{b}$). When $-S>0$, the macroscopic shape is of droplet type (see \eqref{par-intro}). Our results confirm, in the case of negligible gravitational effects, both the relation between $S$ and the macroscopic contact angle and, in the limiting case $m=3$, the estimate for the transitional thickness (compare \eqref{our-t} and \eqref{our-thick} with \cite[(2.54) and the discussion below it]{DG}).

\smallskip

{\bf Limiting case} ($-S=0$ in $Q_{b}$). In the limiting case $-S=0$, the macroscopic shape is given by $f_p^{-1}$ (see \eqref{crit-intro}). In particular, in the limiting case $m=3$ and for $p=4$, we recover the same scaling exponents for the microscopic and intermediate regimes in \cite[(2.55)-(2.56)]{DG},
\begin{equation*}
u(x)\sim \left\{
\begin{array}{lll}
\frac{u_{0}}{\rr^2} \left(\rr^2 - \tfrac{f_p^2(0)}{4}x^2\right) &   \rr- x\approx \rr & \mbox{(macroscopic profile),}
\\
\frac{u_{0}}{\rr^{2/5}} \left(\tfrac{5 f_p(0)}{4} (\rr-x)\right)^{2/5} &   1\ll \rr- x\ll \rr & \mbox{(intermediate profile),}
\\
\left(8A\right)^{1/4}(\rr- x)^{1/2} & 0\le \rr- x\ll 1 & \mbox{(microscopic profile),}
\end{array}\right. \quad
\end{equation*}
the only difference being in the multiplicative constants, which turn out to depend on $f_p(0)$ and are therefore expressed in terms of $\Gamma$ functions.

\smallskip

{\bf ``Dry'' complete wetting} ($-S<0$ in $Q_{b}$, and $Q_{b,g}$). In \cite{DG}, only the case $Q_{b,g}$ (with gravity) with $-S<0$ is discussed. However, we see from Table \ref{tabella} that the same qualitative result (pancake shape) holds for two other cases which do not seem to have been discussed there:
\begin{itemize}
\item model $Q_{b}$ (without gravity) when $-S<0$;
\item model $Q_{b,g}$ with $-S\ge 0$.
\end{itemize}
The characterization of $\e_*$ in \eqref{char-e} coincides with that in \cite[(2.63)]{DG} whenever $\e_*$ is uniquely defined. In particular, one easily checks that if $D$ is relatively small and $S$ is relatively large, namely
$$
S^{n-m} \gg A^{n-1}|B|^{1-m} \quad\mbox{and}\quad D^{n-m} \ll A^{n+1}B^{-m-1},
$$
then
$$
u Q'(u)-Q(u)\sim -Amu^{1-m}+S \quad\mbox{for $u\ll\left(\frac{|B|}{A}\right)^{\frac{1}{n-m}}$}, \quad\mbox{hence}\quad \e_*\sim \left(\frac{mA}{S}\right)^\frac{1}{m-1},
$$
which coincides with  \cite[(2.72)]{DG} in the critical case $m=3$. However, the reader can easily realize that there are various other possibilities, depending on the relation between the four parameters $S,A,B,D$.

\begin{remark}{\rm
As Table \ref{tabella} shows, the above conclusions holds not only for model $Q_b$, but also for model $Q_a$ when $B\le 0$ (which, in this case, is also purely repulsive). In addition, the qualitative aspects of our results remain true for the second model potential considered by de Gennes (``double-layer forces''):
\begin{equation*}
Q(u)= \left\{\begin{array}{ll}
A|\log u|(1+o(1)) & \mbox{as $u\to 0^+$}
\\[1ex]
-S- B\text{e}^{-u}(1+o(1)) & \mbox{as $u\to +\infty$.}
\end{array}\right.
\end{equation*}
However, quantitative information need be modified in this case, taking into account that a log singularity of the potential corresponds to the limiting case ``$m=1$''. We refrain from doing that for the sake of brevity.
}\end{remark}

\subsection{Repulsive/attractive potentials} As we mentioned in Section \ref{ss-frame}, potentials which are short-range repulsive and long-range attractive, such as  model case $Q_{a}$ with $B>0$, have been widely discussed in the thin-film literature, using various forms of them, especially in order to model and analyze coarsening dynamics and dewetting phenomena; however, qualitative studies of mild singularities ($Q(u)\equiv 0$ for $u\le 0$ and $1< m<3$) seem to have been missing so far.

\smallskip

\begin{minipage}[t]{0.65\textwidth}
When gravity is present, the minimizer for $M\gg 1$ is invariably a pancake. Hence we focus on $Q_a$ with $B>0$. The different possible behaviors are summarized in Figure \ref{fig-Q1a}, where $B\le 0$ is also shown for completeness.

If $-S\le 0$ (complete wetting), a unique minimizer exists with pancake shape. However, as opposed to purely repulsive potentials, a unique pancake-shaped minimizer may exist in the partial wetting regime ($-S>0$), too, provided $B$ is sufficiently large. In addition, as we mentioned already in Section \ref{ss-intro-uniq}, droplet-shaped minimizers can fail to be unique for moderate values of $B$.

\end{minipage}\quad
\begin{minipage}[t]{0.3\textwidth}
\captionsetup{width=0.9\linewidth}
\centering\raisebox{\dimexpr \topskip-\height}{
  \includegraphics[width=\textwidth]{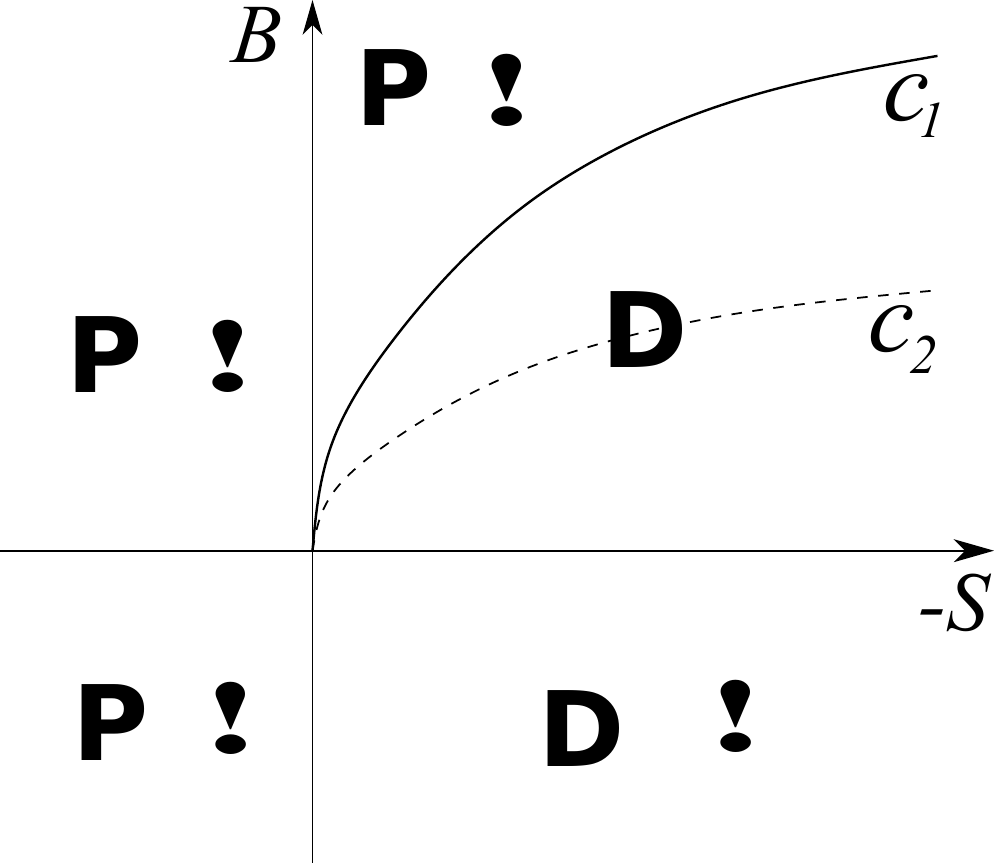}}
  \captionof{figure}{{\footnotesize Model $Q_a$: D=droplet, P=pancake, !=uniqueness.}}
  \label{fig-Q1a}
\end{minipage}

\subsection{Open questions}

We conclude by discussing what we believe to be the main open questions in the problem we discussed.

\smallskip

{\bf Uniqueness.} Uniqueness is left open only in some cases (see Table \ref{tabella}). In particular, we expect non-uniqueness phenomena to occur whenever $Q$ is not convex in $(0,\e_*)$, without the additional assumption that $R$ is not injective (Theorem \ref{no-un}); e.g., model $Q_a$ with $-S>0$ and $0<B< c_2(A,S)$.

\smallskip

{\bf Higher dimension.} Our qualitative study is one-dimensional. In higher dimensions, it is still possible to characterize the eigenvalue $\lambda$. However, the relation between $u$ and $\lambda$ becomes nonlocal, involving integrals of functions of $u$ rather than $u(0)$ alone (cf. \ref{def-R-intro}). In addition, the Euler-Lagrange equation (in radial variable) becomes non-autonomous. The combination of these two features so far prevented us from developing an analogous qualitative study for $N>1$.

\smallskip

{\bf Critical points and their stability.} This manuscript is concerned with global minimizers of $E$ in $\mathcal D$. However, we expect that $E$ also has critical points in $\mathcal D$, consisting of two or more radially decreasing solutions to \eqref{EL-intro}-\eqref{def-R-intro}, suitably translated so to have disjoint positivity sets, with a possibly different $\lambda$ for each of them. It would be very interesting to prove that such configurations are indeed critical points of $E$ in $\mathcal D$ and to study their (in)stability in either variational and/or dynamical sense, see e.g. \cite{BCS,CC10,CC17,CZ,KNC,LP3,N}.

\smallskip

{\bf Full curvature problem.} The gradient part of the functional in \eqref{def-E} may be derived from Stokes system on the basis of the main assumption in lubrication approximation: the vertical lengthscale is much smaller than the horizontal one \cite{GO1,GO2}. If not for the full Stokes system, it would be interesting to perform an analogous study at least for the functional $E$ with $\tfrac12 |\nabla u|^2$ replaced by $\sqrt{1+|\nabla u|^2}$; in other words, the full curvature effect is retained, though the droplet is yet assumed to be a subgraph. In this case, we are only aware of the studies in \cite{NC1,NC2,NC3,NC4}, which concern existence and uniqueness for convex potentials in the one-dimensional case.

\smallskip

{\bf Dynamics.} Since the nineties \cite{BF}, a lot of work has been done on existence \cite{G95,BBD,BP,DGG,Gr} and qualitative properties (such as finite speed of propagation, waiting time, long-time behavior) \cite{B1,B2,BDGG,DGG1,GG,F1,F2} of the spreading dynamics associated to $E$, as modeled by {\em thin-film equations}, which in one space dimension formally read as
\begin{equation}\label{para}
u_t+\left(f(u)(u_{xx}-Q'(u))_x\right)_x=0\quad\mbox{on \ $\{u>0\}$}
\end{equation}
with $f$ depending on the slip condition adopted at the liquid-solid interface ($f(u)=u^3+bu^2$, $b>0$, for Navier slip). When the potential is sufficiently singular ($Q(u)\equiv +\infty$ for $u\le 0$ and/or $m\ge 3$), existence and uniqueness are rather simple, since the datum is to be positive and the solution will as well \cite{GRu,BGW}. On the other hand, for mildly or non- singular potentials, \eqref{para} turns into a genuine free boundary problem. Concerning weak solutions, most efforts in its study concentrated on existence and qualitative properties of ``zero contact-angle'' solutions, which satisfy $u_x=0$ at $\partial\{u>0\}$ for a.e. $t>0$. We mention in particular \cite{BP2,DGS,NS}, where the case of power-law potentials is discussed. However, these zero contact-angle solutions have the property of converging to their mean for the Neumann problem on a bounded domain \cite{BP2}, regardless of their initial mass. It would be very interesting to see whether different classes of solutions to \eqref{para} exist, which instead satisfy a right-angle condition at $\partial \{u>0\}$ and converge to a stable critical point of $E$ for long times. Such achievements would be analogous to the ones regarding weak solutions with $Q\equiv 0$ and finite non-zero microscopic contact-angle \cite{O,BGK,M,CGw}. First results in this direction are contained in \cite{DuGi2}: there, formal arguments support the existence of generic (both advancing and receding) traveling wave solutions of \eqref{para} for any speed and any $m\in (1,3)$, with a contact angle of $\pi/2$ at $\partial\{ u>0\}$. Notably, such waves exist even {\em without} slip conditions (i.e. for $b=0$): hence mildly singular potentials may be seen as an alternative solution to the contact-line paradox.

More recently, a well-posedness theory of ``classical'' solutions has been developed for both zero \cite{GKO,GGKO,Gn1,Gn2,GP,S} and fixed non-zero \cite{K1,K2,KM1,KM2} contact-angle. As a further step, it would also be interesting to develop a theory of ``classical'' solutions for the singular potentials $Q$ addressed here. Another interesting question concerns, in the pancake case with $\e_*\ll 1$, intermediate scaling laws for macroscopic droplets spreading over a microscopic pancake, in the spirit of \cite{GO3}.

\subsection{Notations}

We define $\R_+:=(0,+\infty)$. By $C^{\alpha}_{[loc]}(\Omega)$ we mean the space of [locally] H\"older continuous functions with exponent $\alpha$ in $\Omega\subset \RN$. The subscript $c$ denotes spaces of functions with compact support.
We denote by $|\Omega|$ the Lebesgue measure of a Lebesgue measurable subset $\Omega$ of $\R^N$. We denote by $B_R(x)$ the ball of radius $R$ and center $x$ in $\RN$ and by $\omega_{N-1}$ the $(N-1)$-dimensional measure of the unit sphere $\mathbb S^{N-1}=\partial B_1(0)$. The Sobolev conjugate exponent of $2$ is denoted by $2^*=\frac{2N}{N-2}$.
For a measurable function $f$, we define
$$
f_+:=\max(f,0), \quad f_-:= -\min(f,0), \quad \supp f:=\overline{\{x\in\mbox{dom}\,f : \ f(x)\neq 0\}}.
$$
We omit the domain of integration when it coincides with $\R^N$, and (when no ambiguity occurs) we also omit the differential $dx$ when $x$ is (a rescaling of) the spatial independent variable.
If not otherwise specified, we will denote by $C$ several constants whose value may change from line to line. These values will only depend on the data (for instance, $C$ may depend on $N$).

\smallskip

We say that a function is {\em radially strictly decreasing}, resp. {\em radially non-increasing} (with respect to $x_0\in \R^N$), if it is radially symmetric (with respect to $x_0\in \R^N$) and the corresponding radial function is strictly decreasing, resp. non-increasing.

\section{Existence of a minimizer}
\label{s1}

Note that \eqref{H} implies that
\begin{equation}\label{def-s1}
s_1:=\sup\{s\in \R: \ Q>0\mbox{ in $(0,s)$} \} \in (0,+\infty].
\end{equation}
In this section we discuss existence and basic properties of minimizers.

\begin{theorem}[Existence and basic properties of minimizers]
\label{exist}
Assume \eqref{H}, $m<3$, and $M>0$. Then there exists a minimizer $u$ of $E$ in $\D$. Moreover, $u$ is radially non-increasing w.r.to a certain $x_0\in \R^N$, $u$ is compactly supported, and $u\in C^{1/2}_{loc}(\R^N\setminus \{x_0\})$.
\end{theorem}

We divide the proof into lemmas.

\begin{lemma}
\label{lem1}
There exists $u\in \D$ such that $E[u]<+\infty$.
\end{lemma}

\begin{proof}
We define $u:\RN\to [0,+\infty)$ as
\begin{equation*}
u(x):=d(1-|x|^2)_+^\alpha,\qquad \alpha\in\left(\tfrac{1}{2},\tfrac{1}{m-1}\right),
\end{equation*}
where $d$ is chosen so that $\int u=M$. Note that the range of $\alpha$ is not empty since $1<m<3$.
Straightforward computations show that
\begin{eqnarray*}
\int |\nabla u|^2=4\alpha^2d^2 \int (1-|x|^2)_+^{2(\alpha-1)}|x|^2=4\alpha^2d^2\omega_{N-1} \int_0^1 (1-r^2)^{2(\alpha-1)}r^{N+1} dr\stackrel{\alpha>\frac{1}{2}}<+\infty.
\end{eqnarray*}
Since $\supp u=B_1(0)$, we deduce that $u\in\mathcal D$. Moreover $Q(0)=0$, so that
\begin{eqnarray*}
\int Q(u) &=& \int_{B_1(0)} Q(u) \stackrel{\eqref{H}}\leq C \int_{B_1(0)}\left(1+\frac{1}{(1-|x|^2)^{\alpha(m-1)}}\right) \\
&=& C\omega_{N-1}\int_0^1 \left(1+ \frac{1}{(1-r^2)^{\alpha(m-1)}}\right) r^{N-1} dr\stackrel{\alpha<\frac{1}{m-1}}<+\infty,
\end{eqnarray*}
hence $E[u]<+\infty$.
\end{proof}

\begin{remark}{\rm
\label{m<3}
The previous lemma is false if $m\ge 3$ (see Lemma \ref{m>3}).
}\end{remark}

\begin{lemma}
\label{lem2}
There exists $C\ge 0$ such that
\begin{equation}
\label{lower-bound}
\int Q(u)\ge -C
\end{equation}
for any $u\in \mathcal D$. In particular, $E$ is bounded from below in $\mathcal{D}$.
\end{lemma}

\begin{proof}
The lower bound of $E$ is immediate from \eqref{lower-bound}. Recall that $Q>0$ in  $(0,s_1)$, with $s_1$ defined in \eqref{def-s1}. If $s_1=+\infty$ then $Q$ is non-negative and \eqref{lower-bound} is obvious with $C=0$. Otherwise, we have
\begin{equation*}
\int Q(u) \geq \int_{\{u\geq s_1\}} Q(u) \stackrel{\eqref{H}}\geq \inf Q\int_{\{u\geq s_1\}} 1 \geq \inf Q\int \frac{u}{s_1} \stackrel{\eqref{def-D}}\geq \frac{M}{s_1}(\inf Q)
\end{equation*}
for any $u$ in $\D$, whence \eqref{lower-bound} with $C=-\frac{M}{s_1}(\inf Q)$.
\end{proof}

\begin{lemma}
\label{lem3}
There exists a minimizer $u$ of $E$ in $\mathcal{D}$. Moreover, $u$ is radially non-increasing.
\end{lemma}

\begin{proof}
It follows from Lemma \ref{lem1} and Lemma \ref{lem2} that there exists a minimizing sequence $\{\tilde{u}_k\}$ in $\D$ such that
\begin{equation}\label{limEk}
\lim_{k\to +\infty}E[\tilde{u}_k]=\inf_\D E \in \mathbb R.
\end{equation}
In particular, for all $k\in \N$, we have that
\begin{equation}
\label{precQ}
-C\stackrel{\eqref{lower-bound}}\le \int Q(\tilde u_k) \le E[\tilde u_k] \stackrel{\eqref{limEk}}\le C
\end{equation}
and
\begin{equation}\label{prec}
\frac12 \int|\nabla \tilde u_k|^2 \stackrel{\eqref{def-E}}= E[\tilde u_k] -\int Q(\tilde u_k) \stackrel{\eqref{precQ}}\le C.
\end{equation}
In addition, by Nash inequality \cite{CM},
\begin{equation}\label{prec2}
\int \tilde u_k^2\le C \left(\int \tilde u_k\right)^{\frac{4}{N+2}}\left(\int|\nabla \tilde u_k|^2\right)^{\frac{N}{N+2}}\stackrel{\eqref{prec}}\le C M^{\frac{4}{N+2}}.
\end{equation}
For $k\in\N$, let $u_k$ be the Schwarz symmetrization of $\tilde{u}_k$. Then (\cite[\S 1.3 and 1.4]{Kes})
\begin{equation}
\label{bas-sd}
u_k\ge 0, \quad \mbox{$u_k(x)=v_k(|x|)$ with $v_k$ non-increasing in $(0,+\infty)$}, \quad \int u_k=M.
\end{equation}
Applying the P\'olya-Szeg\H{o} inequality (\cite[Theorem 2.3.1 and Remark 2.3.5]{Kes}) and since the Schwarz symmetrization preserves the $L^2$-norm (\cite[\S 1.3 and 1.4]{Kes}), we have that
\begin{equation}\label{symmprop1}
\int|\nabla u_k|^2\le \int|\nabla \tilde u_k|^2 \quad\mbox{and}\quad  \int u_k^2=\int \tilde u_k^2.
\end{equation}

In order to show that $Q(\tilde u_k)$ is uniformly bounded in $L^1(\RN)$, we estimate
\begin{equation}
\label{min-2}
\int Q_-(\tilde u_k)\stackrel{\eqref{H},\eqref{def-s1}}\leq -\inf Q \int_{\{\tilde u_k\geq s_1\}} 1 \leq -\inf Q \int \frac{\tilde u_k}{s_1}\stackrel{\tilde u_k\in\mathcal D}=-\frac{M}{s_1}(\inf Q)
\end{equation}
and
\begin{equation}
\label{min-3}
\int Q_+(\tilde u_k)=\int Q(\tilde u_k) + \int Q_-(\tilde u_k)\stackrel{\eqref{precQ},\eqref{min-2}}\leq C.
\end{equation}
Combining \eqref{min-2} and \eqref{min-3} we conclude that
\begin{equation}\label{|Q|}
\int |Q(\tilde u_k)| \le C \quad\mbox{for all $k\in\N$.}
\end{equation}
Since $Q(\tilde u_k) \in L^1(\RN)$, we may apply the results in \cite[Section 1.3]{Kes}:
\begin{equation}
\label{symmprop2}
\int Q(u_k)=\int Q(\tilde u_k) \quad\mbox{and}\quad \int |Q(u_k)|=\int |Q(\tilde u_k)| \stackrel{\eqref{|Q|}}\le C \quad\mbox{for every $k\in\N$.}
\end{equation}
In particular,
\begin{equation}\label{boundH1}
\|u_k\|_{\h}\stackrel{\eqref{symmprop1}}\leq \|\tilde{u}_k\|_{\h} \stackrel{\eqref{prec},\eqref{prec2}}\le C\quad\mbox{and}\quad  E [u_k]\stackrel{\eqref{symmprop1},\eqref{symmprop2}}\leq E[\tilde{u}_k]
\end{equation}
for all $k$ in $\N$.
This implies that $\{u_k\}$ is another minimizing sequence in $\D$. In view of \eqref{boundH1}, there exists a non-negative, radially non-increasing function $u\in \h$ such that
\begin{equation}
\label{conv}
\begin{array}{l}
u_k \rightharpoonup u \text{ in } \h,\\
[2ex]
u_k \rightarrow u \text{ a.e. in $\R^N$ and in $L^p(B_R(0))$ for any $R>0$ and any }1\leq p < 2^{*}.
\end{array}
\end{equation}

We now show that $\int u=M$, hence $u\in \mathcal D$. For any $R\ge 1$, we have
\begin{equation*}
M \geq \int_{B_R(0)\setminus B_{R/2}(0)} u_k \stackrel{\eqref{bas-sd}}=\omega_{N-1}\int_{R/2}^R r^{N-1}v_k(r)dr \stackrel{\eqref{bas-sd}}\geq \omega_{N-1}\left(\frac{R}{2}\right)^Nv_k(R),
\end{equation*}
therefore
\begin{equation}
\label{min3}
v_k(R)\leq CR^{-N} \quad \mbox{for all $R\geq 1$ and all $k\in\N$}.
\end{equation}
Since $\int u_k=M$ for all $k$, we have
\begin{eqnarray}
\label{min6}
M && =\lim_{k\to\infty}\int u_k =\lim_{k\to\infty}\int_{B_R(0)} u_k + \lim_{k\to\infty}\int_{\RN\setminus B_R(0)}u_k \nonumber \\
&&\stackrel{\eqref{conv}}=\int_{B_R(0)} u + \lim_{k\to\infty}\int_{\RN\setminus B_R(0)}u_k.
\end{eqnarray}
In order to estimate the second integral on the right-hand side of \eqref{min6}, we note that, in view of \eqref{H} and \eqref{min3},
\begin{equation}
\label{pippo}
\mbox{$Q(u_k)\ge \tfrac{A}{2} u_k^{1-m}\chi_{\{u_k>0\}}$ in $\R^N\setminus B_R(0)$ for $R\gg 1$.}
\end{equation}
Therefore
\begin{eqnarray}\label{rev1}
\nonumber \frac{1}{\omega_{N-1}}\int_{\RN\setminus B_{R}(0)} u_k
&\stackrel{\eqref{bas-sd},\eqref{min3}} \leq  & C R^{-N}\int_{R}^{+\infty}\chi_{\{v_k>0\}}r^{N-1}d r
\\
 & \stackrel{\eqref{bas-sd}}\le & CR^{-N}(v_k(R))^{m-1} \int_{R}^{+\infty} (v_k(r))^{1-m}\chi_{\{v_k>0\}}r^{N-1} dr
\nonumber \\ & \stackrel{\eqref{pippo},\eqref{min3}}\le &  C R^{-Nm} \int_{R}^{+\infty} Q(v_k) r^{N-1} dr
\nonumber \\ & \stackrel{\eqref{bas-sd}}\le & C R^{-Nm} \int |Q(u_k)| \stackrel{\eqref{symmprop2}}\le C R^{-Nm}.
\end{eqnarray}
Passing to the limit in \eqref{min6} as $R\to +\infty$ using \eqref{rev1} and Beppo Levi's theorem, we conclude that $\int u=M$, hence $u\in\mathcal D$.

\smallskip

Now we prove that $u$ is a minimizer of $E$ in $\mathcal{D}$. The passage to the limit in the Dirichlet energy is straightforward by lower semi-continuity:
\begin{equation}
\label{min7}
\int |\nabla u|^2\leq \liminf_{k\to\infty}\int |\nabla u_k|^2.
\end{equation}
Let's focus on the potential energy.
Let $0<\delta<s_1$ be such that $|\{u=\delta\}|=0$ (note that this is the case for a.e. $\delta>0$, see Lemma \ref{aux2}). Choosing $R\gg\delta^{-1/N}$ and using \eqref{min3}, we deduce that $u_k \leq C R^{-N}<\delta$ in $\RN\setminus B_R(0)$ for all $k$ sufficiently large. Therefore
\begin{equation}
\label{min9}
\int Q(u_k)\stackrel{\delta<s_1}\geq \int Q(u_k)\chi_{\{u_k>\delta\}}=\int_{B_R(0)}Q_{+}(u_k)\chi_{\{u_k>\delta\}}-\int_{B_R(0)}Q_{-}(u_k)\chi_{\{u_k>\delta\}}.
\end{equation}
It follows from \eqref{conv} that $u_k\to u$ a.e. in $\RN$. Hence, since $|\{u=\delta\}|=0$, $\chi_{\{u_k>\delta\}}\to \chi_{\{u>\delta\}}$ a.e.. Moreover, by \eqref{H}, $Q_{-}\le C$. Hence by Fatou lemma, applied to the first term on the righ-hand side of \eqref{min9}, and Lebesgue theorem, applied to the second term on the right-hand side of \eqref{min9}, we obtain
\begin{equation}
\label{min10}
\liminf_{k\to\infty}\int Q(u_k) \geq \int_{B_R(0)} Q(u)\chi_{\{u>\delta\}}.
\end{equation}
Since $u<\delta$ in $\RN\setminus B_R(0)$, the right-hand side of \eqref{min10} coincides with $\displaystyle \int Q(u)\chi_{\{u>\delta\}}$. Therefore
\begin{equation*}
\liminf_{k\to\infty}\int Q(u_k) \geq \int Q(u)\chi_{\{u>\delta\}}=\int Q(u)\chi_{\{\delta<u<s_1\}}+\int Q(u)\chi_{\{u\geq s_1\}}.
\end{equation*}
By Lemma \ref{aux2}, $|\{u=\delta\}|=0$ for a.e. $\delta>0$; by Lemma \ref{aux}, $\chi_{\{u>\delta\}}\to \chi_{\{u>0\}}$ a.e. in $\RN$ as $\delta \to 0$. Therefore, by Beppo Levi's theorem, we obtain
\begin{equation}
\label{min11}
\liminf_{k\to\infty}\int Q(u_k) \geq \lim_{\delta\to 0} \int Q(u)\chi_{\{\delta<u<s_1\}}+\int Q(u)\chi_{\{u\geq s_1\}} = \int Q(u).
\end{equation}
Since $\{u_k\}$ is a minimizing sequence, it follows from \eqref{min7} and \eqref{min11} that $E[u]\le \liminf_{k\to +\infty}E[u_k]=\inf_\D E$, hence $u$ is a minimizer of $E$ in $\mathcal{D}$.
\end{proof}

\begin{remark}
\label{bN=1}
{\rm If $N=1$, the compact embedding $H^1(\R)\Subset C^{\frac{1}{2}}(\R)$ implies that any minimizer belongs to $C^{\frac{1}{2}}(\R)$.
}\end{remark}

\begin{remark}
\label{|Qu|}
{\rm
Any minimizer $u$ of $E$ in $\mathcal D$ satisfies $Q(u)\in L^1(\R^N)$. Indeed, with the same argument just used in the proof of Lemma \ref{lem3} (cf. \eqref{min-2}-\eqref{|Q|}), we have
\begin{equation*}
\int Q_-(u)\stackrel{\eqref{H},\eqref{def-s1},u\in\mathcal{D}}\leq -\frac{M}{s_1}(\inf Q) \quad \text{and}\quad \int Q_+(u)=\int Q(u) + \int Q_-(u)\stackrel{E[u]<+\infty}\leq C.
\end{equation*}
}
\end{remark}

The next Lemma implies that the minimizer given by Lemma \ref{lem3} has compact support.

\begin{lemma}
\label{lem4}
Any radially non-increasing function $u:\RN\to [0,+\infty)$ such that
\begin{equation}
\label{hlem4}
u\in L^1(\RN) \qquad\text{and}\qquad \int Q(u)<+\infty
\end{equation}
has compact support.
\end{lemma}

\begin{proof}
We can assume without loss of generality that $u$ is radially symmetric with respect to $x_0=0$ in $\RN$: $u(x)=v(|x|)$ with $v$ non-increasing in $(0,+\infty)$. Since $u\in L^1(\Omega)$, arguing as in the proof of \eqref{min3} we see that $v(R)\leq CR^{-N}$ for all $R\geq 1$. Consequently, arguing as in \eqref{rev1} we obtain for $R\gg 1$
\begin{eqnarray}\label{rev1-part-bis}
\int_{R}^{+\infty}\chi_{\{v>0\}}r^{N-1} dr \le  C R^{-N(m-1)} \int_{\R^N\setminus B_R(0)} Q(u).
\end{eqnarray}
Since
$$
\int_{\R^N\setminus B_R(0)} Q(u) = \int Q(u)-\int_{B_R(0)} Q(u) \le \int Q(u) + \omega_{N-1}\frac{R^N}{N} \max\{0, -\inf Q\},
$$
it follows from \eqref{rev1-part-bis} that
\begin{equation*}
\int_{R}^{+\infty}\chi_{\{v>0\}}r^{N-1} dr \leq C R^{-N(m-1)}\left(\int Q(u) + R^N \max\{0, -\inf Q\}\right) \stackrel{\eqref{hlem4},\eqref{H}} <+\infty.
\end{equation*}
Since $v$ is non-increasing, this implies that $v(r)=0$ for all $r$ sufficiently large and completes the proof.
\end{proof}

Now we recall a simple property of radially symmetric functions in $H^1(\RN)$.
\begin{lemma}
\label{lem5}
If $u\in H^1(\RN)$ is radially symmetric w.r.to $x_0$, then $u\in C^{\frac{1}{2}}_{loc}(\RN\setminus\{x_0\})$.
\end{lemma}

\begin{proof}
We may assume w.l.o.g. that $x_0=0$. Let $\delta>0$. Since $u$ is radially symmetric, there exists a function $v:(0,+\infty)\to\R$ such that $u(x)=v(|x|)$. Since $u\in\h$, there exists $v'$, the distributional derivative of $v$, such that $\nabla u(x)= v'(|x|)x/|x|$. Then
$$
C\geq \int_{\RN\setminus B_\delta(0)} |\nabla u|^2=\omega_{N-1}\int_\delta^{+\infty} |v'(r)|^2r^{N-1}dr\geq \omega_{N-1}\delta^{N-1}\int_\delta^{+\infty} |v'(r)|^2dr.
$$
This implies that $v'$ belongs to $L^2((\delta,+\infty))$. Using the same argument we get $v\in L^2((\delta,+\infty))$. Therefore $v\in H^1((\delta,+\infty))\Subset C^{\frac{1}{2}}((\delta,+\infty))$, by the Sobolev embedding theorem. Hence
$$
|u(x_1)-u(x_2)|= |v(|x_1|)-v(|x_2|)|\le C_\delta ||x_1|-|x_2||^{1/2}\le C_\delta |x_1-x_2|^{1/2}
$$
for all $x_1,x_2\in \R^N\setminus B_\delta(0)$. Since $\delta$ is arbitrary, the proof is complete.
\end{proof}

Collecting Lemma \ref{lem3}, Lemma \ref{lem4} and Lemma \ref{lem5} we obtain Theorem \ref{exist}.

\medskip

We conclude the section by proving the result anticipated in Remark \ref{m<3}.

\begin{lemma}
\label{m>3}
If $m\ge 3$, then $\displaystyle \int Q(u) = +\infty$ for all $u\in \mathcal D$. Consequently, $E[u]=+\infty$ for all $u\in \mathcal D$.
\end{lemma}

\begin{remark}{\rm
Lemma \ref{m>3} is related to Theorem 2 of \cite{lm} (see also \cite{SunZha}): there, in the model case $Q(s)=As^{1-m}$, it is proved that a solution to the Euler-Lagrange equation associated to $E$ (cf. \eqref{eqmin} below) belongs to $H^1$ if and only if $m<3$.
}\end{remark}

\begin{proof}
Assume by contradiction that $\tilde u\in \mathcal D$ exists such that  $\displaystyle \int Q(\tilde u)<+\infty$. Let $u$ be the Schwarz symmetrization of $\tilde{u}$. Then
\begin{equation}
\label{nex2}
u\ge 0 \quad\text{and}\quad \mbox{$u(x)=v(|x|)$ with $v$ non-increasing in $(0,+\infty)$}.
\end{equation}
Arguing as in the first part of the proof of Lemma \ref{lem3}, we deduce that
\begin{equation}
\label{nex3}
u\in\mathcal\D \quad\text{and}\quad \int Q(u)=\int Q(\tilde u) <+\infty.
\end{equation}
It follows from Lemma \ref{lem4} and \eqref{nex2} that $\bar r$ exists such that $\supp v =[0,\bar r]$. Furthermore, by Lemma \ref{lem5}, $C>0$ exists such that
\begin{equation}
\label{nex5}
v(r)= v(r)-v(\bar r) \leq C(\bar r-r)^{\frac{1}{2}} \qquad \mbox{for all $\bar r/2\le r\leq \bar r$}.
\end{equation}
This implies that there exists $R\in [\bar r/2,\bar r)$ such that
\begin{equation}
\label{nex4}
\displaystyle Q(v) \stackrel{\eqref{H}}\ge \tfrac{A}{2} v^{1-m} \stackrel{\eqref{nex5}}\ge C (\rr-r)^{\frac{1-m}{2}} \quad\mbox{for all $v\in [R,\bar r)$}.
\end{equation}
Therefore
\begin{eqnarray*}
\int Q(\tilde u) &\stackrel{\eqref{nex3}}= & \int Q(u) \stackrel{\eqref{nex2}} = \omega_{N-1}\int_0^{\bar r} Q(v)r^{N-1} dr \nonumber \\
& \stackrel{\eqref{H},\eqref{nex4}}\geq &\frac{\omega_{N-1}}{N} R^N \inf Q +C \omega_{N-1} \int_R^{\bar r} (\bar r - r)^{\frac{1-m}{2}} dr
\stackrel{m\ge 3}= +\infty,
\end{eqnarray*}
and we have obtained a contradiction.
\end{proof}

\section{The Euler-Lagrange equation}
\label{s2}

In this section we assume that
\begin{equation}\label{H-EL1}
\mbox{$m<3$ \ and \ $Q\in C^1((0,+\infty))$ is such that \eqref{H} holds.}
\end{equation}
In higher dimension we will also need additional information on the behavior of $Q'(s)$ for large $s$ (more precisely on $Q'_{-}$ and $Q'_{+}$, the negative and positive parts of $Q'$):
\begin{equation}\label{H2}
Q'_{-}(s)\leq C \ \text{ for }s\gg 1 \quad\mbox{and}\quad Q'_{+}(s)\leq Cs^q  \ \text{ for }s\gg 1 \qquad \mbox{if $N\ge 2$}
\end{equation}
for some $C>0$, with $q<+\infty$ if $N=2$ and $q\le 2^*$ if $N\geq 3$. We will show:
\begin{theorem}
\label{alldec}
Assume \eqref{H-EL1}; if $N\ge 2$, assume in addition \eqref{H2}. Then any minimizer of $E$ in $\D$ has compact support, is radially strictly decreasing  w.r.to some $x_0\in \R^N$ in $\{u>0\}$, and is a classical solution to
\begin{equation}
\label{eqmin}
-\Delta u +Q'(u)=\lambda \quad \text{in } \{u>0\},
\end{equation}
for some $\lambda \in \R$, in the sense that $u\in C^{2}(\{u>0\})\cap C^{\frac{1}{2}}(\R^N)$ with $\nabla u(x_0)=\mathbf 0$.
\end{theorem}

\begin{remark}{\rm
Starting from the pioneering works \cite{CRT,lm}, an enormous interest has been given to the homogeneous Dirichlet problem for elliptic equations with a source term which is singular with respect to $u$ (see e.g. the Introduction of \cite{HMV}). Of course, in this case both the domain and $\lambda$ are fixed. Concerning the case $Q'(u)\approx - u^{-m}$ with $1<m<3$, we refer in particular to \cite{GiS,OlPe,GoGu,CaMu,DuOl}.
}\end{remark}

We begin by proving that \eqref{eqmin} is satisfied by any radially non-increasing minimizer.

\begin{proposition}
\label{EL}
Assume \eqref{H-EL1}; if $N\ge 2$, assume in addition \eqref{H2}. Then any radially non-increasing minimizer $u$ of $E$ in $\D$ is a distributional solution to \eqref{eqmin} for some $\lambda\in\R$.
\end{proposition}

\begin{proof}
Let $x_0$ be the symmetry center of $u$. It follows from Lemma \ref{lem4} and Lemma {\ref{lem5}} that $u\in C(\R^N\setminus\{x_0\})$ and that $\supp u =\overline{B_{\bar r}(x_0)}$ for some $\rr\in(0,+\infty)$. Let $\varphi\in \mathcal I$, where
$$
\mathcal{I}=\left\{\varphi\in C^1_c(\R^N) : \supp \varphi\Subset B_{\bar r}(x_0), \int \varphi=0\right\}.
$$
Since $u\in C(\R^N\setminus \{x_0\})$ and $u$ is radially non-increasing, $\varepsilon>0$ (depending on $\varphi$) exists such that $u\geq \varepsilon$ in $\supp \varphi$. Choosing
$|t|<\frac{\varepsilon}{\|\varphi\|_{\infty}},$
we have that
\begin{equation}\label{m1}
u+t\varphi\in \mathcal{D} \quad\mbox{and}\quad \chi_{\{u+t\varphi>0\}}=\chi_{\{u>0\}}=B_{\rr}(x_0) \quad \mbox{for all } \ |t|<\frac{\varepsilon}{\|\varphi\|_{\infty}}.
\end{equation}
It follows that
\begin{equation}
\label{app1}
\displaystyle\frac{E[u+t\varphi]-E[u]}{t} \stackrel{\eqref{def-E},\eqref{m1}}=\int \nabla u\cdot \nabla\varphi+\frac{t}{2}\int |\nabla\varphi|^2+\int \frac{Q(u+t\varphi)-Q(u)}{t}\chi_{\{u>0\}}.
\end{equation}
Passing to the limit as $t\to 0$ on the first two terms on the right-hand side of \eqref{app1} is trivial. For the third one, we will apply Lebesgue theorem. Firstly, since $Q\in C^1((0,+\infty))$ and in view of \eqref{m1}, we have
\begin{equation}\label{L1}
\frac{Q(u+t\varphi)-Q(u)}{t} \to Q'(u)\varphi \quad\mbox{pointwise in $B_{\rr}(x_0)\setminus \{x_0\}$\quad as $t\to 0$.}
\end{equation}
Next, we work out the $L^1$-estimate. Since $u\geq \varepsilon$ in $\supp \varphi$, we have
\begin{equation}\label{m3}
\frac{Q(u+t\varphi)-Q(u)}{t}\equiv 0 \quad\mbox{in } \ \{u<\eps\} \quad\mbox{for all $|t|<\frac{\varepsilon}{\|\varphi\|_{\infty}}$.}
\end{equation}
On $\{u\ge \eps\}\setminus \{x_0\}$, we write
\begin{eqnarray}\label{m2}
\frac{Q(u+t\varphi)-Q(u)}{t} = \tfrac{1}{t}\int_0^{t} Q'(u+\tau\varphi)\varphi d\tau.
\end{eqnarray}
We note that $u+t\varphi\ge \eps/2$ in $\{u\ge \eps\}$ for any $|t|<\frac{\varepsilon}{2\|\varphi\|_{\infty}}$. Therefore, if $N=1$, the boundedness of $u$ guaranteed by Remark \eqref{bN=1} implies that
\begin{equation}\label{N1}
\left|\tfrac{Q(u+t\varphi)-Q(u)}{t}\right| \stackrel{\eqref{m2}}\le \|\varphi\|_\infty \|Q'\|_{L^\infty([\frac{\varepsilon}{2}, \|u\|_\infty])} \quad\mbox{in $\{u\ge \eps\}$ $\forall |t|<\frac{\varepsilon}{2\|\varphi\|_{\infty}}$} \quad\mbox{if \ $N=1$.}
\end{equation}
If $N\ge 2$, since $u+t\varphi\ge \eps/2$ and $Q\in C^1((0,+\infty))$ satisfies \eqref{H2}, a constant $C$ (depending on $\eps$) exists such that $|Q'(s)|\le C(1+s^q)$ for any $s\in [\frac{\eps}{2},+\infty)$. Therefore, by the Sobolev embedding theorem,
\begin{eqnarray}
\nonumber \lefteqn{\left|\tfrac{Q(u+t\varphi)-Q(u)}{t}\right| \stackrel{\eqref{m2}}\le  C\|\varphi\|_\infty \left(1+(u+|t|\varphi)^q\right)}
\\ &\le & C\|\varphi\|_\infty \left(1+u^q\right)\in L^1(\{u\ge \eps\})\quad\mbox{for any $|t|<\frac{\varepsilon}{2\|\varphi\|_{\infty}}$} \quad\mbox{if \ $N\ge 2$.}
\label{N2}\end{eqnarray}
In view of \eqref{L1}, \eqref{m3}, and \eqref{N1}-\eqref{N2}, an application of Lebesgue theorem yields $\frac{Q(u+t\varphi)-Q(u)}{t}\chi_{\{u>0\}}\to Q'(u)\varphi$ in $L^1(B_{\bar r}(x_0))$ as $t\to 0$. Hence, passing to the limit as $t\to 0$ in \eqref{app1}, since $u$ is a minimizer of $E$ in $\mathcal D$, we obtain
\begin{equation}\label{apnea}
\int_{B_{\bar r}(x_0)} \left(\nabla u\cdot\nabla \varphi + Q'(u)\varphi\right)=0 \quad\mbox{for all $\varphi\in \mathcal I$.}
\end{equation}
To conclude, fix $\phi\in C_c^1(B_{\bar r}(x_0))$ such that $\phi\geq 0$ and $\int \phi \neq 0$. For any $\psi\in C_c^1(B_{\bar r}(x_0))$, we have
$$
\varphi= \psi - \left(\frac{\int \psi}{\int \phi}\right)\phi \in \mathcal I.
$$
It follows from \eqref{apnea} that
\begin{equation}\label{apnea2}
\int_{B_{\bar r}(x_0)} \left(\nabla u\cdot\nabla \psi + Q'(u)\psi\right)= \lambda \int_{B_{\bar r}(x_0)} \psi \qquad\mbox{for all $\psi\in C_c^1(B_{\bar r}(x_0))$},
\end{equation}
where
$$
\lambda=\frac{1}{\int \phi} \int_{B_{\bar r}(x_0)} \left(\nabla u\cdot\nabla \phi + Q'(u)\phi\right).
$$
Note that $\lambda\in \R$. Indeed, since $\phi \in C_c^1(B_{\bar r}(x_0))$, $\eps>0$ exists such that $u\geq \eps$ in $\supp \phi$. If $N\geq 2$, it follows from \eqref{H2} and the Sobolev embedding theorem that $Q'(u)\phi\in L^1(B_{\bar r}(x_0))$ (cf. \eqref{N2}). If $N=1$, the boundedness of $u$ implies that $Q'(u)\phi\in L^\infty(B_{\bar r}(x_0))$ (cf. \eqref{N1}). Thus, the arbitrariness of $\psi$ in \eqref{apnea2} implies the result.
\end{proof}

Next, we give a boundedness result for solutions to \eqref{eqmin} if $N\ge 2$.
\begin{lemma}
\label{bound}
Let $N\geq 2$. Assume \eqref{H-EL1} and \eqref{H2}. Let $u\in H^1_0(B_{\bar r}(x_0))$ be a radially non-increasing (with respect to $x_0\in \RN$) distributional solution to \eqref{eqmin} with $\supp u=\overline {B_{\bar r}(x_0)}$. Then $u\in L^{\infty}(B_{\bar r}(x_0))$.
\end{lemma}

\begin{proof}
Since $u$ is a distributional solution to \eqref{eqmin} with $\supp u=\overline {B_{\bar r}(x_0)}$, we have
\begin{equation*}
\int_{B_{\bar r}(x_0)} \nabla u\cdot\nabla \varphi + \int_{B_{\bar r}(x_0)} Q'(u)\varphi = \lambda\int_{B_{\bar r}(x_0)}\varphi \qquad \forall \varphi\in C^1_c(B_{\bar r}(x_0)).
\end{equation*}
Hence, taking a non-negative $\varphi$ and dropping the non-negative term concerning $Q'_{+}$, we obtain
\begin{equation}
\label{b2}
\int_{B_{\bar r}(x_0)} \nabla u\cdot\nabla \varphi \leq \int_{B_{\bar r}(x_0)} \left(\lambda + Q'_{-}(u)\right)\varphi  \qquad \forall \,\, 0 \leq \varphi\in C^1_c(B_{\bar r}(x_0)).
\end{equation}
By Lemma \ref{lem5}, $u$ is continuous outside $x_0$, which, since $u$ is radially non-increasing, implies that $u$ is bounded away from zero in any compact subset of $B_{\bar r}(x_0)$. This fact, together with \eqref{H2}, allow us to extend the class of test functions in \eqref{b2} to non-negative functions in $H^1_c(B_{\bar r}(x_0))$.

\smallskip

Let assume that $\sup u >1$ (otherwise there is nothing to prove). For $k\ge 1$, let $G_k(u)=(u-k)_+$. Note that $G_k(u)\in H^1_0(B_{\bar r}(x_0))$ and $\supp G_k(u)=\{u\geq k\}\subseteq \{u\geq 1\}\Subset B_{\bar r}(x_0)$ (again using continuity and monotonicity of $u$). Hence we can take $G_k(u)$ as test function in \eqref{b2}, obtaining
\begin{equation}
\label{b3}
\int_{B_{\bar r}(x_0)} |\nabla G_k(u)|^2 \leq \left(\lambda + \sup_{[1,+\infty)} Q'_{-}\right)\int_{B_{\bar r}(x_0)}G_k(u) \stackrel{\eqref{H2}}\leq  C\int_{B_{\bar r}(x_0)}G_k(u).
\end{equation}
Now fix
\begin{equation*}
2<q<
\begin{cases} +\infty & \text{ if }N=2, \\
2^* & \text{ if }N\geq 3.
\end{cases}
\end{equation*}
By the Sobolev embedding theorem, applied to the left-hand side of \eqref{b3}, and H\"older inequality, applied to the right-hand side of \eqref{b3}, we have
\begin{equation*}
\left(\int_{B_{\bar r}(x_0)} |G_k(u)|^{q}\right)^{\frac{2}{q}}
\leq C\int_{B_{\bar r}(x_0)}G_k(u) \leq  C|\{u\geq k\}|^{\frac{q-1}{q}} \left(\int_{B_{\bar r}(x_0)}G_k(u)^{q}\right)^{\frac{1}{q}},
\end{equation*}
for a constant $C$ depending on $N$ and $\rr$. Hence, for $h> k$,
\begin{equation}
\label{bn4}
(h-k)^{q}|\{u\geq h\}| \leq \int_{\{u\geq h\}} |G_k(u)|^{q}\leq \int_{\{u\geq k\}} |G_k(u)|^{q} \leq  C|\{u\geq k\}|^{q-1}.
\end{equation}
Starting from inequality \eqref{bn4} and applying Lemma 4.1 of \cite{Stam}, it is standard to conclude that $u\in L^{\infty}(B_{\bar r}(x_0))$.
\end{proof}

The facts that $u$ is bounded and solves the Euler-Lagrange equation yield regularity of $u$:

\begin{corollary}[Regularity of radially non-increasing minimizers]
\label{reg}
Assume \eqref{H-EL1}; if $N\ge 2$, assume in addition \eqref{H2}. Let $u\in \mathcal D$ be a radially non-increasing (w.r.to some $x_0\in \R^N$) minimizer of $E$ in $\D$. Then $u\in C^{2}(\{u>0\})\cap C^{\frac{1}{2}}(\R^N)$ and $\nabla u(x_0)=\mathbf 0$.
\end{corollary}

\begin{proof}
By Proposition \ref{EL}, $u$ is a distributional solution to \eqref{eqmin}. Moreover, by Lemma \ref{lem4}, $\supp u=\overline{B_{\bar r}(x_0)}$ for some $\bar r>0$. It follows from Remark \ref{bN=1} (for $N=1$) and Lemma \ref{bound} (for $N>1$) that $u\in L^{\infty}(B_{\bar r}(x_0))$. Since $Q\in C^1((0,+\infty))$, $Q'$ is locally bounded in $(0,+\infty)$: hence $-\Delta u$ is locally bounded in $\{u>0\}$. This implies that $u$ is locally H\"{o}lder continuous in $\{u>0\}$ (Theorem 8.22 of \cite{giltru}). Then, using once again that $u$ solves \eqref{eqmin} and that $Q'\in C((0,+\infty))$, we deduce that $u\in C^2(\{u>0\})$. Finally, H\"older continuity in $\R^N$ follows from Lemma \ref{lem5} and $\nabla u(x_0)= \mathbf 0$ follows from regularity and symmetry.
\end{proof}

\begin{remark}{\rm
Under the assumptions of Corollary \ref{reg}, if $Q$ is more regular, say $Q\in C^{k}((0,+\infty))$, then by a bootstrap argument we obtain $u\in C^{k+1}(\{u>0\})$.
}\end{remark}

We are now ready to prove Theorem \ref{alldec}.

\begin{proof}[Proof of Theorem \ref{alldec}]
Let $\tilde u$ be a minimizer of $E$ in $\D$ and let $u$ be the Schwarz symmetrization of $\tilde u$. Arguing as in the proof of Lemma \ref{lem3} we deduce that
\begin{equation}
\label{ancora}
\int |\nabla u|^2 \leq \int |\nabla \tilde u|^2, \quad \int Q(u)=\int Q(\tilde u), \quad\mbox{and}\quad \int u=M,
\end{equation}
that is, $u\in \mathcal D$ and $E[u]\le E[\tilde u]$. Therefore $E[u]=E[\tilde u]$, which together with \eqref{ancora} implies that
\begin{equation}
\label{app4}
\displaystyle \int |\nabla u|^2 = \int | \nabla \tilde u|^2.
\end{equation}
Since $u$ is a radially non-increasing minimizer, by Lemma \ref{lem4} it has compact support. Let $v$ be the non-increasing function defined by $u(x)=v(r)$ with $r=|x-x_0|$ and let $\supp v =[0,\bar r]$. By Proposition \ref{EL} and Corollary \ref{reg}, $v$ is a classical solution of the following one-dimensional problem:
\begin{equation*}
\begin{cases}
\displaystyle-\frac{d^2v}{dr^2} -\frac{N-1}{r}\frac{dv}{dr}+Q'(v)=\lambda & \text{ in } (0,\bar r),
\\[1ex]
\displaystyle v(0)=u(x_0), \  v(\bar r)=0, \  \frac{dv}{dr}(0)=0,
\\[1ex]
\displaystyle \int r^{N-1}v(r)dr=\omega_{N-1}^{-1}M.
\end{cases}
\end{equation*}
We claim that
\begin{equation}
\label{claim-lem6}
\frac{dv}{dr} < 0 \quad\mbox{ in } (0,\bar r).
\end{equation}
We argue by contradiction assuming that there exists a first $r_0\in (0, \bar r)$ such that $\frac{dv}{dr}(r_0)=0$.
Since $\frac{dv}{dr}\le 0$, $\frac{d^2v}{dr^2}(r_0)=0$: but then $Q'(v(r_0))=\lambda$, whence $v(r)\equiv v(r_0)>0$ for all $r\geq r_0$, in contradiction with $v(\bar r)=0$.

\smallskip

In view of \eqref{app4} and \eqref{claim-lem6}, it follows from \cite[Theorem 2.3.3]{Kes} or \cite[Theorem 1]{FV} that $\tilde u=u$ (up to a translation of the center of symmetry). This, together with Proposition \ref{EL} and Corollary \ref{reg}, completes the proof of Theorem \ref{alldec}.
\end{proof}

\section{The one-dimensional case}
\label{s3}

The rest of the manuscript is concerned with the case $N=1$. In the next statement we summarize, for $N=1$, the results contained in Section \ref{s2}:
\begin{corollary}
\label{coro1}
Assume \eqref{H-EL1} and $N=1$. Any minimizer $u$ of $E$ in $\D$ is even with respect to some $x_0\in \R$, which up to a translation we may assume to be zero: $x_0=0$. Moreover $\supp u=[-\rr, \rr]$ for some $\rr \in (0,+\infty)$, $u'<0$  in $(0, \rr)$, $u\in C^2((-\rr,\rr))\cap C([-\rr,\rr])$, and $\lambda >0$ exists such that $u$ is a classical solution to
\begin{equation}
\label{eqmin1}
\begin{cases}
-u''+Q'(u)=\lambda & \text{ in }\{u>0\}=(-\rr,\rr),\\
u'(0)=0, \ u(\pm \rr)=0.
\end{cases}
\end{equation}
\end{corollary}

In the rest of the manuscript, we will always assume (up to a translation) that the symmetry point of a minimizer is located at $x_0=0$. In one space dimension, we will be able to obtain uniqueness (up to a translation) and qualitative properties of minimizers of $E$ in $\D$. The key additional information is a characterization of the eigenvalue $\lambda$, which we now discuss under a mild additional information on the behavior of $Q'(s)$ for $s\ll 1$. We assume that $C>0$ exists such that
\begin{equation}
\label{H-EL}
\mbox{\eqref{H-EL1} holds \quad and \quad  $|sQ'(s)|\leq C Q(s)$ for $s\ll 1$}.
\end{equation}
\begin{remark}{\rm
For $s\ll 1$, \eqref{H-EL}$_2$ essentially rules out exponential growth of $Q(s)$, which is however already excluded by \eqref{H}, as well as too wild oscillations of $Q'(s)$, such as $Q(s)=A s^{1-m}(1+s\sin s^{-2})$ for $s\ll 1$. It is obviously satisfied by all model cases.
}\end{remark}

\subsection{The identification of the eigenvalue $\lambda$}
\label{ss-lambda}

The identification of $\lambda$ is based on three identities:
\begin{lemma}
\label{lem8}
Assume \eqref{H-EL} and $N=1$. Let $u$ be a minimizer of $E$ in $\D$. Then a constant $K\in \R$ exists such that
\begin{equation}
\label{eq1a}
\frac{1}{2}u'^2-Q(u)+\lambda u=K \quad \text{in }\{u>0\}.
\end{equation}
In addition
\begin{equation}\label{eq1b}
\int_{\{u>0\}} [u'^2+Q'(u)u-\lambda u]=0
\end{equation}
and
\begin{equation}
\label{eq1c}
\frac{3}{2}\int_{\{u>0\}} |u'|^2 -\int_{\{u>0\}} Q(u) +\int_{\{u>0\}} Q'(u)u = 0.
\end{equation}
\end{lemma}

\begin{remark}{\rm
The function $uQ'(u)\chi_{\{u>0\}}$ in \eqref{eq1b} and \eqref{eq1c} belongs to $L^1(\R)$ since $Q(u)\in L^1(\R)$ (cf. Remark \ref{|Qu|}) and \eqref{H-EL} holds.
}\end{remark}

\begin{proof}
In order to prove \eqref{eq1a}, it suffices to multiply the equation in \eqref{eqmin1} by $-u'$:
$$
\frac{1}{2}[u'^2]'-Q'(u)u'+\lambda u'=\left[\frac{1}{2}u'^2-Q(u)+\lambda u\right]'=0 \quad \text{in }\{u>0\}=(-\rr,\rr),
$$
hence \eqref{eq1a} holds. We now prove \eqref{eq1b}. As $x\to \rr^-$, we have
$$
(u')^2 \stackrel{\eqref{eq1a}}= 2\left(Q(u)-\lambda u+K\right) \sim 2Q(u), \quad\mbox{that is,}\quad  u'\sim -\sqrt{2Q(u)}\stackrel{\eqref{H}}\sim -\sqrt{2A} u^{\frac{1-m}{2}}.
$$
Therefore
\begin{equation}\label{uu'}
uu'\sim -\sqrt{2A} u^{\frac{3-m}{2}}\to 0 \quad\mbox{as $x\to \rr^-$.}
\end{equation}
Multiplying \eqref{eqmin1} by $u$ and integrating over $(-\rr,\rr)$, we have
$$
\int_{(-\rr,\rr)}[-u''u+Q'(u)u-\lambda u]=0.
$$
Integrating by parts the first term and using \eqref{uu'}, we obtain \eqref{eq1b}.

\smallskip

We now prove \eqref{eq1c}. For $\alpha>0$, we consider the mass-preserving rescaling $u_\alpha(\hat x)=\alpha u(\alpha \hat x)\in \mathcal D$. Performing the change of variable $x=\alpha \hat x$, we obtain
$$
E[u_\alpha]= \int \frac12 \left|\frac{d u_\alpha(\hat x)}{d\hat x}\right|^2 d\hat x + \int Q(u_\alpha(\hat x)) d\hat x = \alpha^{3}\int \frac12 |u'|^2 dx + \alpha^{-1}\int Q(\alpha u)d x.
$$
We show that $E[u_\alpha]$ is differentiable with respect to $\alpha$ and
\begin{equation}
\label{Ealp1}
\frac{d}{d\alpha}E[u_\alpha]=\frac{3}{2}\alpha^{2}\int |u'|^2 -\alpha^{-2}\int Q(\alpha u) +\alpha^{-1}\int Q'(\alpha u)u \qquad\mbox{for all $\alpha>0$.}
\end{equation}
The only non-trivial limit is
\begin{equation}\label{apnea21}
\lim_{t\to 0} \int \frac{Q((\alpha + t)u)-Q(\alpha u)}{t} = \int Q'(\alpha u)u,
\end{equation}
for which we apply Lebesgue theorem. Fix $\alpha>0$. First of all, it follows from \eqref{H-EL1} that $\frac{Q((\alpha + t)u)-Q(\alpha u)}{t}\to Q'(\alpha u)u$ pointwise in $(-\rr,\rr)$ as $t\to 0$. For the $L^1$-bound, we write
\begin{eqnarray}\label{pb}
\left|\frac{Q((\alpha+ t)u)-Q(\alpha u)}{t}\right| &\le & \frac{1}{|t|} \int_0^{|t|} |Q'((\alpha+\tau) u)| u d\tau.
\end{eqnarray}
Hereafter in the proof, $C$ denotes a generic constant which may depend on $\alpha$, but not on $t$ and $\tau$.  Take $|t|<\frac{\alpha}{2}$, so that $\frac{\alpha}{2}< \alpha+\tau< \frac{3\alpha}{2}$ for $|\tau|<|t|$. In view of \eqref{H} and \eqref{H-EL}, $\delta>0$ exists such that
\begin{equation}\label{gh}
 \tfrac{A}{2}s^{1-m}\le Q(s)\le 2A s^{1-m} \quad\mbox{and}\quad s|Q'(s)|\le C Q(s) \quad\mbox{for all $s\in (0,\delta)$}.
\end{equation}
Using \eqref{gh}$_1$ twice, we see that
\begin{equation}\label{D1}
Q((\alpha+\tau)u)\le C ((\alpha+\tau)u)^{1-m} \le C u^{1-m} \le C Q(u) \quad\mbox{if $u<\delta_0:=\min\{\frac{2\delta}{3\alpha},\delta\}$.}
\end{equation}
Therefore
$$
|Q'((\alpha+\tau)u)|u
\left\{\begin{array}{ll}
= \frac{|Q'((\alpha+\tau)u)|}{\alpha+\tau}(\alpha+\tau)u \stackrel{\eqref{gh}_2,\eqref{D1}}\le C Q(u) & \mbox{if $u <\delta_0$}
\\[2ex]
\le u \|Q'\|_{L^\infty([\frac{\alpha \delta_0}{2},\frac32 \alpha u(0)])} & \mbox{if $u\ge \delta_0$}
\end{array}\right.
$$
and the $L^1$-bound follows from \eqref{pb} since $Q(u)\in L^1(\R)$ (cf. Remark \ref{|Qu|}). Thus \eqref{apnea21}, whence \eqref{Ealp1}, hold. Since $u=u_1$ is a minimizer of $E$ in $\D$, it follows that $\frac{d}{d\alpha}E[u_\alpha]|_{\alpha=1}=0$: hence \eqref{Ealp1} coincides with \eqref{eq1c}.
\end{proof}

Now we are ready to characterize $\lambda$.
\begin{theorem}
\label{th-lambda}
Assume \eqref{H-EL} and $N=1$. Let $u$ be a minimizer of $E$ in $\D$. Then $K=0$ in \eqref{eq1a}, that is
\begin{equation}
\label{app}
\tfrac{1}{2}u'^2-Q(u)+\lambda u=0 \quad \text{in }\{u>0\},
\end{equation}
and
\begin{equation}
\label{lam1}
\lambda=\tfrac{Q(u(0))}{u(0)}.
\end{equation}
\end{theorem}

\begin{proof}
By Lemma \ref{lem8}, $u$ satisfies \eqref{eq1a}, \eqref{eq1b} and \eqref{eq1c}. Integrating \eqref{eq1a} over $\{u>0\}$ and adding \eqref{eq1b}, we obtain
\begin{equation}
\label{app2}
\frac{3}{2}\int_{\{u>0\}}|u'|^2 -\int_{\{u>0\}}Q(u) +\int_{\{u>0\}}Q'(u)u = \int_{\{u>0\}}K.
\end{equation}
Subtracting \eqref{app2} from \eqref{eq1c}, we have
$$
\int_{\{u>0\}}K=0,
$$
hence $K=0$. Evaluating \eqref{app} in $x=0$ and recalling that $u'(0)=0$, we deduce \eqref{lam1}.
\end{proof}

\subsection{Admissible maximal heights and the first integral}

In view of \eqref{lam1}, it is convenient to introduce the function
\begin{equation}
\label{G1}
R(s):=\frac{Q(s)}{s}\stackrel{\eqref{H-EL1}}\in C^1((0,+\infty)), \quad R'(s)=\frac{1}{s}\left(Q'(s)-R(s)\right), \quad s>0.
\end{equation}
In view of \eqref{eqmin1} and \eqref{lam1}, any minimizer of $E$ in $\mathcal D$ is a solution to
\begin{equation}
\label{eqmin1bis}\tag{$P_{u_0}$}
 \begin{cases}
\displaystyle -u''+Q'(u)=R(u_0) & \quad \mbox{in }\{u>0\},
\\
u(0)=u_0>0,\ u'(0)=0
\end{cases}
\end{equation}
whose solutions we now discuss. First of all, any solution to \eqref{eqmin1bis} is even and
\begin{equation}\label{u''0}
u''(0) \stackrel{\eqref{eqmin1bis}}= Q'(u_0)-R(u_0) \stackrel{\eqref{G1}}= u_0R'(u_0).
\end{equation}
In addition, multiplying \eqref{eqmin1bis} by $u'$ and integrating from $x=0$, we obtain
\begin{equation}\label{appR}
\tfrac12 (u')^2=Q(u)-R(u_0) u \stackrel{\eqref{G1}}= u(R(u)-R(u_0)) \quad\text{in }\{u>0\}.
\end{equation}

In the next lemma we give a necessary and sufficient condition on $u_0$ for a solution of \eqref{eqmin1bis} to have compact support ($\supp u=[-\rr,\rr]$) and negative derivative in $(0,\rr)$. This will identify an {\it admissible set} $\A$ to which the maximal height of minimizers must belong. We also list a few properties of such solutions, which will be used in the sequel.

\captionsetup{width=0.9\linewidth}
{\centering{
\raisebox{\dimexpr \topskip-\height}{
\includegraphics[width=0.4\textwidth]{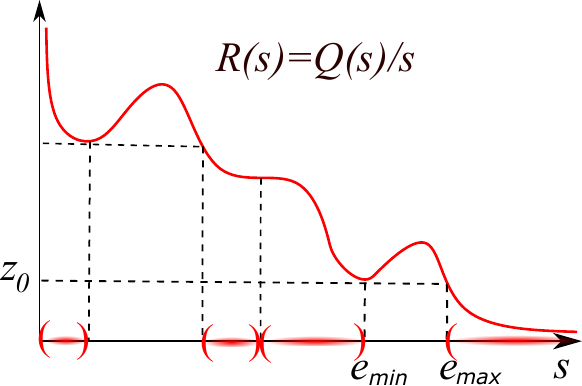}\quad \includegraphics[width=0.4\textwidth]{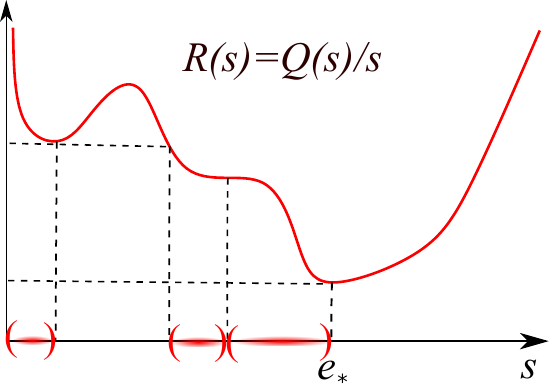}}
  \captionof{figure}{{\footnotesize The set $\A$ in Lemma \ref{fin1} in two cases:  $\e_*=+\infty$ and $z_0<+\infty$ (cf. \eqref{def-s*} and \eqref{def-z0}) and $\e_*<+\infty$. The numbers $\e_*$, $\emin$ and $\emax$ are also visualized (cf. \eqref{def-s*} and \eqref{def-et}).}} \label{fig-A}
  }}

\begin{lemma}
\label{fin1}
Assume \eqref{H-EL} and $N=1$. Let $u$ be a solution of \eqref{eqmin1bis}. Then the following are equivalent:
\begin{itemize}
\item[$(a)$] $\supp u=[-\rr,\rr]$ for some $\rr\in (0,+\infty)$ and $u'<0$ in $(0,\rr)$;
\item[$(b)$] $u_0\in \A$, where $\A\subseteq(0,+\infty)$ is the open set defined by
\begin{equation}
\label{def-setR}
\A:=\left\{s\in (0,+\infty): \ \mbox{$R'(s)<0$ and $R(t)>R(s)$ for all $t<s$}\right\}.
\end{equation}
\end{itemize}
In particular, $\max u\in\A$ for any minimizer $u$ of $E$ in $\mathcal D$. If $(a)$ or $(b)$ hold, then
\begin{equation}\label{as7}
u'=-\sqrt{2u(R(u)-R(u_0))}\quad\mbox{in $[0,\rr)$},
\end{equation}
\begin{equation}\label{as10}
u_0-u(x) = \int_0^x \sqrt{2u(R(u)-R(u_0))} \quad \mbox{for all $x\in [0,\rr)$,}
\end{equation}
and for all $x\in [0,\rr)$ it holds that
\begin{equation}
\label{as4}
Z(u(x))=x, \ \mbox{ where }\ Z(u):=\tfrac{1}{\sqrt 2}\int_u^{u_0} \left(s(R(s)-R(u_0))\right)^{-\frac{1}{2}} ds \quad \text{for }u\in [0,u_0].
\end{equation}
\end{lemma}
Note that $Z$ is well defined since $R'(u_0)\ne 0$.

\begin{proof}
{\sl Proof of $(a) \Longrightarrow (b)$.}
Since $u$ is strictly decreasing in $(0,\rr)$, $R'(u_0)\le 0$ (by \eqref{u''0}). If $R'(u_0)=0$ we would have $u\equiv u_0$, in contradiction with $u'<0$. Therefore $R'(u_0)<0$. In view of \eqref{appR}, $R(u(x))>R(u_0)$ for any $x\in (0,\rr)$. Since $u$ is continuous and $u(\rr)=0$, we deduce that $R(t)>R(u_0)$ for every $t\in (0,u_0)$, hence $u_0\in \A$.

\smallskip

{\sl Proof of $(b) \Longrightarrow (a)$.}
We now assume that $u_0\in \A$, which implies $R'(u_0)<0$. Because of \eqref{u''0}, $u$ is strictly decreasing in a right neighborhood of $x=0$. Assume by contradiction that $x>0$ exists such that $u'(x)=0$ and $u'<0$ in $(0,x)$. In particular, $u(x)<u_0$ and, by \eqref{appR}, $R(u(x))=R(u_0)$. This contradicts the definition of $\A$. Therefore $u'<0$ as long as $u$ is defined, and \eqref{as7} follows from \eqref{appR}. Integrating \eqref{as7} with respect to $x$ we obtain \eqref{as10}. Integrating \eqref{as7} with respect to $u$, we obtain \eqref{as4}. In particular,
$$
\sqrt 2 x=\int_{u(x)}^{u_0} \left(s(R(s)-R(u_0))\right)^{-\frac{1}{2}} ds\le  \int_{0}^{u_0} \left(s(R(s)-R(u_0))\right)^{-\frac{1}{2}} ds.
$$
Since $R'(u_0)\ne 0$ and $s(R(s)-R(u_0))\sim Q(s)\to +\infty$ as $s\to 0^+$, the right-hand side is finite: therefore $u$ has compact support and the proof is complete.
\end{proof}

\subsection{Asymptotics near the interface}

Now we investigate the asymptotic behaviour near $\partial\{u>0\}=\{-\rr,\rr\}$ of solutions to \eqref{eqmin1bis} with $u_0\in \A$; in particular, for minimizers of $E$. Since any solution of \eqref{eqmin1bis} is even, it suffices to study the behaviour of $u$ as $x\to \rr^-$.

\begin{theorem}
\label{asbe1}
Assume \eqref{H-EL} and $N=1$. Let $u$ be a solution of \eqref{eqmin1bis} with $u_0\in \A$. Then
\begin{eqnarray}
\label{asu}
u(x) & \sim & \left(\tfrac{A(m+1)^2}{2}\right)^{\frac{1}{m+1}}(\rr- x)^{\frac{2}{m+1}} \quad \text{as } x \to \rr^-\, ,
\\
\label{asu'}
u'(x) & \sim & -\tfrac{2}{m+1}\left(\tfrac{A(m+1)^2}{2}\right)^{\frac{1}{m+1}}(\rr- x)^{\frac{2}{m+1}-1}\quad \text{as } x \to \rr^-\,.
\end{eqnarray}
In particular, $u\in H^1_0((-\rr, \rr))$.
\end{theorem}

\begin{proof}
By Lemma \ref{fin1}, \eqref{as4} holds. Since both $Z$ and $u$ are strictly decreasing, whence invertible, $x\to \rr^-$ if and only if $u\to 0^+$. Therefore
\begin{eqnarray}
\label{as5}
\lim_{x\to \rr^-}\frac{\rr - x}{u^\alpha(x)} & \stackrel{\eqref{as4}}= & \lim_{u\to 0^+} \frac{\rr - Z(u)}{u^\alpha} = \lim_{u\to 0^+} \frac{\frac{1}{\sqrt{2}}\left(Q(u)-R(u_0)u\right)^{-1/2}}{\alpha u^{\alpha-1}} \nonumber \\
& \stackrel{\eqref{H}}= &
\lim_{u\to 0^+} \frac{1}{\sqrt{2A}\alpha u^{\alpha-1+\tfrac{1-m}{2}}},
\end{eqnarray}
where in the second equality of \eqref{as5} we used L'H\^opital's rule. Choosing $\alpha=\frac{m+1}{2}$ we obtain \eqref{asu}. Consequently, we obtain \eqref{asu'}:
\begin{eqnarray*}
u'(x) &\stackrel{\eqref{as7}} \sim & -\sqrt{2Q(u)} \stackrel{\eqref{H}}\sim -\sqrt{2A} u(x)^{\frac{1-m}{2}} \nonumber \stackrel{\eqref{asu}} \sim  -\sqrt{2A}\left(\sqrt{\tfrac{A(m+1)^2}{2}}\right)^{\tfrac{1-m}{m+1}}(\rr- x)^{\frac{1-m}{m+1}}
\\
&=& -\tfrac{2}{m+1}\left(\tfrac{A(m+1)^2}{2}\right)^{\tfrac{1}{m+1}}(\rr- x)^{\frac{2}{m+1}-1}
\quad \text{as } x \to \rr^-\,.
\end{eqnarray*}
Finally, since $u$ is even, $u\in C^2((-\rr,\rr))$, and \eqref{asu'} holds, $C>0$ exists such that
$$
\int_{-\rr}^{\rr} (u')^2\leq C \int_{0}^{\rr} (\rr- x)^{\frac{2(1-m)}{m+1}}\stackrel{m<3}<+\infty.
$$
\end{proof}

\subsection{Bounds on the maximal height}

To avoid pathological situations, we assume in what follows that $\delta>0$ exists such that
\begin{equation}\label{H-R}
\mbox{$R'\ne 0$ in $(0,\delta)\cup (\delta^{-1},+\infty)$.}
\end{equation}
In particular, the limit of $R$ as $s\to +\infty$ exists, and it follows from $\inf Q>-\infty$ that
\begin{equation}
\label{Rlim}
\lim_{s\to +\infty} R(s) = \lim_{s\to +\infty} \frac{Q(s)}{s} \geq 0.
\end{equation}
We let $\e_*$ be the smallest among the absolute minimum points of $R$, provided they exist:
\begin{equation}
\label{def-s*}
\e_*:=\left\{\begin{array}{ll}
+\infty & \mbox{if $\not\exists \, \min R$},
\\[1ex]
\min R^{-1}(\min R) & \mbox{otherwise}.
\end{array}\right.
\end{equation}

\begin{remark}
\label{rem:R2}{\rm{
If $\e_*=+\infty$, then $\inf\limits_{(0,+\infty)}R = \lim\limits_{s\to +\infty} R(s)$. In particular, in view of \eqref{Rlim}, $Q>0$ in $(0,+\infty)$. Note that in this case $Q(s)\lesssim s$ for $s\gg 1$.  Moreover, if $R(+\infty)=0$ the converse holds true: $Q>0$ in $(0,+\infty)$ implies $\e_*=+\infty$.
}}\end{remark}

\begin{remark}\label{rem:R1}{\rm
In view of \eqref{Rlim}, $Q\not > 0$ in $(0,+\infty)$ --that is, $R\not > 0$ in $(0,+\infty)$-- implies that $\e_*<+\infty$.
}\end{remark}

If $\e_*=+\infty$, we introduce the smallest level at which $R$ ceases to be injective:
\begin{equation}
\label{def-z0}
z_0:=\sup\{z>0 : \#\{R^{-1}(z')\} \leq 1 \text{ for every }z'\in (0,z)\}\quad\mbox{if $\e_*=+\infty$.}
\end{equation}
Thanks to \eqref{H-R} and Remark \ref{rem:R2}, the set is non-empty and $z_0\in (0,+\infty]$.  If $z_0=+\infty$, then $R$ is strictly decreasing in $(0,+\infty)$. If instead $z_0<+\infty$, we define (cf. Fig. \ref{fig-A})
\begin{equation}
\label{def-et}
\emax:=\max\{s>0 : \ R(s)= z_0 \}
\quad \text{and}\quad \emin:=\min\{s>0 :\ R(s)=z_0\}.
\end{equation}
It follows from \eqref{H-R} that $0<\emin<\emax<+\infty$.

\smallskip

In the following lemma we give a characterization of $\e_*$ in terms of $\A$ and some restrictions on the maximal height.
\begin{lemma}
\label{bnd-inf2}
Assume $N=1$, \eqref{H-EL} and \eqref{H-R}. Then $\e_*=\sup \A$, hence $\A\subseteq (0,\e_*)$. In addition, if $\e_*=+\infty$ and $z_0<+\infty$ (cf. \eqref{def-z0}), then $\A \subseteq (0,\emin) \cup (\emax, +\infty)$ (cf. \eqref{def-et}).
\end{lemma}

\begin{proof}
It is obvious that $\sup \A\leq \e_*$. If by contradiction $\sup \A<\e_*$, then for all $s\in (\sup \A,\e_*)$ we would have either $R'(s)\ge 0$ or $R(t)\le R(s)$ for some $t<s$, in contradiction with the definition of $\e_*$. Therefore $\sup \A=\e_*$.

If $\e_*=+\infty$ and $z_0<+\infty$, let $u_0\in \A$. By its definition, $R(\emin)$ is a global minimum for $R$ in $(0,\emax]$, which implies that $R'(\emin)=0$ and $R(s)\geq R(\emin)$ for $s\in [\emin,\emax]$. The former implies that $u_0\ne\emin$ and the latter implies that $u_0\notin (\emin,\emax]$.
\end{proof}

\subsection{Uniqueness}

We will now prove comparison  and uniqueness results for minimizers of $E$ in $\mathcal D_M$ under the following additional assumption on $Q$:
\begin{equation}
\label{H3}
Q'(s) \text{ is non-decreasing for }s\in(0,\e_*).
\end{equation}
In Section \ref{ss-non} we will provide examples of non-uniqueness for a large class of potentials which do not satisfy \eqref{H3}.

\begin{lemma}
\label{lem11}
Assume $N=1$, \eqref{H-EL}, \eqref{H-R}, and \eqref{H3}. Let $u_1$ and $u_2$ be minimizers of $E$ in $\D_{M_1}$, resp. $\D_{M_2}$, both symmetric with respect to $x_0=0$. Then the following are equivalent:
\begin{itemize}
\item[$(i)$] $M_1<M_2$;
\item[$(ii)$] $u_1(0)=:u_{01}<u_{02}:=u_2(0)$;
\item[$(iii)$] $u_1<u_2$ in $\supp u_1$.
\end{itemize}
\end{lemma}

\begin{remark}{\rm
\label{rem:nonbaro}
If $\e_*=+\infty$, \eqref{H3} can be attained only if $z_0=+\infty$.
}\end{remark}

\begin{proof}
The proof of $(iii)\Rightarrow (i)$ is obvious.

\smallskip

{\it Proof of $(ii)\Rightarrow (iii)$.} It follows from Lemma \ref{fin1} that $u_{0i}\in \A$. Since $u_i$ are strictly decreasing, it follows from Lemma \ref{bnd-inf2} that $u_1,u_2< \e_*$. Subtracting the corresponding equations, we obtain
\begin{equation*}
-(u_1-u_2)''=R(u_{01})-R(u_{02})-(Q'(u_1)-Q'(u_2)) \quad \text{in }\{u_1>0\}\cap\{u_2>0\}.
\end{equation*}
Since $u_{02}\in \A$, $R(u_{02})<R(t)$ for all $t\in (0,u_{02})$: in particular, $R(u_{02})<R(u_{01})$. As long as $u_1\le u_2$, by \eqref{H3}, we have $Q'(u_1)\le Q'(u_2)$. Hence $(u_1-u_2)''< 0$ as long as $u_1\le u_2$. Integrating twice using $u_1'(0)=u_2'(0)=0$, we deduce that $(u_1-u_2)\leq (u_1-u_2)(0)<0$ as long as $u_1\le u_2$, whence $u_1<u_2$ in $\supp u_1$.

\smallskip

{\it Proof of $(i)\Rightarrow (ii)$.} First we note that $u_{01}\ne u_{02}$: otherwise, $u_1$ and $u_2$ would solve the same equation \eqref{eqmin1bis}, whence $u_1=u_2$ by Picard-Lindel\"of theorem, in contradiction with $M_1\neq M_2$. On the other hand, using $(ii)\Rightarrow (iii)$, $u_{01}>u_{02}$ implies $u_1>u_2$ in $\supp u_2$, in contradiction with $M_1<M_2$.
\end{proof}

As a by-product of Lemma \ref{lem11} we obtain the uniqueness result:
\begin{theorem}
\label{uniq1d}
Assume $N=1$, \eqref{H-EL}, \eqref{H-R}, and \eqref{H3}. Then for any $M>0$ there exists at most one minimizer of $E$ in $\D_M$ (up to translation).
\end{theorem}

\begin{proof}
Let $u_1$ and $u_2$ be two minimizers of $E$ in $\D_M$ (both of them symmetric with respect to $x_0=0$). Since $u_1$ and $u_2$ have the same mass, it follows from Lemma \ref{lem11} that $u_1(0)=u_2(0)$. This implies $u_1=u_2$ by Picard-Lindel\"of theorem.
\end{proof}

\section{Pancakes versus droplets}
\label{s4}

We assume throughout the section that
\begin{equation}\label{Q'0}
\mbox{$N=1$, \eqref{H-EL}, and \eqref{H-R} hold.}
\end{equation}
For $s\in \A$, we consider the solution $u_s$ to $(P_s)$ with $\supp u_s=[-\rr_s,\rr_s]$ (cf. Lemma \ref{fin1}) and we define
\begin{equation}\label{def-mu}
\mu:\A\to (0,+\infty), \quad \mu(s):=2\int_0^{\rr_s} u_s.
\end{equation}
We claim that
\begin{equation}
\label{M-cont}
\mu\in C(\A).
\end{equation}
Let $u_0\in \A$. Since $\A$ is open, $s\in \A$ in a neighbourhood of $u_0$. By Lemma \ref{fin1}, we know that $u_{s}$ is even, has compact support, say $[-\rr_{s},\rr_{s}]$, and is strictly decreasing in $(0,\rr_{s})$. By Lemma \ref{lemtec}, $u_s\to u_{u_0}$ in $C^2_{loc}((-\rr_{u_0},\rr_{u_0}))$ and $\rr_s\to \rr_{u_0}$ as $s\to u_0$, hence a.e. in $(-\rr_{u_0},\rr_{u_0})$. Therefore, by dominated convergence, $\mu(s)\to \mu(u_0)$.

\smallskip

We will also need the following a-priori estimate, in the spirit of Theorem \ref{asbe1}.

\begin{lemma}\label{lem:lb}
Assume \eqref{Q'0}. A constant $C>0$ exists such that for all $K\ge 1$ there exists $\delta_K>0$ such that
\begin{equation}
\label{lb}
u(x)\ge C(\rr-x)^\frac{2}{m+1} \quad \mbox{for all $x\in (\rr-\delta_K,\rr)$}
\end{equation}
for all $u_0\in \A\cap [K^{-1},K]$ and for all $u_0\in \A\cap [K^{-1},+\infty)$ if $R(+\infty)<+\infty$, where $u$ is the solution to $(P_{u_0})$ and $[-\rr,\rr]=\supp u$.
\end{lemma}

\begin{proof}
For $u_0$ as in the statement, the properties of $Q$ imply that $\eps_K<K^{-1}$ exist such that
\begin{equation}\label{contedeche}
Q(s)-R(u_0)s\le 2As^{1-m} \quad\text{ for all } s\in (0,\eps_K),
\end{equation}
\begin{equation}
\label{conte4bis}
Q(s)\geq \tfrac{A}{2} s^{1-m} \quad\text{ for all } s\in (0,\eps_K),
\end{equation}
\begin{equation}
\label{conte5}
R(s)\geq 2|R(u_{0})| \quad \mbox{for all } s\in (0,\eps_K).
\end{equation}
We preliminarily work out an upper bound on $u$. It follows from \eqref{as7} that
$$
-u'=\sqrt{2(Q(u)-R(u_{0})u)} \stackrel{\eqref{contedeche}}\le C u^\frac{1-m}{2} \quad\mbox{as long as $u<\eps_K$.}
$$
Integrating it in $(x,\rr)$, we see that
$$
u(x)\le C(\rr-x)^\frac{2}{m+1}\quad\mbox{as long as $u<\eps_K$}.
$$
Choosing $\delta_K$ such that $C\delta_K^\frac{2}{m+1}=\eps_K/2$, we conclude that
$u(x)<\eps_K/2$ for all $x\in (\rr-\delta_K,\rr)$. Now we can work out the lower bound:
\begin{equation}
\label{conte7}
-u'\stackrel{\eqref{as7}}= \sqrt{2u(R(u)-R(u_{0}))} \stackrel{\eqref{conte5}}\geq \sqrt{Q(u)} \stackrel{\eqref{conte4bis}}\geq C u^\frac{1-m}{2} \quad\mbox{in $(\rr-\delta_K,\rr)$}.
\end{equation}
Integrating \eqref{conte7} in $(x,\rr)$, $x\in(\rr-\delta_K,\rr)$, we deduce
\begin{equation*}
C (\rr-x) \stackrel{\eqref{conte7}}\le -\int_{x}^{\rr} u^\frac{m-1}{2}u' = \tfrac{2}{m+1} u^{\frac{m+1}{2}}\quad\mbox{for all $x\in (\rr-\delta_K,\rr)$},
\end{equation*}
whence \eqref{lb}.
\end{proof}

\subsection{The droplet case}\label{ss-par}

In this subsection we prove the following result.
\begin{lemma}[the droplet case]\label{thm:par}
Assume \eqref{Q'0}, $\e_*=+\infty$,
\begin{equation}\label{H-add}
\mbox{$Q(s)\to -S\in (0,+\infty)$,\ \ and \ \ $sQ'(s)\to 0$ as $s\to +\infty$.}
\end{equation}
If $\A \ni u_{0k}\to +\infty$ as $k\to +\infty$ for a sequence, then $\mu(u_{0k})\to +\infty$,
\begin{equation}\label{par23}
u_{0k}^4 \sim \frac{9|S|}{32} \mu^2(u_{0k}),\quad \rr_k^4\sim \frac{9}{8|S|} \mu^2(u_{0k}) \quad\mbox{as $k\to +\infty$}
\end{equation}
and
\begin{equation}\label{par24}
w_k(y)= u_{0k}^{-1} u_k(\rr_ky) \to 1-y^2\quad\mbox{in $C^2_{loc}((-1,1))$}\quad\mbox{as $k\to +\infty$},
\end{equation}
where $u_k$ is the solution to $(P_{u_{0k}})$.
\end{lemma}

\begin{proof}
We recall that $\e_*=+\infty$ implies $Q>0$ (thus $R>0$) in $(0,+\infty)$ (Remark \ref{rem:R2}), and we note for later reference that only $-S\ge 0$ is used in Step 1 of this proof.

\smallskip

{\it Step 1.} We show that $C\ge 1$ exists such that
\begin{equation}\label{par-s2}
\rr_k\to +\infty \quad\mbox{and}\quad \frac{u_{0k}}{\rr_k}\le C \quad\mbox{as $k\to +\infty$}.
\end{equation}
Let $\overline s$ be such that $Q(s)\le |S|+1$ in $(\overline s, +\infty)$, and let $x_k=u_k^{-1}(\overline s)$ (which is well defined for $k\gg 1$ since $u_{0k}\to +\infty$). Since $R>0$ we obtain \eqref{par-s2}:
$$
+\infty \leftarrow u_{0k}-\overline s \stackrel{\eqref{as10}}= \int_0^{x_k} \sqrt{2(Q(u_k)-R(u_{0k})u_k)} \le  \sqrt{2( |S|+1)} x_k\le \sqrt{2( |S|+1)} \rr_k.
$$

{\it Step 2.} Note that $-S>0$ and $Q>0$ in $(0,+\infty)$ imply that $\inf\limits_{(0,+\infty)} Q>0$.  Let $w_k(y)=u_{0k}^{-1} u_k(\rr_k y)$.
It follows from \eqref{as7} that
\begin{equation}\label{par12}
-w_k'=\sqrt 2 \frac{\rr_k}{u_{0k}}\sqrt{Q(u_{0k}w_k)-Q(u_{0k})w_k} \qquad \text{in } [0,1),
\end{equation}
that is,
\begin{equation}\label{par30}
1=\sqrt 2 \frac{\rr_k}{u_{0k}}\int_0^1\sqrt{Q(u_{0k}w_k(y))-Q(u_{0k})w_k(y)}dy.
\end{equation}
We claim that a constant $C\ge 1$ exists such that
\begin{equation}\label{par20}
\frac{u_{0k}}{\rr_k}\ge C^{-1}.
\end{equation}
Assume by contradiction that $\frac{u_{0k}}{\rr_k}\to 0$ as $k\to +\infty$ for a subsequence (not relabeled). Then \eqref{par30} yields
\begin{equation}\label{par13}
Q(u_{0k}w_k(y))-Q(u_{0k})w_k(y) \to 0 \quad\mbox{for a.e. $y\in (0,1)$ as $k\to +\infty$.}
\end{equation}
Take one of such $y$'s and any convergent subsequence $w_k(y)\to c\in [0,1]$. If $c=0$, we would have $Q(u_{0k})w_k(y)\to |S|\cdot 0$, whence by \eqref{par13} $Q(u_{0k}w_k(y))\to 0$, as $k\to +\infty$, which is impossible since $\inf\limits_{(0,+\infty)} Q>0$. Therefore $c>0$, and \eqref{par13} implies $|S|-|S|c=0$, i.e. $c=1$. Thus $w_k(y)\to 1$ for a.e. $y\in (0,1)$. Now, take $0<b<1$. Integrating \eqref{par12} in $(0,b)$, we obtain
\begin{equation}\label{par21}
\frac{1}{\sqrt 2 }\int_{w_k(b)}^{1}\frac{dw}{\sqrt{Q(u_{0k}w)-Q(u_{0k})w}} =\frac{\rr_k}{u_{0k}} b\to +\infty \quad\mbox{as $k\to +\infty$.}
\end{equation}
On the other hand, by Lagrange theorem, for all $w\in [w_k(b),1]$ we have
$$
Q(u_{0k}w)-Q(u_{0k})w = \left(Q(u_{0k}) - u_{0k}Q'(u_{0k}\eta_{k,w})\right)(1-w) \quad\mbox{for some $\eta_{k,w}\in [w_k(b),1]$.}
$$
Noting that $\eta_{k,w}\to 1$ and, by \eqref{H-add}, that $u_{0k}Q'(u_{0k}\eta_{k,w})=  \frac1{\eta_{k,w}}u_{0k}\eta_{k,w} Q'(u_{0k}\eta_{k,w})\to 0$ uniformly with respect to $w\in [w_k(b),1]$ as $k\to +\infty$, since $Q(s)\to |S|>0$ as $s\to +\infty$ we obtain
\begin{equation}
\label{par30bis}
Q(u_{0k}w)-Q(u_{0k})w \ge \frac{|S|}{2}(1-w) \quad\mbox{for all $w\in [w_k(b),1]$ and all $k\gg 1$,}
\end{equation}
hence
$$
\frac{1}{\sqrt 2 }\int_{w_k(b)}^1\frac{dw}{\sqrt{Q(u_{0k}w)-Q(u_{0k})w}} \le  \frac{1}{\sqrt{|S|}}\int_{w_k(b)}^{1}\frac{d w}{\sqrt{1-w}} <+\infty,
$$
in contradiction with \eqref{par21}. Hence \eqref{par20} holds.

\smallskip

{\it Step 3.} We now prove that, as $k\to +\infty$,
\begin{equation}\label{par22}
\left(\frac{u_{0k}}{\rr_k}\right)^2\to \frac{|S|}{2} \quad\mbox{and}\quad w_k(y)\to 1- y^2\ \quad\mbox{in $C^2_{loc}((-1,1))$.}
\end{equation}
Take any subsequence $k\to+\infty$ (not relabeled) such that $\frac{u_{0k}}{\rr_k}\to \beta$. By \eqref{par-s2}  and \eqref{par20}, $\beta\in (0,+\infty)$. Since $w_k$ is strictly decreasing, $y_k:=w_k^{-1}$ is well-defined in $[0,1]$. Integrating \eqref{par12} in $(0,y_k(w))$ with $w>0$ we obtain
\begin{equation}
\label{par25}
y_k(w)=\frac{1}{\sqrt{2}}\frac{u_{0k}}{\rr_k}\int_{w}^1\frac{d\tilde w}{\sqrt{Q(u_{0k}\tilde w)-Q(u_{0k})\tilde w}} \sim \frac{\beta}{\sqrt{2}}\int_{w}^1\frac{d\tilde w}{\sqrt{Q(u_{0k}\tilde w)-Q(u_{0k})\tilde w}}.
\end{equation}
Since $u_{0k}\to +\infty$, by \eqref{H-add} we deduce that $Q(u_{0k}\tilde w)-Q(u_{0k})\tilde w\to |S|(1-\tilde w)$ a.e. in $(w,1)$. Now, repeating the same argument used to obtain \eqref{par30bis} we get
$$
\frac{1}{\sqrt{Q(u_{0k}\tilde w)-Q(u_{0k})\tilde w}}\leq \frac{\sqrt 2}{\sqrt{|S|(1-\tilde w)}}\in L^1((w,1)) \quad\mbox{for all $k\gg 1$}.
$$
Thus we can apply Lebesgue theorem to \eqref{par25}, obtaining
\begin{equation}
\label{par26}
\lim_{k\to+\infty} y_k(w)=y(w):=\frac{\beta}{\sqrt{2}}\int_{w}^1\frac{d\tilde w}{\sqrt{|S|(1-\tilde w)}}=\frac{\sqrt{2}\beta}{\sqrt{|S|}}\sqrt{1- w}\in \left[0,\tfrac{\sqrt{2}\beta}{\sqrt{|S|}}\right)
\end{equation}
for every $w\in (0,1]$. By construction, $\tfrac{\sqrt{2}\beta}{\sqrt{|S|}}\leq 1$. Since $y_k$ and $y$ are strictly decreasing, \eqref{par26} is equivalent to
\begin{equation}\label{par26-eq}
\lim_{k\to+\infty} w_k(y)= w(y):= 1-\frac{|S|}{2\beta^2}y^2 \quad \mbox{for all $y\in \left[0,\tfrac{\sqrt{2}\beta}{\sqrt{|S|}}\right)$}.
\end{equation}
Assume by contradiction that $\tfrac{\sqrt{2}\beta}{\sqrt{|S|}}<1$ and let $y\in\left(\tfrac{\sqrt{2}\beta}{\sqrt{|S|}},1\right)$. Then, by \eqref{par26-eq}, $w_k(y)\to 0$ as $k\to +\infty$, which implies for $k\gg 1$
\begin{equation}
\label{par28}
-w_k'(y)\stackrel{\eqref{par12}}=\sqrt 2 \frac{\rr_k}{u_{0k}}\sqrt{Q(u_{0k}w_k(y))-Q(u_{0k})w_k(y)} \geq \frac{1}{\beta} \inf\limits_{(0,+\infty)} \sqrt Q>0.
\end{equation}
Therefore
$$
w_k(y)=-\int_y^1 w'_k(z)dz \stackrel{\eqref{par28}}\geq \frac{1}{\beta} \inf\limits_{(0,+\infty)} \sqrt Q >0 \quad \text{for }k\gg 1,
$$
a contradiction. Hence $\tfrac{\sqrt{2}\beta}{\sqrt{|S|}}= 1$ and \eqref{par22} follows.

\smallskip

{\it Conclusion.} As $k\to +\infty$, we have
\begin{equation}\label{M-drop}
\frac {\mu(u_{0k})} 2 = \int_0^{\rr_k} u_k dx = u_{0k} \rr_k \int_0^1 w_k dy \sim u_{0k} \rr_k \int_0^1 (1-y^2)d y = \frac23 u_{0k} \rr_k.
\end{equation}
In particular, $\mu(u_{0k})\to +\infty$. Combined with \eqref{par22}, \eqref{M-drop} leads after straightforward computations to \eqref{par23}, and \eqref{par24} follows from \eqref{par22}.
\end{proof}

\begin{lemma}[the droplet's energy]\label{thm:par:en}
Under the assumptions of Lemma \ref{thm:par},
$$
E[u_k]\lesssim \sqrt{\mu(u_{0k})}\quad\mbox{for $k\gg 1$}.
$$
\end{lemma}

\begin{proof}
We let $u=u_{k}$  and $\rr=\rr_{k}$ for notational convenience. Recalling that $Q$, thus $R$, is positive in $(0,+\infty)$ (Remark \ref{rem:R2}), we note that
\begin{equation}\label{E-min}
E[u] = 2\int_0^{\rr} \left(\tfrac12(u')^2+ Q(u)\right) \stackrel{\eqref{appR}}= 2\int_0^{\rr} \left(2 Q(u)-R(u_{0k})u\right) \lesssim \int_0^{\rr} Q(u).
\end{equation}
Also, note that $Q>0$, $-S>0$, and \eqref{H} imply that
\begin{equation}
\label{Eest1}
1+u^{1-m}\lesssim Q(u)\lesssim 1+u^{1-m}\quad\mbox{for all $u>0$}.
\end{equation}
In particular,
\begin{equation}\label{E-min1}
E[u] \stackrel{\eqref{E-min}}\lesssim \int_0^{\rr} Q(u) \stackrel{\eqref{Eest1}}\lesssim \rr + \int_0^{\rr} u^{1-m}.
\end{equation}
Note that \eqref{H-add} implies that $R(+\infty)=0$. Therefore we may apply Lemma \ref{lem:lb} with $K=1$, leading to
\begin{equation}\label{Eest2}
u(x)\gtrsim (\rr-x)^\frac{2}{m+1}\quad\mbox{for $x\in (\rr-\delta_K,\rr)$},
\end{equation}
which by monotonicity implies that
\begin{equation}\label{Eest3}
u(x)\ge u(\rr-\delta_K) \stackrel{\eqref{Eest2}}\sim 1\quad\mbox{for $x\in (0,\rr-\delta_K)$}.
\end{equation}
Therefore, since $\frac{2(1-m)}{m+1}+1=\frac{3-m}{m+1}>0$,
\begin{equation*}
E[u] \stackrel{\eqref{E-min1}}\lesssim \rr + \int_0^{\rr-\delta_K} u^{1-m} +\int_{\rr-\delta_K}^{\rr} u^{1-m} \stackrel{\eqref{Eest2},\eqref{Eest3}}\lesssim \rr+ \int_{\rr-\delta_K}^{\rr} (\rr-x)^\frac{2(1-m)}{m+1} \lesssim \rr,
\end{equation*}
hence the result follows from \eqref{par23}$_2$.
\end{proof}

\subsection{The pancake case}

\begin{lemma}[the pancake case]\label{thm:pk}  Assume \eqref{Q'0} and $\e_*<+\infty$. If
$\A\ni u_{0k}\to \e_* (>0)$ as $k\to +\infty$ for a sequence, then $\mu(u_{0k})\to +\infty$ as $k\to +\infty$ and
\begin{equation}\label{panlim}
\left.
\begin{array}{c}
\dfrac{2}{\mu(u_{0k})} \rr_k\to \dfrac{1}{\e_*} \quad\mbox{and}
\\[2ex]
 v_k(y)=u_k(\rr_k y) \to \e_* \quad\mbox{in $C^2_{loc}((-1,1))$}
\end{array}
\right\}\quad\mbox{as $k\to +\infty$,}
\end{equation}
where $u_k$ is the solution to $(P_{u_{0k}})$.
\end{lemma}

\begin{proof}
Since $\e_* <+\infty$ is a stationary solution of $(P_{e_*})$ (recall \eqref{u''0} and that $R'(\e_*)=0$), by continuous dependence (Theorem 8.40 of \cite{KePe}) $u_k\to \e_*$ in $C^2_{loc}(\R)$ and $\rr_k\to +\infty$ as $k\to +\infty$. Since $u_{0k}\in \A$, Lemma \ref{bnd-inf2} implies that $u_{0k}<\e_*$ for all $k$.
Let $v_k(y)=u_k(\rr_k y)$.
It follows from \eqref{as7} that
\begin{equation*}
-v_k'=\sqrt 2 \rr_k\sqrt{Q(v_k)-R(u_{0k})v_k} \qquad \text{in } [0,1).
\end{equation*}
Integrating, we obtain
\begin{equation*}
\e_*\leftarrow  u_{0k}=\sqrt 2 \rr_k\int_{0}^{1} \sqrt{Q(v_k)-R(u_{0k})v_k}\, dy \quad\mbox{as $k\to+\infty$}.
\end{equation*}
Since $\rr_k\to +\infty$ as $k\to +\infty$,  we deduce that
\begin{equation}
\label{pan11}
0=\lim_{k\to +\infty} \left(Q(v_k)-R(u_{0k})v_k\right) \quad \text{a.e. in } [0,1).
\end{equation}
Take any such $y$, and take any subsequence (not relabeled) such that $v_k(y)\to \overline v\in [0,\e_*]$ as $k\to +\infty$ (recall that $v_k(y)\leq u_{0k}\to \e_*$). If $\overline v=0$ we would have $0=+\infty$ in \eqref{pan11}. Therefore $\overline v>0$ and $Q(\overline v)=R(\e_*)\overline v$, that is, $R(\e_*)= R(\overline v)$. By the definition of $\e_*$ (cf. \eqref{def-s*}), $R(s)>R(\e_*)$ for $s\in(0,\e_*)$: therefore $\overline v=\e_*$. The arbitrariness of the subsequence and of $y$ (in this order) implies that $v_k(y)\to \e_*$ a.e. in $(0,1)$. By dominated convergence we obtain \eqref{panlim}$_1$:
$$
\frac{\mu(u_{0k})}{\rr_k}= \frac{1}{\rr_k}\int_{-\rr_k}^{\rr_k}u_k(x)dx = \int_{-1}^{1}v_k(y)dy \to \int_{-1}^{1} \e_* =2\e_*.
$$
Finally, fix $\overline y<1$ such that $v_k(\overline y)$ converges. We have
$$
0 =\e_*-\e_* \leftarrow u_{0k}-v_k(\overline y)= \omega_k:= \sqrt 2 \rr_k\int_{0}^{\overline y} \sqrt{Q(v_k)-R(u_{0k})v_k}\, dy
$$
with $\omega_k\to 0$ as $k\to +\infty$. Hence $|v_k(y_1)-v_k(y_2)| \le \omega_k$ for all $y_1,y_2\in [0,\overline y]$. Therefore $v_k\to \e_*$ in $C_{loc}([0,1))$, and the conclusion follows from standard ODE theory.
\end{proof}

\begin{lemma}[the pancake's energy]\label{pan:en}
Under the assumptions of Lemma \ref{thm:pk},
\begin{equation*}
\frac{E[u_k]}{\mu(u_{0k})} \to R(e_*)  \quad\mbox{as $k\to +\infty$}.
\end{equation*}
\end{lemma}

\begin{proof}
With $v_k(y)=u_k(\rr_ky)$, we write
\begin{eqnarray*}
\frac{E[u_k]}{\mu(u_{0k})} &=& \frac{2}{\mu(u_{0k})}\int_0^{\rr_k} \left(\tfrac12(u_k')^2+ Q(u_k)\right) \stackrel{\eqref{appR}}=  \frac{2}{\mu(u_{0k})}\int_0^{\rr_k}  \left(2 Q(u_k)-R(u_{0k})u_k\right)
\\ &  \stackrel{[x=\rr_k y]}=& \frac{2\rr_k}{\mu(u_{0k})}\int_{0}^{1} \left(2 Q(v_k)-R(u_{0k})v_k\right)dy.
\end{eqnarray*}
By \eqref{panlim}, $\frac{2\rr_k}{\mu(u_{0k})}\to \frac{1}{\e_*}$ and $R(u_{0k})v_k\to R(\e_*)\e_*=Q(\e_*)$ in $L^1((0,1))$ as $k\to +\infty$. Hence, it suffices to prove that $Q(v_k)\to Q(\e_*)$ in $L^1((0,1))$. Applying Lemma \ref{lem:lb} with $K=\e_*+1$, we have $u_k(x)\ge C(\rr_k-x)^\frac{2}{m+1}$ for all $x\in (\rr_k-\delta_K,\rr_k)$, which in terms of $v_k$ means that
$$
v_k(y)\ge C \rr_k^\frac{2}{m+1}(1-y)^\frac{2}{m+1}\quad\mbox{for all $\frac{\rr_k-\delta_K}{\rr_k}< y< 1$.}
$$
Therefore, by monotonicity of $v_k$,
\begin{equation}
\label{bnm}
|Q(v_k)| \stackrel{\eqref{H}}\le C \left(1+\rr_k^\frac{2(1-m)}{m+1}(1-y)^\frac{2(1-m)}{m+1}\right) \stackrel{\rr_k>1, m>1}\le  C\left(1+ (1-y)^\frac{2(1-m)}{m+1}\right).
\end{equation}
Since $\frac{2(1-m)}{m+1}>-1$ if $m<3$, the right-hand side of \eqref{bnm} belongs to $L^1((0,1))$: therefore, an application of Lebesgue theorem yields the result.
\end{proof}

\subsection{Conclusion}

Let $u_M$ be a minimizer of $E$ in $\mathcal D_M$, $\supp u_M=[-\rr_M,\rr_M]$, and $u_{0M}=u_M(0)$.  We are interested in the behaviour of $u_M$ as $M\to +\infty$. We preliminarily estimate the energy of any sequence of minimizers whose maximal height remains finite.

\begin{lemma}\label{thm:pk-fake}  Assume \eqref{Q'0}. If $\A\ni u_{0M}\to \alpha\in [0,+\infty)$ for a sequence $M\to +\infty$ (not labeled), then the corresponding solutions $u_M$ to $(P_{u_{0M}})$ are such that:
\begin{itemize}
\item[$(i)$] $\supp u_M=[-\rr_M,\rr_M]$, $\rr_M<+\infty$, with $\rr_M\to +\infty$ as $M\to +\infty$;
\item[$(ii)$] $\displaystyle \liminf_{M\to +\infty}\frac{1}{M}E[u_M] \ge R(\alpha)$ ($R(0)=+\infty$).
\end{itemize}
\end{lemma}

\begin{proof}
$(i)$ follows from Corollary \ref{coro1} and mass constraint. For $(ii)$, since $u_{0M}\in \A$, we have $R(u_M)>R(u_{0M})$ in $(0,\rr_M)$. Therefore, as $M\to +\infty$,
\begin{eqnarray*}
E[u_M] &=& 2\int_0^{\rr_M} \left(\tfrac12(u_M')^2+ Q(u_M)\right) \stackrel{\eqref{appR}}=  2\int_0^{\rr_M}  \left(2 Q(u_M)-R(u_{0M})u_M\right)
\\&=& 2\int_0^{\rr_M} u_M \left(2 R(u_M)-R(u_{0M})\right) \ge R(u_{0M})M.
\end{eqnarray*}
\end{proof}

To avoid pathological situations, we assume a minimal monotonicity property on $R$:
\begin{eqnarray}
\label{Q'11}
\mbox{$R'<0$ in a left neighbourhood of $\e_*$.}
\end{eqnarray}
Note that \eqref{Q'11} is already included in \eqref{H-R} if $\e_*=+\infty$. Assumption \eqref{Q'11} suffices to infer the following.

\begin{lemma}
\label{rd1}
Assume \eqref{Q'0} and \eqref{Q'11}. Then a left neighbourhood of $\e_*$ is contained in $\A$.
\end{lemma}

\begin{proof}
If $\e_*=+\infty$, then $\inf R= R(+\infty)$ (cf. Remark \ref{rem:R2}): together with \eqref{Q'11}, this immediately implies the conclusion. If $\e_*<+\infty$, let $\delta>0$ such that $R'<0$ in $(\e_* -\delta,\e_*)$. If $\e_*$ is the first local minimum point of $R$, $R$ is non-increasing in $(0,\e_*)$ and there is nothing to prove. Otherwise, let $R_0:=\min_{(0,\e_*-\delta]} R>R(\e_*)$ and take $\tilde \delta\le \delta$ such that $R(e_*-\tilde\delta)<R_0$. Then $(\e_*-\tilde\delta,e_*)\subseteq \A$.
\end{proof}

We are now ready to conclude the analysis of generic macroscopic shapes of minimizers.

\begin{theorem} \label{thm:finale}
Assume \eqref{Q'0} and \eqref{Q'11}. Let $u_M$ be a minimizer of $E$ in $\mathcal D_M$. Let $\supp u_M=[-\rr_M,\rr_M]$ and $u_{0M}=u_M(0)$.
\begin{itemize}
\item (droplet) If $\e_*=+\infty$ and \eqref{H-add} holds, then
\begin{equation*}
\left.
\begin{array}{c}
\displaystyle u_{0M}^4 \sim \frac{9|S|}{32} M^2,\quad \rr_M^4\sim \frac{9}{8|S|} M^2, \quad\mbox{and}
\\[2ex]
\displaystyle u_{0M}^{-1} u_M(\rr_M y) \to 1-y^2\quad\mbox{in $C^2_{loc}((-1,1))$}
\end{array}
\right\}\quad\mbox{as $M\to +\infty$.}
\end{equation*}
\item  (pancake) If $\e_*<+\infty$, then
\begin{equation*}
\left.
\begin{array}{c}
\displaystyle u_{0M} \sim e_*, \quad \rr_M\sim \frac{1}{2\e_*}M, \quad\mbox{and}
\\[2ex]
\displaystyle u_M(\rr_M y) \to \e_* \quad\mbox{in $C^2_{loc}((-1,1))$}
\end{array}
\right\}\quad\mbox{as $M\to +\infty$.}
\end{equation*}
\end{itemize}
\end{theorem}

\begin{proof}
{\sl The droplet's case.} By Lemma \ref{rd1}, $a>0$ exists such that $(a,+\infty)\subset \A$. By \eqref{M-cont}, $\mu\in C((a,+\infty))$. Hence Lemma \ref{thm:par} is applicable and yields $\mu(s)\to +\infty$ as $s\to +\infty$. Therefore, for any sequence $M_k\to +\infty$ there exists a sequence $\A \ni u_{0k}\to +\infty$ such that $\mu (u_{0k})=M_k$. Moreover, by Theorem \ref{asbe1}, $u_k$ (the solution to $(P_{u_{0k}})$) belongs to $\mathcal D_{M_k}$. It follows from Lemma \ref{thm:par:en} that $u_k$ is such that $E[u_k]\lesssim \sqrt{M_k}$.

\smallskip

To conclude, assume by contradiction that $u_{0M}$ does not converge to $+\infty$. Then a subsequence $M_k$ exists such that $u_{0M_k}\to \alpha\in [0,+\infty)$. By Lemma \ref{fin1}, $u_{0M_k}\in \A$ for all $k$. By Lemma \ref{thm:pk-fake}, since $R>0$ in $(0,+\infty)$ (cf. Remark \ref{rem:R2}), we have $E[u_{M_k}] \gtrsim M_k$ as $k\to +\infty$. Hence, for $k$ sufficiently large we would have $E[u_k]<E[u_{M_k}]$, in contradiction with the definition of $u_M$. Thus $u_{0M}\to +\infty$ as $M\to +\infty$, and the result follows from Lemma \ref{thm:par}.

\smallskip

{\sl The pancake's case.}
By Lemma \ref{bnd-inf2}, $\A\subseteq (0,e_*)$. By Lemma \ref{rd1}, $I=(\e_*-\delta,\e_*)\subseteq \A$ for some $\delta\le\e_*$. By \eqref{M-cont}, $\mu\in C(I)$. By Lemma \ref{thm:pk}, $\mu(s)\to +\infty$ as $s\to e_*^-$. Therefore, for any sequence $M_k\to +\infty$ there exists a sequence $\A \ni u_{0k}\to \e_*^-$ such that $\mu (u_{0k})=M_k$. By Theorem \ref{asbe1} the solution $u_k$ to $(P_{u_{0k}})$ belongs to $\mathcal D_{M_k}$ and, by Lemma \ref{pan:en}, $u_k$ is such that $\frac{E[u_k]}{M_k} \to R(\e_*)$.

\smallskip

Assume by contradiction that $u_{0M}$ does not converge to $\e_*$. Then a subsequence $M_k$ exists such that $u_{0M_k}\to \alpha\in [0,\e_*)$. By Lemma \ref{fin1}, $u_{0M_k}\in \A$ for all $k$. By Lemma \ref{thm:pk-fake}, we have $\liminf\limits_{M_k\to +\infty}\frac{E[u_{M_k}]}{M_k} \ge R(\alpha)$.
Since $R(\e_*) <R(\alpha)$ (cf. \eqref{def-s*}), for $k$ sufficiently large we would have $E[u_{M_k}]>E[u_k]$, in contradiction with the definition of $u_M$. Hence $u_{0M}$ converges to $\e_*$ as $M\to +\infty$, and the conclusion follows from Lemma \ref{thm:pk}.
\end{proof}

\subsection{Transition profiles}

This is a limiting case: the same assumptions leading to Lemma \ref{thm:par} are assumed to hold, except that in this case $S=0$. In addition, we need to assume monotonicity of $Q$:
\begin{equation}\label{Q'<0}
Q'(s)<0 \quad \text{for all }\ s\in(0,+\infty).
\end{equation}

\begin{lemma}
\label{rd3}
Assume \eqref{Q'0}, \eqref{Q'<0}, and $S=0$.
Let $u_M$ be a minimizer of $E$ in $\mathcal D_M$ and let $u_{0M}=u_M(0)$. Then $u_{0M}\to +\infty$ as $M\to +\infty$.
\end{lemma}

\begin{remark}\label{rk:Q>0}{\rm
Note that \eqref{Q'<0} and $S=0$ already imply $Q>0$ and $R>0$ in $(0,+\infty)$. Thus $R'(s)=s^{-1} Q'(s)- s^{-2} Q(s)<0$ in $(0,+\infty)$, which implies $\A=(0,+\infty)$ and $\e_*=+\infty$.
}
\end{remark}

\begin{proof}
Let $\supp u_M=[-\rr_M,\rr_M]$. We recall that $u_M$ solves $(P_{u_{0M}})$. Let $u_{0M}$ be any sequence (not relabeled) such that $u_{0M}\to \alpha\in [0,+\infty]$ as $M\to +\infty$. We will exclude that $\alpha=0$ and that $\alpha\in (0,+\infty)$, which implies the statement.

\smallskip

Assume that $\alpha=0$. Then $R(u_{0M}) \to +\infty$. By Lemma \ref{thm:pk-fake} $(i)$, $(0,\rr_M)\supset (0,1)$ for $M\gg 1$. Since $u_M\le u_{0M}$ in $(0,\rr_M)$, by \eqref{Q'<0} we have
$$
u_M'(y)=u_M'(0) + \int_{0}^y u_M'' \stackrel{(P_{u_{0M}})}= \int_{0}^y \left(Q'(u_M)-R(u_{0M})\right) \le -R(u_{0M}) y
$$
for all $y\in(0,1)$, that is,
$$
0<u_M(x)= u_{0M} + \int_{0}^{x} u_M' \le u_{0M} - R(u_{0M})\int_{0}^{x} y dy  = u_{0M}-\tfrac12 R(u_{0M}) x^2 \to -\infty
$$
for all $x\in (0,1)$, a contradiction.

\smallskip

Assume that $\alpha\in (0,+\infty)$. By continuous dependence (Theorem 8.40 of \cite{KePe}), $u_M\to u$ in $C^2_{loc}(\{u>0\})$, where $u$ is the solution of $(P_\alpha)$. By Remark \ref{rk:Q>0} we have $\alpha\in \A$, hence it follows from Lemma \ref{fin1} that $u$ has compact support and, therefore, finite mass $\mu(\alpha)$. Since $\mu\in C(\A)$, we have $M=\mu(u_{0M}) \to \mu(\alpha)$ which contradicts $M\to +\infty$.
\end{proof}

In this limiting case, the macroscopic shape turns out to depend on the behavior of $Q$ and its derivatives as $s\to +\infty$: we assume that
\begin{equation}
\label{crit2}
\begin{array}{l}
\mbox{$K> 0$ and $p>1$ exist such that}
\\[1ex]
Q(s)\sim K s^{1-p}\quad\mbox{and}\quad Q'(s)\sim -(p-1) K s^{-p} \quad\mbox{as $s\to +\infty$.}
\end{array}
\end{equation}
We note for later reference that
\begin{equation}\label{crit2G}
-s^2R'(s) =G(s):=Q(s)-sQ'(s) \sim p K s^{1-p}\quad\mbox{as $s\to +\infty$.}
\end{equation}

\begin{remark}\label{rem:KvsAB}
{\rm
In the model cases $Q_a$ and $Q_b$ with $S=0$, we have $K=|B|>0$ and $p=n$ for $Q_{b}$, $K=-B>0$ and $p=n$ for $Q_{a}$ with $B<0$, and $K=A>0$ and $p=m$ for $Q_{a}$ with $B=0$ (cf. \eqref{modelQ1a} and \eqref{modelQ1b}).
}\end{remark}

\begin{theorem}[transition profiles]\label{thm:crit}
Assume \eqref{Q'0}, \eqref{Q'<0}, and \eqref{crit2}. Then, as $M\to +\infty$,
\begin{equation}
\label{crit3quad}
u_{0M}^{p+3} \sim \frac{pK}{2 c_p^2f_p(0)^2} M^2, \quad
\rr_M^{p+3} \sim \frac{f_p(0)^2}{2^{p+2} p K c_p^{p+1}} M^{p+1}, \quad\mbox{thus}\quad\frac{u_{0M}^{p+1}}{\rr_M^2}\sim \frac{2pK}{f_p(0)^2},
\end{equation}
and
\begin{equation}
\label{crit3bis}
w_M(y)= u_{0M}^{-1} u_M(\rr_My) \to f_p^{-1}(f_p(0)|y|) \quad\mbox{in $C^2_{loc}((-1,1))$},
\end{equation}
where
\begin{equation}
\label{def-fp}
f_p(w):=\int_{w}^1 \frac{\sqrt{p}\tilde w^{\frac{p-1}{2}}d\tilde w}{\sqrt{1-\tilde w^p}} \qquad \text{for }w\in [0,1],
\end{equation}
and $\displaystyle c_p=\int_0^1 f_p^{-1}(f_p(0)y)dy.$
\end{theorem}

\begin{proof}
We let $f=f_p$ for notational convenience. Thanks to Lemma \ref{rd3}, $u_{0M}\to +\infty$ as $M\to +\infty$. As we noticed there, Step 1 in the proof of Lemma \ref{thm:par} holds for $-S\ge 0$: therefore
\begin{equation*}
\rr_M\to +\infty \quad\mbox{as $M\to +\infty$.}
\end{equation*}
Let $w_M(y)=u_{0M}^{-1} u_M(\rr_M y)$. It follows from \eqref{as7} that
\begin{equation}
\label{critnew}
-\frac{w_M'}{\sqrt{2}\sqrt{Q(u_{0M}w_M)-Q(u_{0M})w_M}}=\frac{\rr_M}{u_{0M}} \qquad \text{in } [0,1).
\end{equation}
We claim that $C\ge 1$ exists such that
\begin{equation}
\label{crit6}
C^{-1}\leq \frac{\rr_M^2 G(u_{0M})}{u_{0M}^2}\stackrel{(\ref{crit2G})}\sim \frac{pK\rr_M^2}{u_{0M}^{p+1}} \leq C \quad\mbox{for all $M\gg 1$}.
\end{equation}
Integrating \eqref{critnew} in $(0,1)$ and multiplying by $\sqrt{G(u_{0M})}$ we obtain
\begin{equation}
\label{crit5}
\frac{\rr_M\sqrt{G(u_{0M})}}{u_{0M}} = \frac{1}{\sqrt 2 }\int_0^1\frac{\sqrt{G(u_{0M})}dw}{\sqrt{Q(u_{0M}w)-Q(u_{0M})w}}.
\end{equation}
Since $u_{0M}\to +\infty$, we have
\begin{equation}\label{crit-1/2-1}
\int_{1/2}^1\frac{\sqrt{G(u_{0M})}dw}{\sqrt{Q(u_{0M}w)-Q(u_{0M})w}} \stackrel{\eqref{crit2},\eqref{crit2G},\eqref{def-fp}}\sim f(1/2),
\end{equation}
which, by \eqref{crit5}, already proves the lower bound in \eqref{crit6}.  For the upper bound, we note that, since $Q>0$ in $(0,+\infty)$ (cf. Remark \ref{rk:Q>0}),
\begin{eqnarray}
\frac{d}{dw}\left(Q(u_{0M}w)-Q(u_{0M})w\right) = u_{0M}Q'(u_{0M}w)-Q(u_{0M})<u_{0M}Q'(u_{0M}w) \stackrel{\eqref{Q'<0}}<0.
\label{crit7}
\end{eqnarray}
Therefore, again since $u_{0M}\to +\infty$,
\begin{equation*}
\int_0^{\frac{1}{2}}\frac{\sqrt{G(u_{0M})}dw}{\sqrt{Q(u_{0M}w)-Q(u_{0M})w}} \stackrel{\eqref{crit7}}\leq \frac{1}{2}\frac{\sqrt{G(u_{0M})}}{\sqrt{Q\left(\frac{u_{0M}}{2}\right)-\frac12 Q(u_{0M})}} \stackrel{\eqref{crit2},\eqref{crit2G}}\sim \frac{1}{\sqrt 2}\sqrt{\frac{p}{2^p-1}}
\end{equation*}
as $M\to +\infty$, which together with \eqref{crit-1/2-1} yields the upper bound in \eqref{crit6}.

\smallskip

Thanks to \eqref{crit6}, any sequence $M\to +\infty$ has a subsequence (not relabeled) such that
\begin{equation}
\label{crit15}
\frac{\rr_M\sqrt{G(u_{0M})}}{u_{0M}} \sim \frac 1 \ell >0 \qquad \text{as }M\to +\infty
\end{equation}
for some $\ell>0$. We will show that $\ell=\sqrt{2}/f(0)$: together with the arbitrariness of the (sub)sequence and with \eqref{crit2G}, this will prove \eqref{crit3quad}$_3$. Since $w_M$ is strictly decreasing in $(0,1)$, $y_M=w_M^{-1}$ is well defined. Fix $0<w\leq 1$. Multiplying \eqref{critnew} by $\sqrt{G(u_{0M})}$ and integrating in $(0,y_M(w))$, since $u_{0M}\to +\infty$ we obtain
\begin{eqnarray}
\nonumber
f(w) &\stackrel{\eqref{def-fp}} =&\int_{w}^1 \frac{\sqrt{p}\tilde w^{\frac{p-1}{2}}d\tilde w}{\sqrt{1-\tilde w^p}} \stackrel{\eqref{crit2}}\sim \int_{w}^{1}\frac{\sqrt{G(u_{0M})}d\tilde w}{\sqrt{Q(u_{0M}\tilde w)-Q(u_{0M})\tilde w}}
\\ &=& \frac{\rr_M\sqrt{2 G(u_{0M})}}{u_{0M}} y_M(w) \stackrel{\eqref{crit15}}\sim \frac{\sqrt{2}}{\ell} y_M(w)\quad\mbox{as $M\to +\infty$}.
\label{crit17}
\end{eqnarray}
It follows from \eqref{crit17} that, for every $0< w\leq 1$,
\begin{equation}
\label{crit18}
y_M(w) \to  y(w):= \tfrac{\ell}{\sqrt{2}}f(w)\in \left[0,\tfrac{\ell f(0)}{\sqrt{2}}\right)
\ \text{ as $M\to +\infty$}.
\end{equation}
By construction, $\frac{\ell f(0)}{\sqrt{2}}\leq 1$. Since $w_M$ and $f$ are strictly decreasing, \eqref{crit18} is equivalent to
\begin{equation}
\label{critnew2}
w_M(y)\to w(y):=f^{-1}\left(\tfrac{\sqrt 2}{\ell}y\right) \ \mbox{ as $M\to +\infty$}\quad \mbox{for all $y\in \left[0,\frac{\ell f(0)}{\sqrt{2}}\right)$}.
\end{equation}
We now prove that $\frac{\ell f(0)}{\sqrt{2}}=1$. It follows from \eqref{critnew} that for all $M\gg 1$
\begin{eqnarray}
\label{crit25}
-w_M'(y)&=&\frac{\rr_M\sqrt{G(u_{0M})}}{u_{0M}} \sqrt{\frac{2(Q(u_{0M}w_M(y))-Q(u_{0M})w_M(y))}{G(u_{0M})}} \nonumber \\
&\stackrel{\eqref{crit15},\eqref{crit2G}}\geq & \ell^{-1} \sqrt{\frac{Q(u_{0M}w_M(y))-Q(u_{0M})w_M(y)}{pKu_{0M}^{1-p}}}.
\end{eqnarray}
By \eqref{Q'<0} and since $w_M(y)\leq 1$, we have for all $M\gg 1$
$$
Q(u_{0M}w_M(y))-Q(u_{0M})w_M(y)\geq Q(u_{0M})(1-w_M(y))\stackrel{\eqref{crit2}}\geq \frac{K}{2}u_{0M}^{1-p}(1-w_M(y)),
$$
which inserted in \eqref{crit25} implies
\begin{equation}
\label{crit26}
-w_M'(y)\geq \frac{\ell^{-1}}{\sqrt{2p}}\sqrt{1-w_M(y)} \qquad\text{for every }y\in [0,1).
\end{equation}
Assume by contradiction that $\frac{\ell f(0)}{\sqrt{2}}<1$. Then, by \eqref{critnew2},
$w_M(y)\to 0$ for every $y\in \left(\frac{\ell f(0)}{\sqrt{2}},1\right]$. Hence, using \eqref{crit26}, we deduce for all $M\gg 1$ that $-w_M'(y)\geq \frac{\ell^{-1}}{2\sqrt{p}}$ for every $y\in \left(\frac{\ell f(0)}{\sqrt{2}},1\right]$. It follows that
$$
w_M(y)=-\int_y^1 w_M'(z)dz \geq \frac{\ell^{-1}}{2\sqrt{p}}(1-y) \qquad \text{for all }y\in \left(\tfrac{\ell f(0)}{\sqrt{2}},1\right], \ M\gg 1,
$$
a contradiction. Therefore $\ell = \sqrt{2}/f(0)$.

The pointwise convergence in \eqref{crit3bis} now follows from \eqref{critnew2} recalling that $w_M$ is an even function. To prove the uniform convergence, we note that $w_M$ solves
\begin{equation*}
\begin{cases}
\displaystyle -\tfrac{u_{0M}^2}{\rr_M^2 G(u_{0M})} w_M''  = \frac{Q(u_{0M})- u_{0M}Q'(u_{0M}w_M)}{G(u_{0M})} \\
w_M(0)=1, \ w_M'(0)=0.
\end{cases}
\end{equation*}
In view of \eqref{crit15} and \eqref{crit2}, as $M\to +\infty$, $w_M$ converges in $C^2_{loc}(\{w>0\})$ to the unique solution of
\begin{equation*}
\begin{cases}
\displaystyle -\tfrac{2}{f(0)^2} w''  = \frac{w^p+p-1}{pw^p} \\
w(0)=1, \ w'(0)=0,
\end{cases}
\end{equation*}
which coincides with $f^{-1}_p(f_p(0)|y|)$ for $|y|<1$. In particular, as $M\to + \infty$ we have
\begin{equation}
\label{critnew3}
\frac{M}{2} = \int_0^{\rr_M} u_M dx = u_{0M} \rr_M \int_0^1 w_M(y)dy \stackrel{\eqref{crit3bis}}\sim u_{0M} \rr_M \int_0^1 f_p^{-1}(f_p(0)y)d y = c_p u_{0M} \rr_M.
\end{equation}
Combining \eqref{crit3quad}$_3$ and \eqref{critnew3}, after straightforward computations we obtain \eqref{crit3quad}$_1$ and \eqref{crit3quad}$_2$.
\end{proof}

\begin{remark}\label{rem:f}{\rm
The function $f_p$ behaves as
$$
f_p(w)\sim \left\{\begin{array}{ll}
f_p(0) -\frac{2\sqrt{p}}{p+1} w^{\frac{p+1}{2}} & \mbox{as $w\to 0^+$}
\\[1ex]
2\sqrt{1-w} & \mbox{as $w\to 1^-$.}
\end{array}\right.
$$
Hence
$$
w_p(y):=f_p^{-1}(f_p(0)y) \sim \left\{\begin{array}{ll}
\left(\frac{p+1}{2\sqrt{p}}f_p(0)(1-y)\right)^\frac{2}{p+1} & \mbox{as $y\to 1^-$}
\\[1ex]
1-\frac{f_p(0)^2}{4}y^2 & \mbox{as $y\to 0^+$.}
\end{array}\right.
$$
}
\end{remark}

\begin{remark}\label{rem:f1}{\rm
The previous remark implies that the critical case is a transition between droplets and pancakes:
\begin{itemize}
\item[$(i)$] for $p=1$, we have $f_1(w)=2\sqrt{1-w}$, $w_1(y)=1-y^2$, and $c_1=2/3$, so that
$$
u_{0M}^{4} \sim \frac{9 K}{ 32} M^2 \quad\text{and}\quad
\rr_M^{4} \sim \frac{9}{8 K} M^{2}
\quad\text{as $M\to +\infty$ and $p\to 1$},
$$
thus recovering the droplet case with $K=-S>0$ in the limit as $p\to 1$;

\item[$(ii)$] as $p\to +\infty$, we have $w_p(y)\to 1$, $c_p\to 1$, and
$$u_{0M} \sim 1 \quad\text{and}\quad
\rr_M \sim \frac{1}{2} M \quad\text{as $M\to +\infty$ and $p\to +\infty$},
$$
thus recovering the pancake case with $\e_*=1$ in the limit as $p\to +\infty$.
\end{itemize}
To see $(ii)$, note that $f_p(0)\to 0$, $\sqrt{p}f_p(0)\to \pi$, and $\left(\frac{f_p(0)^2}{2^{p+2} p}\right)^{\frac{1}{p+3}} \to \frac12$ as $p\to +\infty$. Since $w_p$ is strictly decreasing, $w_p(0)=1$ and $w_p(1)=0$, $y_p$ exists such that $w_p(y_p)=1-\frac{f_p(0)^2}{4}$ for $p\gg 1$. Moreover $p\to +\infty$ if and only if $w_p(y_p)\to 1$, hence we deduce that as $p\to +\infty$
$$
f_p(0)y_p=f_p(w_p(y_p))\sim 2\sqrt{1-w_p(y_p)}=f_p(0),
$$
which implies that $y_p\to 1$ as $p\to +\infty$. Therefore $w_p\to 1$ and $(ii)$ follows.
}
\end{remark}

\section{Asymptotic behaviour for small masses}\label{ss-small}

As expected, minimizers tend uniformly to zero as $M\to 0$.

\begin{lemma}
\label{smallM}
Assume \eqref{Q'0}. Let $u_M$ be a minimizer of $E$ in $\mathcal D_M$. Then $u_{0M}:=u_M(0)\to 0$ as $M\to 0^+$.
\end{lemma}

\begin{proof}
By Corollary \ref{coro1}, Theorem \ref{th-lambda}, and Lemma \ref{fin1}, $\supp u_M=[-\rr_M,\rr_M]$ with $r_M<+\infty$, $\{u_{0M}\}\subset \A$, and $\lambda_M=R(u_{0M})$. Assume by contradiction that $u_{0M}\to \alpha \in \overline{\A}\setminus \{0\}$ for a subsequence (not relabeled).

We first analyze the case $\alpha<+\infty$. Since $\lambda_M\to \lambda:=R(\alpha)$, it follows from Lemma \ref{lemtec} that $\rr_M\to \rr\in (0,+\infty]$ and $u_M\to u$ in $C_{loc}((-\rr,\rr))$ as $M\to 0^+$, where $u$ is the solution to $(P_{\alpha})$ and $\{u>0\}=(-\rr,\rr)$. This implies, applying Fatou's Lemma, that
$$
0=\lim_{M\to 0^+} M=\liminf_{M\to 0^+} \int_{-\rr_M}^{\rr_M} u_M \geq \int_{-\rr}^{\rr} u>0,
$$
a contradiction.

If instead $u_{0M}\to +\infty$, then $R>0$ in $(0,+\infty)$: indeed, otherwise we would have $\e_*<+\infty$ (Remark \ref{rem:R1}), which would imply $u_{0M}<\e_*$ (Lemma  \ref{bnd-inf2}), a contradiction. In addition, $Q(s)\leq c s$ for $s\geq 1$ for some $c>0$ (otherwise, we would again have $\e_*<+\infty$). Since $u_{0M}\to +\infty$ and $u_M$ is strictly decreasing with compact support, $x_M\in (0,\rr_M)$ exists such that $x_M=u_M^{-1}(1)$ for $M$ sufficiently small. Hence
\begin{eqnarray}
\label{smallapp}
u_{0M}-1 &= & \int_{0}^{x_M}-u'(x) \stackrel{\eqref{as7}}= \int_0^{x_M} \sqrt{2(Q(u_M(x))-R(u_{0M})u_M(x))} \nonumber \\
& \stackrel{R>0}\leq &\int_0^{x_M} \sqrt{2Q(u_M(x))} \stackrel{Q(s)\leq cs^2}\leq \sqrt{2c} \int_{-\rr_M}^{\rr_M} u_M(x) = \sqrt{2c}M.
\end{eqnarray}
Letting $M\to 0$ in \eqref{smallapp} we obtain a contradiction.
\end{proof}

\section{Non-uniqueness of minimizers}\label{ss-non}

In this subsection we prove non-uniqueness of minimizers of $E$ in $\mathcal D_M$ (here we explicit the dependence of $\mathcal D$ on $M$) when the weighted potential $R$ is such that
\begin{equation}
\label{nonunfond}
\A\neq (0,\e_*).
\end{equation}
Moreover, we need an additional property of $Q$, namely:
\begin{equation}
\label{strongH}
Q'(s)\sim -A(m-1)s^{-m} \quad \text{as }s\to 0^+.
\end{equation}

Throughout the section, $\{u_{0k}\}\subset \A$ and $u_k$ is the solution of $(P_{u_{0k}})$, defined in $\{u_k>0\}=(-\rr_k,\rr_k)$. We let $E(u_{0k}):=E[u_k]$ and we recall that $\mu(u_{0k})$ denotes the mass of $u_k$ (cf. \eqref{def-mu}). In what follows we study the behaviour of $\mu$ and $E$ on $\overline{\A}$. Let's start by showing what happens when $u_{0k}$ tends to $0$.

\begin{lemma}
\label{tendus0}
Assume \eqref{Q'0} and \eqref{strongH}. Let $\{u_{0k}\}\subset \A$ be such that $u_{0k}\to 0$ as $k\to +\infty$. Then $\mu(u_{0k}) \to 0$ as $k\to +\infty$.
\end{lemma}

\begin{proof}
Let $\omega_k= \sqrt{{A/u_{0k}^{m+1}}}$. The function $w_k(y)=u_{0k}^{-1}u_k\left(\omega_k^{-1}y\right)$ is a solution to
\begin{equation*}
\begin{cases}
- w''_k= f_k(w_k):=A^{-1} (u_{0k}^{m-1}Q(u_{0k})-u_{0k}^m Q'(u_{0k}w_k)) \quad\mbox{in $(-\omega_k \rr_k,\omega_k \rr_k)$,}
\\[1ex]
w_k(0)=1, \ w_k'(0)=0, \ w_k(\pm \omega_k \rr_k)=0.
\end{cases}
\end{equation*}
By \eqref{H} and \eqref{strongH}, $f_k(w)\to  1+(m-1)w^{-m}$ in $C_{loc}((0,+\infty))$ as $k\to +\infty$ and
$$
F_k(w):=\int_w^1 f_k(t)dt=A^{-1}u_{0k}^{m-1}(Q(u_{0k}w)-Q(u_{0k})w) \stackrel{\eqref{H}} \sim w^{1-m}-w \to +\infty
$$
as $(k,w)\to (+\infty,0^+)$. Applying Lemma \ref{lemtec}, we deduce that, as $k\to +\infty$, $w_k\to w$ in $C^2_{loc}(\{w>0\})$, where $w$ is the solution to
\begin{equation*}
\begin{cases}
-w''=1+(m-1)w^{-m}  & \quad \mbox{in }\{w>0\}=(-\rr,\rr), \\
w(0)=1, \ w'(0)=0, \ w(\rr)=0,
\end{cases}
\end{equation*}
and $\omega_k\rr_k\to \rr$. Note that $w$ is concave, hence $\rr<+\infty$. Therefore $\rr_k \to 0$ as $k\to +\infty$, whence
$$
\mu(u_{0k})=2\int_{0}^{\rr_k} u_k(x)dx \leq 2\rr_k u_{0k} \to 0\quad\mbox{as $k\to +\infty$.}
$$
\end{proof}

Now we show that $E$ is continuous on $\A$.
\begin{lemma}
\label{contE}
Assume \eqref{Q'0}.
Then $\mu\in C(\A)$ and $E\in C(\A)$.
\end{lemma}

\begin{proof}
Continuity of $\mu$ follows from \eqref{M-cont}, hence we only show continuity of $E$. Let $\A\ni u_{0k}\to \alpha \in \A$ as $k\to +\infty$. Since $\alpha\in \A$, we have $\alpha>0$. By Lemma \ref{lemtec}, $u_k\to u$, where $u$ solves $(P_{\alpha})$ in $C^2_{loc}((-\rr,\rr))$ with $\{u>0\}=(-\rr,\rr)$, and $\rr_k\to \rr>0$ as $k\to +\infty$. Note that $\rr<+\infty$ since $\alpha\in \A$ (cf. Lemma \ref{fin1}).

\smallskip

Let $\rho_k=\rr_k/\rr$ and $w_k(y)=u_k(\rho_k y)$, $y\in (-\rr,\rr)$. Obviously $\rho_k\to 1 $ and $w_k\to u$ locally uniformly, hence a.e., in $(-\rr,\rr)$. We have
\begin{eqnarray}
\nonumber
\frac12 E(u_{0k}) &=& \int_{0}^{\rr_k} \left(\tfrac{1}{2}(u'_k)^2 + Q(u_k)\right) \stackrel{\eqref{appR}}= \int_{0}^{\rr_k}\left(2 Q(u_k)-R(u_{0k})u_k\right)
\\ &=& \rho_k \int_{0}^{\rr}\left(2 Q(w_k)-R(u_{0k})w_k\right).
\label{jjmm}
\end{eqnarray}
The properties of $Q$ imply that
\begin{equation}
\label{conte4}
|Q(s)|\leq C(1+s^{1-m}) \quad\text{for all } s\in(0,\alpha+1).
\end{equation}
Applying Lemma \ref{lem:lb} with $K=\alpha+1$, we have $u_k(x)\geq C(\rr_k-x)^{\frac{2}{m+1}}$ for all $x\in (\rr_k-\delta_K,\rr_k)$ and $k$ sufficiently large. In terms of $w_k$, this means that
\begin{equation}
\label{conte8}
w_k(y)\geq C(\rr-y)^{\frac{2}{m+1}} \quad \text{for all }y\in \left(\rr-\tfrac{\delta_K}{2},\rr\right)
\end{equation}
for $k$ sufficiently large. It follows from the monotonicity of $w_k$ and the locally uniform convergence $w_k\to u$ that
\begin{equation}
\label{conte9}
|Q(w_k)| \stackrel{\eqref{conte4},\eqref{conte8}}\leq C\left(1+(\rr-y)^{\frac{2(1-m)}{m+1}}\right).
\end{equation}
Since $\frac{2(1-m)}{m+1}>-1$ if $m<3$, the right-hand side of \eqref{conte9} belongs to $L^1((0,\rr))$: therefore an application of Lebesgue theorem in \eqref{jjmm} implies the result.
\end{proof}

Finally, we study the behavior of $\mu$ on $\partial{\A}\setminus \{0\}$.

\begin{lemma}
\label{muuE}
Assume \eqref{Q'0}; if $\e_*=+\infty$, assume in addition \eqref{H-add}. Let $\{u_{0k}\}\subset \A$ be such that $u_{0k}\to \alpha\in \partial{\A}\setminus \{0\}$ as $k\to +\infty$,
where $\partial \A$ denote the boundary of $\A$ with respect to the topology of $\R^*$. Then $\mu(u_{0k})\to +\infty$ as $k\to +\infty$.
\end{lemma}

\begin{proof}
If $\alpha=+\infty$, it follows from Lemma \ref{bnd-inf2} that $\e_*=+\infty$. Thus, by Lemma \ref{thm:par}, we have that $\mu(u_{0k})\to +\infty$.

\smallskip

Now let $\alpha\in \partial \A\setminus \{0,+\infty\}$. By continuous dependence (Theorem 8.40 of \cite{KePe}), $u_k\to u$ in $C^2_{loc}(\{u>0\})$, where $u$ solves $(P_{\alpha})$ and $\{u>0\}=(-\rr,\rr)$ for $\rr\in (0,+\infty]$. Since $R$ is continuous in $(0,+\infty)$ and $\alpha\in \partial \A\setminus \{0,+\infty\}$, by definition of $\A$ at least one of the following holds:
\begin{itemize}
\item[$(i)$]  $R'(\alpha)=0$. Then $u\equiv \alpha$, hence (by Fatou's Lemma)
$
\liminf\limits_{k\to +\infty} \mu(u_{0k}) \geq \int_{\R} \alpha =+\infty;
$
\item[$(ii)$] $0<t<\alpha$ exists such that $R'(t)=0$ and $R(t)=R(\alpha)$. In this case, we claim that $u>t$ for every $x\in \R$. Indeed, if $x>0$ exists such that $u(x)=t$, then
$$
u''(x)\stackrel{(P_\alpha)}= Q'(u(x))-R(\alpha) = Q'(t)-R(t) \stackrel{\eqref{G1}}=t R'(t)=0 \quad\text{and} \quad u'(x)\stackrel{\eqref{appR}}=0,
$$
whence $u\equiv t$, in contradiction with $u(0)=\alpha>t$. It follows from Fatou's Lemma that
$
\liminf\limits_{k\to +\infty} \mu(u_{0k}) \geq \int_{\R} t =+\infty.
$
\end{itemize}
\end{proof}

Now we are ready to prove the non-uniqueness result.
\begin{theorem}
\label{no-un}
Assume \eqref{Q'0}, \eqref{Q'11}, \eqref{nonunfond}, and \eqref{strongH}; if $\e_*=+\infty$, assume in addition \eqref{H-add}.
Then $M>0$ exists such that $E$ has at least two minimizers in $\mathcal D_M$.
\end{theorem}

\begin{proof}
We argue by contradiction assuming that for every $M>0$ there exists a unique minimizer $u_M$ of $E$ in $\mathcal D_M$. By Lemma \ref{fin1}, we know that $u_{0M}:=u_M(0)\in \A$. Then the function $\mathcal P:(0,+\infty)\to \A$, $\mathcal P(M):=u_{0M}$, is well defined, and $u_M$ is the unique solution to $(P_{u_{0M}})$.

\smallskip

We claim that $\mathcal P$ is continuous. Take any sequence $M_k\to M\in (0,+\infty)$ as $k\to +\infty$ and take any subsequence (not relabeled) such that $\mathcal P(M_k)\to \alpha\in {\overline{\A}}$. Since $\alpha\in \partial\A$ is excluded by Lemmas \ref{tendus0} and \ref{muuE}, in fact $\alpha\in \A$. We will show that $\alpha=\mathcal P(M)$, which implies continuity in view of the arbitrariness of the subsequence.

\smallskip

Let $u$ be the solution of $(P_\alpha)$. By Lemma \ref{contE},
\begin{equation}
\label{noun1}
M_k=\mu(\mathcal P(M_k))\to \mu(\alpha)=M \quad\text{and}\quad E(\mathcal P(M_k))\to E(\alpha)=E[u] \quad\mbox{as $k\to +\infty$.}
\end{equation}
Since $\mu(\alpha)=M$ and since by assumption $u_M$ is unique, it suffices to prove that $E[u]\leq E[u_M]$. Indeed, then $u=u_M$ and thus $\alpha=\mathcal P(M)$, proving continuity of $\mathcal P$.

\smallskip

We set $\gamma_k=\frac{M_k}{M}\to 1$ as $k\to+\infty$. Since $\gamma_k u_M\in \mathcal D_{M_k}$ and $u_{M_k}$ is a minimizer in $\mathcal D_{M_k}$, it holds that
\begin{equation*}
\int \left(\frac{\gamma_k^2}{2} (u'_M)^2 + Q(\gamma_ku_M)\right)=E[\gamma_ku_M]\geq E[u_{M_k}] \qquad\forall k>0.
\end{equation*}
A straightforward application of dominated convergence theorem then yields $E[u]\le E[u_M]$:
$$
E[u_M] \longleftarrow E[\gamma_ku_M]\geq E[u_{M_k}]\stackrel{\eqref{noun1}}\longrightarrow E[u].
$$

We are now ready to conclude. We know that
\begin{equation*}
\begin{array}{ll}
\mathcal P(M)\to 0 \ \text{ as }M\to 0 \quad & \mbox{(Lemma \ref{smallM}),}
\\
\mathcal P(M)\to \e_* \ \text{ as }M\to +\infty  \quad & \mbox{(Theorem \ref{thm:finale}).} \\
\end{array}
\end{equation*}
These two information, together with the assumption $\A\neq (0,\e_*)$, contradict the continuity of $\mathcal P$ in $(0,+\infty)$ and complete the proof.
\end{proof}

\section{Model cases}
\label{s:mod}

Here we take a closer look at the four model cases referred to in the Introduction. We recall that large-mass asymptotic results depend on two quantities: $\e_*$, which is the smallest among the global minimum points of $R(s)=Q(s)/s$, if any, and $+\infty$ otherwise; and the sign of $S$. Uniqueness also depends on the sign of $Q''(s)$ for $s\in (0,\e_*)$.

\smallskip

Throughout the section, $N=1$ and $u_M$ is a minimizer of $E$ in $\mathcal D_M$.  Any $Q$ (and $R$) in this section obviously satisfies \eqref{Q'0}, \eqref{Q'11}, \eqref{strongH}, \eqref{H-add} if $-S>0$ and $D=0$, and \eqref{Q'<0}-\eqref{crit2} if $-S=0$, $B\le 0$ and $D=0$. Therefore, for the macroscopic shape we will only need to check whether $\e_*<+\infty$ or not (cf. Theorem \ref{thm:finale} and Theorem \ref{thm:crit}), and for uniqueness we will only need to discuss convexity of $Q$ in $(0,\e_*)$ (cf. Theorem \ref{uniq1d}).

\begin{proposition}[model $Q_a$]\label{prop:Qa}
Let
\begin{equation}
\label{modelQ1a-disc}
Q(s)=Q_{a}(s) = \begin{cases}
As^{1-m}- Bs^{1-n}-S & \mbox{for $s>0$},\quad A>0, \ B\in \R, \ S\in\R, \ 1<n<m<3, \\
0 & \mbox{for $s\leq 0$}.
\end{cases}
\end{equation}
\begin{itemize}
\item[$(i)$] {\it Uniqueness.} $u_M$ is:
\begin{itemize}
\item unique for any $M\in (0,+\infty)$ if $B\le 0$ or if $B>0$ and $-S\le 0$ or if $-S>0$ and $B\ge c_1(A,S)$, where
\begin{equation}\label{def:c1}
c_1(A,S):= \textstyle (m-1)\left(\frac{A}{n-1}\right)^{\frac{n-1}{m-1}} \left(\frac{-S}{m-n}\right)^{\frac{m-n}{m-1}};
\end{equation}
\item not unique for at least one $M\in (0,+\infty)$ if $-S>0$ and $c_2(A,S) \leq B < c_1(A,S)$, where
$$
c_2(A,S):=  \tfrac{m-1}{n}\left(\tfrac{Am}{n-1}\right)^{\frac{n-1}{m-1}} \left(\tfrac{-S}{m-n}\right)^{\frac{m-n}{m-1}}.
$$
\end{itemize}
\item[$(ii)$] {\it Macroscopic behavior.} For $M\gg 1$, $u_M$ is:
\begin{itemize}
\item droplet-shaped (Theorem \ref{thm:finale}) if $-S>0$ and $B< c_1(A,S)$;
\item pancake-shaped (Theorem \ref{thm:finale}) if $-S<0$, or if $-S=0$ and $B>0$, or if $-S>0$ and $B\ge c_1(A,S)$;
\item transition-shaped (Theorem \ref{thm:crit}) if $S=0$ and $B\le 0$.
\end{itemize}
\end{itemize}
\end{proposition}

Uniqueness of minimizers remains open if $-S>0$ and $0< B<c_2(A,S)$. Figure \ref{fig-Q1a} summarizes the results.

\begin{proof}
We have
\begin{equation*}
R(+\infty)=0, \qquad G(s):= -s^2 R'(s) = A m s^{1 - m} - B n s^{1 - n} - S \quad \mbox{for $s>0$},
\end{equation*}
hence $G(0)=+\infty$ and $G(+\infty)=-S$. Simple computations show that $G'$ has no zeroes if $B\le 0$. In this case $G$ is monotone decreasing in $(0,+\infty)$, which means that $G$, whence $R'$, never changes sign if $-S\ge 0$, whereas it changes sign once if $-S<0$. If instead $B>0$, $G'$ changes sign once, with its zero located at $g=\left(\frac{Am(m-1)}{Bn(n-1)}\right)^{\frac{1}{m-n}}$, and
\begin{eqnarray*}
g^{m-1} G(g) &=& A m - B n g^{m - n} - Sg^{m-1} = Am\tfrac{n-m}{n-1}-S \left(\tfrac{Am(m-1)}{Bn(n-1)}\right)^{\frac{m-1}{m-n}}.
\end{eqnarray*}
Hence, if $B>0$,
$$
\textstyle G(g)< 0 \quad \Leftrightarrow \quad -S\le 0 \ \mbox{ or } \ -S> 0 \ \mbox{ and } \ -S\left(\tfrac{Am(m-1)}{Bn(n-1)}\right)^{\frac{m-1}{m-n}}<Am\tfrac{m-n}{n-1},
$$
and the latter inequality is equivalent to $B>c_2(A,S)$. Summarizing:
$$
\mbox{$R'$} \left\{
\begin{array}{ll}
\mbox{changes sign once}  & \mbox{if $-S<0$ or if $-S=0$ and $B>0$}
\\[1ex]
\le 0  & \mbox{if $-S\ge 0$ and $B\le c_2(A,S)$}
\\[1ex]
 \mbox{changes sign twice}  & \mbox{if $-S>0$ and $B>c_2(A,S)$}.
\end{array}\right.
$$
If $R'$ never changes sign, then $\e_*=+\infty$. If $R'$ changes sign once, then $\e_*<+\infty$ and $R(\e_*)<0=R(+\infty)$. If $R'$ changes sign twice, then
$$
\mbox{$Q>0$ in $(0,+\infty)$} \ \iff \ 0<B<c_1(A,S)=n m^{\frac{1-n}{m-1}}c_2(A,S),
$$
wtih $c_1(A,S)$ as in \eqref{def:c1}. Note that $c_1(A,S)>c_2(A,S)$ since $n<m$. Therefore, recalling Remark \ref{rem:R2} and Remark \ref{rem:R1}, we deduce that $\e_*<+\infty$ in the cases listed in the pancake case of Proposition \ref{prop:Qa}, and $\e_*=+\infty$ otherwise. Distinguishing between droplets and transition is obvious.

\smallskip

As to $(i)$, we need to discuss the sign of $Q''(s)=Am(m-1) s^{-m-1} -Bn(n-1) s^{-n-1}$ in $(0,\e_*)$.  We distinguish various cases.

\smallskip

- If $B\le 0$, then $Q''\ge 0$ in $(0,+\infty)$.

\smallskip

- If $B>0$ and $-S\le 0$, then $e_*<+\infty$ and $Q''$ has one zero. Since $Q''(s)\to +\infty$ as $s\to 0^+$ and
    $$
    Q''(\e_*)= \e_*^{-2}(m-1) \left(Am \e_*^{1-m} -Bn\tfrac{n-1}{m-1}\e_*^{1-n}\right)\stackrel{R'(\e_*)=0}=  e_*^{-2}(m-1) \left(Bn\tfrac{m-n}{m-1}e_*^{1-n}+S\right)\ge 0,
    $$
$Q''>0$ in $(0,\e_*)$.

\smallskip

- If $-S>0$ and $B\ge c_1(A,S)$, we have seen that $\e_*<+\infty$ and that $R'$ changes sign twice. Since $G'(s)= -sQ''(s)$, the unique zero of $Q''$ is located at the above defined point $s_2=g$. Above, we have also seen that $G(g)<0$, hence $R'(g)>0$. Since $R'$ changes sign twice and recalling the definition of $\e_*$, this implies that $\e_*\le g$. Therefore $Q''\ge 0$ in $(0,\e_*)$.

\smallskip

- If $-S>0$ and $c_2(A,S) \leq B<c_1(A,S)$,  we have seen that $\e_*=+\infty$ and that $R'$ has at least one zero. Hence $\A\ne (0,+\infty)$ and Theorem \ref{no-un} applies.

\end{proof}

\begin{proposition}[model $Q_b$]\label{prop:Qb}
Assume \eqref{strange} and let
\begin{equation}
\label{modelQ1b-disc}
Q(s)=Q_{b}(s) = \begin{cases}
\frac{A|B| s}{|B|s^m+As^{n}}-S & \mbox{for $s>0$},\quad B< 0, \quad  1<m<n, \ m<3,\\
0 & \mbox{for $s\leq 0$}.
\end{cases}
\end{equation}
\begin{itemize}
\item[$(i)$] {\it Uniqueness.} $u_M$ is unique for any $M\in (0,+\infty)$.
\item[$(ii)$] {\it Macroscopic behavior.} For $M\gg 1$, $u_M$ is droplet-shaped if $-S>0$, pancake-shaped if $-S<0$, and transition-shaped if $S=0$ (cf. Theorems \ref{thm:finale} and \ref{thm:crit}).
\end{itemize}
\end{proposition}

\begin{proof}
Simple computations show that $Q'<0$ in $(0,+\infty)$ and that $Q''\ge 0$ in $(0,+\infty)$ if and only if \eqref{strange} holds. Hence $(i)$ follows.  As to $(ii)$, we have
$$
G(s)=-s^2R'(s)= \frac{A|B|s(|B|ms^m+Ans^n)}{(|B|s^m+As^n)^2}-S,
$$
which has no zeroes if $-S\geq 0$ (hence $\e_*=+\infty$) and at least one zero if $-S<0$. In the latter case, since $-S<0$, $Q$ is not always positive, hence, by Remark \ref{rem:R1}, $\e_*<+\infty$.
\end{proof}

\begin{proposition}[model $Q_{a,g}$]\label{prop:Qag}
With $Q_a$ as in \eqref{modelQ1a-disc}, let
$$
Q(s)=Q_{a,g}(s)= Q_a(s) +\tfrac12 Ds^2 \quad\mbox{for $s>0$,} \quad D>0.
$$
\begin{itemize}
\item[$(i)$] {\it Uniqueness.} $u_M$ is unique for any $M\in (0,+\infty)$ if $-S\le 0$ or
$$\textstyle
B\leq \frac{m+1}{n(n-1)}\left(\frac{Am(m-1)}{n+1}\right)^{\frac{n+1}{m+1}}\left(\frac{D}{m-n}\right)^{\frac{m-n}{m+1}}=:c_3(A,D).
$$
\item[$(ii)$] {\it Macroscopic behavior.} For $M\gg 1$, $u_M$ is pancake-shaped (Theorem \ref{thm:finale}).
\end{itemize}
\end{proposition}

In the case $-S>0$ and $B>c_3(A,D)$, we do not know how to ascertain the sign of $Q''$ in $(0,\e_*)$.

\begin{proof} Since $R(0)=R(+\infty)=+\infty$, we always have $\e_*<+\infty$ (cf. \eqref{def-s*}), hence $(ii)$ is trivial. As for uniqueness, we discuss the sign of $Q''(s)=Am(m-1) s^{-m-1} -Bn(n-1) s^{-n-1} +D$ in $(0,\e_*)$. If $B\le 0$, then $Q''>0$ in $(0,+\infty)$. Hence we only consider $B>0$. It follows from the study of $Q'''$ that $s_3= \left(\frac{A(m+1)m(m-1)}{B(n+1)n(n-1)}\right)^{\frac{1}{m-n}}$ is the unique point of minimum of $Q''$. After straightforward computations we obtain that $Q''(s_3)\geq 0$ if and only if $B\le c_3(A,D)$. In this case $Q''\ge 0$ in $(0,+\infty)$. It remains to consider $B>c_3(A,D)$. If $-S\leq 0$, we take advantage of the results seen for $Q_a$. Indeed, in this case $R_a'$ has a unique zero, located at $\e_a$, with $R'_a\geq 0$ in $[\e_a,+\infty)$ and $Q''_a\geq 0$ in $(0,\e_a)$. Thus $R'=R'_{a}+\tfrac{D}{2}>0$ in $[\e_a,+\infty)$. Since by definition $R'(\e_*)=0$, this implies that $\e_*<\e_a$. Since $Q''=Q''_a+D\ge D >0$ in $(0,\e_a)$, we conclude that $Q''>0$ in $(0,\e_*)$.
\end{proof}

\begin{proposition}[model $Q_{b,g}$]\label{prop:Qbg}
Assume \eqref{strange}. With $Q_b$ as in \eqref{modelQ1b-disc}, let
$$
Q(s)=Q_{b,g}(s)= Q_b(s) +\tfrac12 Ds^2 \quad\mbox{for $s>0$,} \quad D>0.
$$
\begin{itemize}
\item[$(i)$] {\it Uniqueness.} $u_M$ is unique for any $M\in (0,+\infty)$.
\item[$(ii)$] {\it Macroscopic behavior.} For $M\gg 1$, $u_M$ is pancake-shaped (Theorem \ref{thm:finale}).
\end{itemize}
\end{proposition}

\begin{proof}
Since $D>0$, we have $R(0)=R(+\infty)=+\infty$, thus $\e_*<+\infty$. Under assumption \eqref{strange}, $Q''_b\geq 0$ in $(0,+\infty)$, thus $Q''=Q''_b + D>0$ in $(0,+\infty)$.
\end{proof}

\section*{Appendix}
\setcounter{definition}{0}
\setcounter{equation}{0}
\renewcommand{\theequation}{A.\arabic{equation}}
\renewcommand{\thedefinition}{A.\arabic{definition}}
\renewcommand{\thetheorem}{A.\arabic{theorem}}
\renewcommand{\theremark}{A.\arabic{remark}}
\renewcommand{\thelemma}{A.\arabic{lemma}}

\begin{lemma}
\label{aux2}
Let $u$ be a non-negative measurable function. Then the set $\mathcal{C}=\{\delta>0 : |\{u=\delta\}|>0\}$ is countable.
\end{lemma}

\begin{proof}
Let $\mu: [0,+\infty)\to [0,+\infty)$ be the distribution function of $u$, that is, $\mu(\delta)=|\{u>\delta\}|$. By definition, $\delta$ is a discontinuity point of $\mu$ if and only if $\delta\in\mathcal{C}$. Since $\mu$ is non-increasing (\cite[Remark 1.1.1]{Kes}) and since a monotone function has at most a countable set of discontinuity points (\cite[Theorem 4.30]{Rud}), $\mathcal{C}$ is at most countable.
\end{proof}

\begin{lemma}
\label{aux}
Let $u$ be a non-negative measurable function on $\RN$. Then $\chi_{\{u>\delta\}}\to \chi_{\{u>0\}}$ almost everywhere in $\RN$ as $\delta\to 0^+$.
\end{lemma}

\begin{proof}
First we note that
$$
\chi_{\{u>0\}}-\chi_{\{u>\delta\}}=\chi_{\{0<u\leq\delta\}}.
$$
Therefore the result is equivalent to proving that $\chi_{\{0<u\leq\delta\}}\to 0$ a.e. as $\delta\to 0$. Since $\{0<u\leq\delta\}$ monotonically shrinks as $\delta$ decreases, we have
\begin{equation*}
\lim_{\delta\to 0}\left|\{0<u\leq \delta\}\right| =|\omega|, \quad \text{with } \omega=\bigcap_{\delta>0}\{0<u\leq \delta\}.
\end{equation*}
In order to show that $|\omega|=0$, we note that, since $0<u\le \delta$ a.e. in $\omega$ for any $\delta>0$,
$$
\int_\omega u \le \delta |\omega| \quad \forall \delta>0 \qquad\Rightarrow\qquad \int_\omega u=0.
$$
Since $u>0$ a.e. in $\omega$, we conclude that $|\omega|=0$.
\end{proof}

We conclude with an ODE lemma.
\begin{lemma}
\label{lemtec}
Let $D:=(0,+\infty)$ and let $f,f_k:D\to \R$ be continuous functions such that $f_k\to f$ in $C_{loc}(D)$. Let $u_k\in C^2((-r,r))\cap C^0([-r,r])$ be a solution of
\begin{equation}
\label{pbapp}
\begin{cases}
\displaystyle -u_k''= f_k(u_k) & \mbox{in }\{u_k>0\}=(-r_k,r_k) \\[1ex]
u_k(0)=u_{0k},\ u'_k(0)=0,\ u_k(\pm r_k)=0,
\end{cases}
\end{equation}
non-increasing in $(0,r_k)$, and let $u\in C^2((-r,r))\cap C^0([-r,r])$ be the unique solution of
\begin{equation*}
\left\{
\begin{array}{ll}
\displaystyle -u''= f(u) & \mbox{in }\{u>0\}=(-r,r) \\[1ex]
u(0)=u_0,\ u'(0)=0, \mbox{ $u(\pm r)=0$ if $r<+\infty$.}
\end{array}\right.
\end{equation*}
Let $\displaystyle F_k(s)=\int_{s}^{u_{0k}} f_k(t) dt$ be such that
\begin{equation}
\label{F-lim}
\lim_{(k,s)\to (+\infty,0^+)} F_k(s)= +\infty.
\end{equation}
If $u_{0k}\to u_0\in D$ as $k\to +\infty$, then $r_k \to r$ and $u_k \to u$ in $C^2_{loc}((-r,r))$.
\end{lemma}

\begin{proof}
Thanks to Theorem 8.39 and Theorem 8.40 of \cite{KePe}, $u_k\to u$ in $C^2_{loc}((-r,r))$. It follows that for all $x\in (-r,r)$, $u_k(x)$ is positive for $k$ sufficiently large; hence $\displaystyle r \leq \liminf_{k\to +\infty} r_k$. It remains to prove that $\displaystyle r\geq \limsup_{k\to +\infty} r_k=:R$. If $r=+\infty$, nothing is to be proved. If $r<+\infty$, assume by contradiction that $\displaystyle r< R$. Multiplying  the equation in \eqref{pbapp} by $-u_k'$, integrating in $(0,x)$ and using the initial conditions, we have
\begin{equation}
\label{int}
(u'_k(x))^2=2 F_k(u_k(x)).
\end{equation}
Since $u$ is continuous at $x=r$ and $u(r)=0$, for every $\varepsilon>0$ there exists $\delta>0$ such that $u(r-\delta)<\varepsilon/2$. By locally uniform convergence and recalling that $u_k$ is non-increasing, $u_k<\varepsilon$ in $(r-\delta, R-\delta)$ for all $k$ sufficiently large.
Fix $x\in (r-\delta, R-\delta)$. Since $u_k$ is non-increasing, it follows from \eqref{int} that $u'_k(x)=-\sqrt{2} \sqrt{F_k(u_k(x))}.$
Hence, \eqref{F-lim} implies that for every $M>0$ we can choose $\varepsilon$ sufficiently small and $k$ sufficiently large such that
\begin{equation}
\label{edoapp}
u'_k(y)<-M \qquad \forall y\in(r-\delta,x),
\end{equation}
whence we deduce that for every $k$ sufficiently large
$$
u_k(x)=\int_{r-\delta}^x u'_k(y) dy + u_k(r-\delta) \stackrel{\eqref{edoapp}}< -M(x-r+\delta) + \varepsilon.
$$
Choosing $M$ sufficiently large and recalling that $u_k$ is non-negative, we obtain a contradiction.
\end{proof}

\subsection*{Acknowledgements} The first author gratefully thanks the Department of Mathematics ``Ulisse Dini'' of the University of Florence for having supported this work with a Research Grant.

\subsection*{Data availability statement} All data generated or analysed during this study are included in this published article.

\bibliography{biblio}{}
\bibliographystyle{abbrv}

\end{document}